\documentclass[10pt]{amsart}

\usepackage{mathrsfs}
\usepackage{amsmath}
\usepackage{mathtools}
\usepackage{amssymb}
\usepackage{amscd}
\usepackage{appendix}
\usepackage{amsthm}
\usepackage{amsfonts}
\usepackage[all]{xy}
\usepackage{tikz-cd}
\usepackage{array}
\usepackage{graphicx}
\usepackage{hyperref}
\usepackage{longtable}
\usepackage[OT2,T1]{fontenc}
\DeclareSymbolFont{cyrletters}{OT2}{wncyr}{m}{n}
\DeclareMathSymbol{\Sha}{\mathalpha}{cyrletters}{"58}

\theoremstyle{plain}
\newtheorem{thm}{Theorem}[section]
\newtheorem{lem}[thm]{Lemma}
\newtheorem{prop}[thm]{Proposition}
\newtheorem{cor}[thm]{Corollary}

\theoremstyle{definition}
\newtheorem{defn}[thm]{Definition}

\newtheorem{apn}[thm]{Assumption}

\theoremstyle{remark}
\newtheorem{rem}[thm]{Remark}

\renewcommand{\AA}{\mathbb{A}}
\newcommand{\FF}{{\mathbb F}}

\newcommand{\ZZ}{{\mathbb Z}}
\newcommand{\QQ}{{\mathbb Q}}

\newcommand{\RR}{{\mathbb R}}
\newcommand{\CC}{{\mathbb C}}

\newcommand{\NN}{{\mathbb N}}

\newcommand{\PP}{{\mathbb P}}

\newcommand{\boz}{\boldsymbol{z}}
\newcommand{\boell}{\boldsymbol{\ell}}
\newcommand{\afk}{\mathfrak{a}}

\newcommand{\cfk}{\mathfrak{c}}
\newcommand{\dfk}{\mathfrak{d}}
\newcommand{\mfk}{\mathfrak{m}}
\newcommand{\nfk}{\mathfrak{n}}
\newcommand{\pfk}{\mathfrak{p}}

\newcommand{\Hfk}{\mathfrak{H}}

\newcommand{\Ecal}{\mathcal{E}}

\newcommand{\Mcal}{\mathcal{M}}

\newcommand{\Ocal}{\mathcal{O}}

\newcommand{\Scal}{\mathcal{S}}

\newcommand{\Ical}{\mathcal{I}}
\newcommand{\Ccal}{\mathcal{C}}
\newcommand{\Dcal}{\mathcal{D}}

\newcommand{\Kcal}{\mathcal{K}}

\newcommand{\Zcal}{\mathcal{Z}}

\newcommand{\Bscr}{\mathscr{B}}
\newcommand{\Dscr}{\mathscr{D}}

\newcommand{\tr}{\operatorname{Tr}}

\newcommand{\Nr}{\operatorname{Nr}}
\newcommand{\Mat}{\operatorname{Mat}}

\newcommand{\ord}{\operatorname{ord}}

\newcommand{\GL}{\operatorname{GL}}
\newcommand{\SL}{\operatorname{SL}}

\newcommand{\Oo}{\operatorname{O}}

\newcommand{\gal}{\operatorname{Gal}}

\newcommand{\re}{\operatorname{Re}}

\newcommand{\Stab}{\operatorname{Stab}}

\numberwithin{equation}{section}

\newcommand\subfrac[2]{\genfrac{}{}{0pt}{}{#1}{#2}}

\begin{document}

\title[Class number relations, intersections, and GL(2)--tale over function fields]{On class number relations, intersections, and GL(2)--tale over the function field side}

\author[Jia-Wei Guo]{Jia-Wei Guo}
\address{Department of Mathematics, National Taiwan Univeristy, Taiwan}
\email{jiaweiguo312@gmail.com}

\author[Fu-Tsun Wei]{Fu-Tsun Wei}
\address{Department of Mathematics, National Tsing Hua Univeristy, Taiwan}
\email{ftwei@math.nthu.edu.tw}



\maketitle

\begin{abstract}
The aim of this paper is to 
study class number relations over function fields and the intersections of Hirzebruch-Zagier type divisors on the Drinfeld-Stuhler modular surfaces.
The main bridge is a particular ``harmonic'' theta series with nebentypus.
Using the strong approximation theorem, the Fourier coefficients of this series are expressed in two ways; one comes from modified Hurwitz class numbers and another gives the intersection numbers in question. 
An elaboration of this approach enables us to interpret these class numbers as a ``mass sum'' over the CM points on the Drinfeld-Stuhler modular curves, and even realize the generating function as a metaplectic automorphic form.
\end{abstract}

\section{Introduction}

\subsection{Classical story}
Given a negative integer $d$ with $d\equiv 0 \text{ or } 1 \bmod 4$, let $h(d)$ be the proper ideal class number of the imaginary quadratic order $\Ocal_d$ with discriminant $d$. Put $w(d):=\#(\Ocal_d^\times)/2$.
The classical Kronecker-Hurwitz class number relation says that for a non-square $n \in \NN$,
\begin{eqnarray}\label{eqn: HCNR}
\sum_{\subfrac{t \in \ZZ}{t^2< 4n}}\left(\sum_{\subfrac{d \in \NN}{d^2\mid (t^2-4n)}}\frac{h\big((t^2-4n)/d^2\big)}{w\big((t^2-4n)/d^2\big)}\right)&=&\sum_{\subfrac{m \in \NN}{m\mid n}} \max(m,n/m).
\end{eqnarray}
One can derive the above identity via ``modular polynomial'', i.e.\ the defining equation of the graph of the Hecke correspondence $T_n$, $n \in \NN$, on the modular curve $X$ of full level (cf.\ \cite{G-K}).
In particular, the quantity in \eqref{eqn: HCNR} is equal to the ``finite part'' of the intersection number of the divisors $T_1$ and $T_n$ on the surface $X\times X$. Taking the ``infinite part'' (from cuspidal intersections) into account, the total intersection number of $T_1$ and $T_n$ becomes
$$T_1\cdot T_n = 2 \sigma(n),$$
where $\sigma(n) :=  \sum_{m\mid n} m$ is precisely the $n$-th Fourier coefficient of the weight-two Eisenstein series (normalized so that the first Fourier coefficient equals to $1$).
This provides a very concrete example in the following connections:
$$
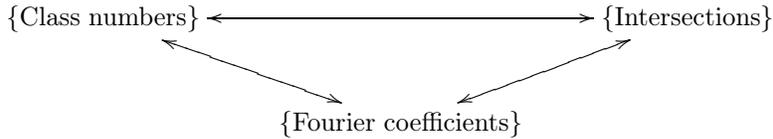
\xymatrix{
\{\text{Class numbers}\}\ar@{<->}[rr]\ar@{<->}[dr]  &  & \{\text{Intersections}\} \ar@{<->}[dl] \\
&\{\text{Fourier coefficients}\} &
}
$$

In the celebrated work of Hirzebruch and Zagier \cite{H-Z}, the whole theory on the ground of the Hilbert modular surfaces associated with real quadratic fields is well-established.
More precisely, they express the intersections of certain special divisors in terms of Hurwitz class numbers, and show that the generating function associated with these intersection numbers is actually a particular Eisenstein series with nebentypus.
The interpretations for the Fourier coefficients of Eisenstein series, which have been generalized to the ``Kudla-Millson'' theta integrals (cf.\ \cite{K-MI} and \cite{K-M}) on the quotients of symmetric spaces for orthogonal and unitary groups, are viewed as \textit{geometric Siegel-Weil formula} and have various applications (cf.\ \cite{Kud2}, \cite{Cog}, \cite{K-MII}, \cite{Kud}, and \cite{Fun} ).
Moreover, connections with the class numbers make it possible to compute explicitly the intersections in question (cf.\ \cite{H-Z}, \cite{Fun}, and \cite{G-Y}).

The purpose of this paper is to attempt an exploration of this phenomenon in the function field setting, and to derive a Hirzebruch-Zagier style geometric interpretation for the class number relations in the world of positive characteristic.

\subsection{Function Field analogue}
Let $A = \FF_q[\theta]$, the polynomial ring with one variable $\theta$ over a finite field $\FF_q$ with $q$ elements, and let $k$ be the field of fractions of $A$.
Let $k_\infty$ be the completion of $k$ with respect to the ``degree valuation'' (cf.\ Section~\ref{sec: BS}), and denote by $\CC_\infty$ the completion of a chosen algebraic closure of $k_\infty$.
The \textit{Drinfeld half plane} is $\Hfk:= \CC_\infty - k_\infty$, equipped with the M\"obius left action of $\GL_2(k_\infty)$.
Let $B$ be a quaternion algebra over $k$ which is split at $\infty$ (i.e.\ $B\otimes_k k_\infty \cong \Mat_2(k_\infty)$), and $O_B$ be an Eichler $A$-order in $B$ of type $(\nfk^+,\nfk^-)$ (cf.\ Section~\ref{sec: TSP}).
Then the embedding $B^\times \hookrightarrow \GL_2(k_\infty)$ induces an action of $\Gamma(\nfk^+,\nfk^-):= O_B^\times$ on $\Hfk$.
The quotient space $$
X(\nfk^+,\nfk^-)\ :=\
\Gamma(\nfk^+,\nfk^-)\backslash \Hfk
$$ 
is called the \textit{Drinfeld-Stuhler modular curve for $\Gamma(\nfk^+,\nfk^-)$}.
When $B = \Mat_2(k)$, the group $\Gamma(\nfk^+,\nfk^-)$ coincides (up to conjugations) with the congruence subgroup 
$$
\Gamma_0(\nfk^+):=
\left\{\begin{pmatrix} a&b \\ c&d\end{pmatrix} \in \GL_2(A)\ \bigg|\ c \equiv 0 \bmod \nfk^+\right\},
$$
and the compactification of $X(\nfk^+,\nfk^-)$ is the so-called \textit{Drinfeld modular curve for $\Gamma_0(\nfk^+)$}.
\begin{rem}
As in the classical case, the study of Drinfeld modular polynomials in \cite{Bae}, \cite{B-L}, and \cite{Hsia} give an analogue of the Kronecker-Hurwitz class number relation for ``imaginary'' quadratic $A$-orders (cf. \cite{JKYu} and \cite{W-Y}). Also, the connection with the  intersections of the Hecke correspondence on the Drinfeld modular curves is derived in \cite{JKYu} when $q$ is odd. 
Moverover, these intersection numbers appear in the Fourier expansion of the ``improper'' Eisenstein series on $\GL_2(k_\infty)$ which is introduced by Gekeler (cf.\ \cite{Gek1} and \cite{Gek2}). Thus, a parallel story for the Kronecker-Hurwitz case over rational function fields is developed. We may also expect to see these connections when the base field $k$ is an arbitrary global function field.
\end{rem}

From now on, we always assume that $q$ is \textbf{odd}.
Fix a monic square-free $\dfk \in A$ with even degree. Then the quadratic field $F:= k(\sqrt{\dfk})$ is \textit{real} over $k$, (i.e.\ the infinite place of $k$ is split in $F$). The embedding $F\hookrightarrow F\otimes_k k_\infty \cong k_\infty\times k_\infty$ induces
$$\GL_2(F) \hookrightarrow \GL_2(k_\infty) \times \GL_2(k_\infty),$$
providing an action of $\GL_2(F)$ on $\Hfk_F:= \Hfk \times \Hfk$.
Let $O_F$ be the integral closure of $A$ in $F$.
Given a monic $\nfk \in A$, put
$$
\Gamma_{0,F}(\nfk) := \left\{
\begin{pmatrix} a&b \\ c&d \end{pmatrix} \in \GL_2(O_F)
\ \bigg|\ ad-bc \in \FF_q^\times \text{ and } c \equiv 0 \bmod \nfk\right\}.
$$
The \textit{Drinfeld-Stuhler modular surface for $\Gamma_{0,F}(\nfk)$} is
$$\Scal_{0,F}(\nfk):= \Gamma_{0,F}(\nfk) \backslash \Hfk_F,$$
which is a coarse moduli scheme for the so-called \textit{Frobenius-Hecke sheaves} (with additional ``level-$\nfk$ structure'') introduced by Stuhler in \cite{Stu} (see also \cite{Ozg}).

We are interested in the intersections between the ``Hirzebruch-Zagier-type divisors'' on $\Scal_{0,F}(\nfk)$ defined as follows.
For $x = \begin{pmatrix}a&b \\ c&d\end{pmatrix} \in \Mat_2(F)$, we put
$$
\bar{x} := \begin{pmatrix}d&-b \\ -c&a\end{pmatrix}
\quad \text{ and } \quad
x' := \begin{pmatrix} a'&b' \\ c'&d'\end{pmatrix},
$$
where for every $\alpha \in F$, $\alpha'$ is the conjugate of $\alpha$ under the action of the non-trivial element in $\gal(F/k)$.
Consider the involution $*$ on $\Mat_2(F)$ defined by
$$ x^* \ :=\ \begin{pmatrix}0&1/\nfk \\ 1&0\end{pmatrix} \bar{x}'\begin{pmatrix}0&1 \\ \nfk & 0 \end{pmatrix}, \quad \forall x \in \Mat_2(F).
$$
Let
$$
\Lambda:= \{x \in \Mat_2(O_F)\mid x^* = x\}.
$$
We have a left action of $\Gamma:= \Gamma_{0,F}(\nfk)$ on $\Lambda$ by
$$ \gamma \star x := \gamma x \gamma^* \cdot (\det \gamma )^{-1}, \quad \gamma \in \Gamma_{0,F}(\nfk) \text{ and } x \in \Lambda.
$$
For each $x$ in $\Lambda$ with $\det x \neq 0$, let
$$B_x:= \{b \in \Mat_2(F) \mid x b^* = \bar{b} x\}
\quad \text{ and } \quad 
\Gamma_x := B_x^\times \cap \Gamma.
$$
From Lemma~\ref{lem: B_x}, we know that $B_x$ is an \textit{indefinite} quaternion algebra over $k$ (i.e.\ unramified at the infinite place of $k$), and so the quotient $\Ccal_x:= \Gamma_x \backslash \Hfk$ becomes the Drinfeld-Stuhler modular curve for $\Gamma_x$ (cf.\ Section~\ref{sec: CMBES}).
Put
$$S_x := \begin{pmatrix}0&1 \\ \nfk &0\end{pmatrix} \bar{x}.$$
The embedding from $\Hfk$ into $\Hfk_F$ defined by $(z \mapsto (z,S_x z))$ gives rise to a (rigid analytic) morphism $f_x: \Ccal_x \rightarrow \Scal_{0,F}(\nfk)$, and we set
$$\Zcal_x := f_{x,*}(\Ccal_x),$$
the push-forward divisor of $f_x$ on $S_{0,F}(\nfk)$.
For non-zero $a \in A$, the \textit{Hirzebruch-Zagier divisor of discriminant $a$} is:
$$\Zcal(a) := \sum_{x\in\Gamma \backslash \Lambda_a}\Zcal_x, \quad \text{ where } \Lambda_a:= \{x \in \Lambda \mid \det(x) = a\}.$$
${}$

Notice that we may identify $B_1$ with the quaternion algebra
$$\left(\frac{\dfk, \nfk}{k}\right) := k+k\mathbf{i}+k\mathbf{j}+k\mathbf{ij} \quad \text{ with } \mathbf{i}^2 = \dfk,\ \mathbf{j}^2 = \nfk, \text{ and } \mathbf{ji} = -\mathbf{ij}.
$$
In particular, suppose $\nfk$ is square-free and coprime to $\dfk$.
Write $\nfk = \nfk^+\cdot \nfk^-$ and $\dfk = \dfk^+ \cdot \dfk^-$, where for each prime factor $\pfk$ of $\nfk^{\pm}$ (resp.\ $\dfk^{\pm}$) we have the Legendre quadratic symbol $\left(\frac{\dfk}{\pfk}\right) = \pm 1$ (resp.\ $\left(\frac{\nfk}{\pfk}\right) = \pm 1$).
Then $B_1$ is ramified precisely at the prime factors of $\dfk^-\nfk^-$, 
and 
$O_{B_1}:= B_1\cap \Mat_2(O_F)$ is an Eichler $A$-order of type $(\dfk^+\nfk^+,\dfk^-\nfk^-)$ in $B_1$.
In other words, $\Ccal_1$ is actually the Drinfeld-Stuhler modular curve for $\Gamma(\dfk^+\nfk^+,\dfk^-\nfk^-)$.
We pick $\Zcal_1$ as our ``base'' divisor on $\Scal_{0,F}(\nfk)$, and determine the intersection number of $\Zcal_1$ and $\Zcal(a)$ for non-zero $a \in A$ in the following:

%To present the first theorem in this paper, we make the following assumptions:
%\begin{apn}\label{apn: main}
%(1) $\nfk = \nfk^+\cdot \nfk^-$ where $\nfk^+$ and $\nfk^-$ are square-free with $\text{gcd}(\nfk^+,\nfk^-) = 1$, and the number of prime factors of $\nfk^-$ is postive and even;\\
%(2) For each prime factor $\pfk$ of $\nfk^+$ (resp.\ $\nfk^-$), the Legendre quadratic symbol
%$\left(\frac{\dfk}{\pfk}\right) =  1$ (resp.\ $-1$).
%\end{apn}

\begin{thm}\label{thm: MT1}
Suppose $\nfk \in A_+$ is square-free with coprime to $\dfk$ and $\deg (\dfk^-\nfk^-) >0$. 
The intersection number of $\Zcal_1$ and $\Zcal(a)$ for non-zero $a \in A$ is equal to
$$
\Zcal_1\cdot \Zcal(a) = 2\cdot \sum_{\subfrac{t \in A}{t^2-4a\preceq 0}}H^{\dfk^+\nfk^+,\dfk^-\nfk^-}(\dfk(t^2-4a)).
$$
Here for $d \in A$, we write $d\preceq 0$ if $d=0$ or $k(\sqrt{d})$ is an ``imaginary'' quadratic extension of $k$ (i.e.\ the infinite place of $k$ does not split), and $H^{\dfk^+\nfk^+,\dfk^-\nfk^-}(d)$ is the modified Hurwitz class number in {\rm Definition~\ref{defn: HCN}}
and Remark~{\rm \ref{rem: H(0)}}.
\end{thm}
%Put
%$$\text{vol}(\Zcal_1) := \frac{2}{q^2-1} \cdot \prod_{\pfk \mid \dfk^+ \nfk^+}(q^{\deg \pfk} + 1) \prod_{\pfk \mid \dfk^-\nfk^-}(q^{\deg \pfk} -1) \quad  \big(= -2 H^{\dfk^+\nfk^+,\dfk^-\nfk^-}(0)\big).$$
We point out that the self-intersection number of $\Zcal_1$ is defined to be
an analogue of the ``Euler characteristic'' of $\Zcal_1$ in Definition~\ref{defn: SIN}.\\
%$$\Zcal_1 \cdot \Zcal_1 \ :=\ - \text{vol}(\Zcal_1)\ =\ 2 H^{\dfk^+\nfk^+,\dfk^-\nfk^-}(0).$$

To establish the equality in Theorem~\ref{thm: MT1}, the main bridge is the (adelic) theta integral $I(\cdot;
\varphi_\Lambda)$ associated with a particular chosen Schwartz function $\varphi_\Lambda$ in Section~\ref{sec: TSN-PSF}.
Our strategy is briefly sketched as follows.
Notice that using adelic language, we may express very naturally the $a$-th Fourier coefficient of $I(\cdot;\varphi_\Lambda)$ for a given non-zero $a \in A$ in terms of the modified Hurwitz class numbers (cf.\ Theorem~\ref{thm: TINFC}).
On the other hand, the strong approximation theorem (for the indefinite quaternion algebra ramified precisely at the prime factors of $\nfk^-$) leads to an alternative expression of the $a$-th Fourier coefficient of $I(\cdot;\varphi_\Lambda)$ (cf.\ Theorem~\ref{thm: AFCN}), which enables us to connect the Fourier coefficient with the intersection number $\Zcal_1\cdot \Zcal(a)$ (Theorem~\ref{thm: CN-Int} and Corollary~\ref{cor: TINFC}).
This completes the proof.\\

The theta integral $I(\cdot;\varphi_\Lambda)$ has nice invariant property and transformation law (cf.\ Propostion~\ref{prop: TLN}).
In particular, the choice of the ``infinite component'' of $\varphi_\Lambda$ in \eqref{eqn: phi_infty} is most crucial in bridging two sides of the equality in Theorem~\ref{thm: MT1}, and ensures as well the ``harmonicity'' of $I(\cdot;\varphi_\Lambda)$ (cf.\ Lemma~\ref{lem: har}).
This allows us to extend $I(\cdot;\varphi_\Lambda)$ to a ``Drinfeld-type'' automorphic form on $\GL_2(k_\infty)$ (an analogue of weight-$2$ modular forms over function fields, see \textit{Remark}~\ref{rem: wt-2} and \cite{G-R}) with nebentypus character $\left(\frac{\cdot}{\dfk}\right)$ for $\Gamma_0^{(1)}(\dfk\nfk):= \Gamma_0(\dfk\nfk)\cap \SL_2(A)$, cf.\ Theorem~\ref{thm: Ext-DTAF}.
In other words, we have the following theorem (cf.\ \textit{Remark}~\ref{rem: DTS-FC}):

\begin{thm}\label{thm: MT2}
Under the assumptions in \text{\rm Theorem~\ref{thm: MT1}}, there exists a Drinfeld-type automorphic form $\vartheta_\Lambda$ on $\GL_2(k_\infty)$ with nebentypus character $\left(\frac{\cdot}{\dfk}\right)$ for the congruence subgroup $\Gamma_0^{(1)}(\dfk\nfk)$ whose Fourier expansion is given by: for $(x,y) \in k_\infty \times k_\infty^\times$,
$$
\vartheta_\Lambda\begin{pmatrix} y & x \\ 0&1 \end{pmatrix} = -\text{\rm vol}(\Zcal_1)\beta_{0,2}(y) + \sum_{0 \neq a \in A}\big(\Zcal_1\cdot \Zcal(a)\big) \cdot \big(\beta_{a,2}(y) \psi_\infty(ax)\big).
$$
Here 
\begin{itemize}
\item $\beta_{a,s}(y)$ is the following truncated analogue of the Bessel function:
$$\beta_{a,s}(y):=|y|_\infty^{s/2}\cdot
\begin{cases} 1, & \text{ if $\deg a + 2 \leq \ord_\infty(y)$;}\\
0, & \text{ otherwise;}
\end{cases}$$
\item $|\cdot|_\infty$ is the absolute value on $k_\infty$ normalized so that $|\theta| = q$, \item $\psi_\infty: k_\infty \rightarrow \CC^\times$ is a fixed additive character on $k_\infty$ defined in {\rm Section~\ref{sec: TMk}},
\item 
$$\text{\rm vol}(\Zcal_1)
:= 2 H^{\dfk^+\nfk^+,\dfk^-\nfk^-}(0).
$$
\end{itemize}
\end{thm}

\begin{rem}
%{\color{blue}Our approach is completely different from the one in \cite{H-Z}, where Hirzebruch and Zagier consider the intesection of the full-level modular curve with other divisors on the full-level Hilbert modular surface.}
(1) Our approach for the geometric interpretation of the Fourier coefficients of $\vartheta_\Lambda$ as the corresponding intersection numbers is basically in the framework of \cite{K-M} and \cite{Fun}, while the rigid analytic geometry is utilized and the test function constructed by Kudla-Millson at the archimedean place is replaced by a proper Schwartz function in \eqref{eqn: phi_infty} for this positive-characteristic setting.

(2) The technical assumption ``$\deg(\dfk^-\nfk^-)>0$'' in Theorem~\ref{thm: MT1} implies that the Drinfeld-Stuhler modular curve $\Ccal_1$ has no ``cusps''.
Therefore there are no contributions of the ``cuspidal intersections'' to $\Zcal_1 \cdot \Zcal(a)$ in Theorem~\ref{thm: MT1} and Theorem~\ref{thm: MT2}.
When $\dfk^-\nfk^- = 1$,
this argument would need to be adjusted by ``regularizing the theta integral $I(\cdot;\varphi_\Lambda)$'' as in \cite{Fun}, and taking the cuspidal intersections into account for $\Zcal_1\cdot \Zcal(a)$ in a suitable ``compactification of the surface $\Scal_{0,F}(\nfk)$''.
However, due to a lack of studies in the literature for these two technical issues in the function field context, we make this assumption in Theorem~\ref{thm: MT1} first.
The general case will be explored in future work.

(3) Given non-zero $d \in A$, a point $\boz$ in $\Ccal_1 = X(\dfk^+ \nfk^+,\dfk^-\nfk^-)$ is \textit{CM with discriminant $d$} if it has a representative $z \in \Hfk$ and there exists $x_z \in O_{B_1}$ with $x_z^2 = d$ and $x_z \cdot z = z$
(cf.\ Definition~\ref{defn: CMP}).
As $z \notin k_\infty$, the existence of such a CM point $\boz$ ensures that $d \prec 0$. Moreover, for non-zero $a \in A$, the intersection of $\Zcal_1$ and $\Zcal(a)$, excluding possible self-intersection of $\Zcal_1$, is actually supported by the image of the CM points on $\Ccal_1$ with discriminant $\dfk(t^2-4a)$ for  $t \in A$ and $t^2-4a \prec 0$ (cf.\ \textit{Remark}~\ref{rem: CM-sup}).
\end{rem}

We also study the ``mass'' over CM points on the Drinfeld-Stuhler modular curves via ``metaplectic'' theta series.
Let $\nfk^+, \nfk^- \in A_+$ be two square-free polynomials with $\text{gcd}(\nfk^+,\nfk^-) = 1$, and the number of prime factors of $\nfk^-$ is positive and even.
Given a point $\boz$ on $X(\nfk^+,\nfk^-)$ represented by $z \in \Hfk$, let $\Stab_{\Gamma(\nfk^+,\nfk^-)}(z)$ be the stablizer of $z$ in $\Gamma(\nfk^+,\nfk^-)$, and put
$$w(\boz):= \frac{q-1}{
\#\big(\text{Stab}_{\Gamma(\nfk^+,\nfk^-)}(z)\big)}.$$
For non-zero $d \in A$,
the \textit{mass} over the CM points with discriminant $d$ on $X(\nfk^+,\nfk^-)$ is:
$$
\Mcal(d) := \sum_{\subfrac{\text{CM } \boz \in X(\nfk^+,\nfk^-)}{\text{with discriminant $d$}}}\frac{1}{w(\boz)}.
$$
Comparing with the classical story, it is natural to expect that this ``mass'' amounts to generating function of a metaplectic form. 
We derive that
(cf.\ Theorem~\ref{thm: CNo-geo} and Corollary~\ref{cor: TSo-FC}):

\begin{thm}\label{thm: Intro.MS-HCN}
Let $\nfk^+, \nfk^- \in A_+$ be two square-free polynomials with $\text{gcd}(\nfk^+,\nfk^-) = 1$, and the number of prime factors of $\nfk^-$ is positive and even.
Given non-zero $d \in A$, the following equality holds:
$$\Mcal(d) = \begin{cases}
H^{\nfk^+,\nfk^-}(d), & \text{ if $d \prec 0$;}\\
0, & \text{ otherwise.}
\end{cases}$$
Moreover, if we put 
$\Mcal(0):= H^{\nfk^+,\nfk^-}(0)$,
then there exists a ``weight-$\frac{3}{2}$'' metaplectic form $\vartheta^o$ on $\widetilde{\GL}_2(k_\infty)$ (introduced in {\rm Section~\ref{sec: Ext-PTI}}) with the following Fourier expansion: for $(x,y) \in k_\infty\times k_\infty^\times$,
$$
\vartheta^o\left(\begin{pmatrix} y&x\\ 0&1\end{pmatrix},1\right)
=
\sum_{d \in A}
\Mcal(d) \cdot \big(\beta_{d,\frac{3}{2}}(y)\psi_\infty(-dx)\big).
$$
%Here $\beta(\cdot)$ is an analogue of the Bessel function on $k_\infty^\times$ introduced in {\rm Theorem~\ref{thm: FC1}}.
\end{thm}

\begin{rem}
${}$
\begin{itemize}
\item[(1)] When the number of prime factors of $\nfk^-$ is odd, we are still able to extend the generating function
$$
\sum_{d \in A, \ d \preceq 0}H^{\nfk^+,\nfk^-}(d)\cdot \big(\beta_{d,\frac{3}{2}}(y)\psi_\infty(-dx)\big), \quad \forall (x,y) \in k_\infty\times k_\infty^\times
$$
to a weight-$\frac{3}{2}$ metaplectic form (coming from the theta series associated with ``definite'' pure quaternions ramified precisely at the prime factors of $\nfk^-$ and $\infty$ (cf.\ \cite[Proposition 2.2]{Wei} and \cite[Corollary 5.4]{CLWY}).
\item[(2)] When $\nfk^- = 1$, we would expect that, after further work, the above generating function becomes the ``mock'' part of a particular metaplectic automorphic form as in the classical case (cf.\ \cite{Coh} and \cite{Zag})
by ``regularizing'' the theta integral $I(\cdot; \varphi^o)$ in Section~\ref{sec: PTI1}.
This will be studied in a subsequent paper.
\end{itemize}
\end{rem}

\subsection{Content}
The contents of this paper go as follows:
\begin{itemize}
\item (\textit{Preliminaries.}) We set up basic notations used throughout this paper in Section~\ref{sec: BS}.
The modified Hurwitz class number and the needed properties are reviewed in Section~\ref{sec: HCN}.
The Tamagawa measures on the groups appearing in this paper are given in Section~\ref{sec: TMk}, \ref{sec: HCN}, and \ref{sec: TMB}, respectively.
The definition of the Weil representation and theta series are recalled in Section~\ref{sec: WRTS}.
\item (\textit{Metaplectic theta series and mass over CM points.}) In Section~\ref{sec: TSP}, we focus on the theta integral associated with a particular Schwartz function $\varphi^o$ on the quadratic space of pure quaternions in an indefinite quaternion algebra over $k$.
The Fourier coefficients of the theta integral associated with $\varphi^o$ are expressed via two different ways in Theorem~\ref{thm: FC1} and Corollary~\ref{cor: AFC}, respectively.
In Section~\ref{sec: CMBES}, we study the mass over the CM points with a given discriminant on a Drinfeld-Stuhler modular curve, and establish
the equality between $\Mcal(d)$ and $H^{\nfk^+,\nfk^-}(d)$ in Theorem~\ref{thm: CNo-geo}.
%via two expressions for the Fourier coefficients in the last section.
\item (\textit{Class number relations and intersections.}) In Section~\ref{sec: TSN},
we take a particular Schwatz function $\varphi_\Lambda$ associated with the $A$-lattice $\Lambda$, and express the Fourier coefficients of the theta integral $I(\cdot;\varphi_\Lambda)$ explicitly in terms of the modified Hurwitz class numbers in Theorem~\ref{thm: TINFC}.
In Section~\ref{sec: Ext-TI}, we show the harmonicity of $I(\cdot;\varphi_\Lambda)$ and extend it to a Drinfeld-type automorphic form $\vartheta_\Lambda$ on $\GL_2(k_\infty)$ in Theorem~\ref{thm: Ext-DTAF}.
On the other hand, the intersection between $\Zcal_1$ and Hirzebruch-Zagier-type divisors are described in Section~\ref{sec: Intersection}.
Via the alternative expression of the Fourier coefficients of $I(\cdot;\varphi_\Lambda)$ in Theorem~\ref{thm: AFCN}, we prove Theorem~\ref{thm: MT1} and Theorem~\ref{thm: MT2} in the end.
\item (\textit{Appendix: local optimal embeddings.}) The needed results in Eichler's theory of local optimal embeddings are recalled in Appendix~\ref{sec: App-LOE}, and we express the technical local integrals used in Theorem~\ref{thm: FC1} by the number of local optimal embeddings in Appendix~\ref{sec: App-SLI}.
\end{itemize}

\subsection*{Acknowledgements}
This work is initiated during the conference ``2019 Postech-PMI and NCTS Joint workshop on number theory'', at the Pohang University of Science and Technology, Korea.
The authors are very grateful to Professor Jeehoon Park and Professor Chia-Fu Yu for organizing such a wonderful conference. The authors would also like to thank Jing Yu for many helpful suggestions and comments.
The first author is supported by the Ministry of Science and Technology 
(grant no.\ 109-2115-M-002-017-MY2).
The second author is supported by the Ministry of Science and Technology 
(grant no.\ 107-2628-M-007-004-MY4 and 109-2115-M-007-017-MY5)
and the National Center for Theoretical Sciences.

\section{Preliminaries}

\subsection{Basic settings}\label{sec: BS}
Let $\FF_q$ be a finite field with $q$ elements.
Throughout this paper, we always assume $q$ to be \textbf{odd}.
Let $A:=\FF_q[\theta]$, the polynomial ring with one variable $\theta$ over $\FF_q$, and $k:= \FF_q(\theta)$, the field of fractions of $A$.
Let $\infty$ be the infinite place of $k$, i.e.\ the place corresponding to the \lq\lq degree\rq\rq\ valuation $\ord_\infty$ defined by
$$\ord_\infty\left(\frac{a}{b}\right) := \deg b- \deg a, \quad \forall a,b \in A \text{ with $b \neq 0$.}
$$
The associated absolute value on $k$ is normalized by $|\alpha|_\infty := q^{-\ord_\infty(\alpha)}$ for every $\alpha \in k$.
Let $k_\infty$ be the completion of $k$ with respect to $|\cdot|_\infty$, which can be identified with the Laurent series field $\FF_q(\!(\theta^{-1})\!)$.
Put $\varpi := \theta^{-1}$, a fixed uniformizer at $\infty$, and $O_\infty:= \FF_q[\![\varpi]\!]$,
the valuation ring in $k_\infty$.

Let $A_+$ be the set of monic polynomials in $A$.
By abuse of notations, we identify $A_+$ with the set of non-zero ideals of $A$.
In particular, for $\afk \in A_+$ we put
$$\Vert \afk \Vert \ :=\ \#(A/\afk) \quad (= |\afk|_\infty).$$
Given a non-zero prime ideal $\pfk$ of $A$, 
the normalized absolute value associated with $\pfk$ is:
$$|\alpha|_\pfk := \Vert\pfk\Vert^{-\ord_\pfk(\alpha)}, \quad \forall \alpha \in k.
$$
Here $\ord_\pfk(\alpha)$ is the order of $\alpha$ at $\pfk$ for every $\alpha \in k$.
The completion of $k$ with respect to $|\cdot|_\pfk$ is denoted by $k_\pfk$, and put $O_\pfk$ the valuation ring in $k_\pfk$. We also refer the non-zero prime ideals of $A$ to the finite places of $k$. \\

Let $k_\AA:= \prod_v' k_v$, the adele ring of $k$.
The maximal compact subring of $k_\AA$ is denoted by $O_\AA$.
The adelic norm $|\cdot|_\AA$ on the idele group $k_\AA^\times$ is:
$$
|(\alpha_v)_v|_\AA := \prod_v |\alpha_v|_v, \quad \forall (\alpha_v)_v \in k_\AA^\times.
$$

\subsubsection{Additive character and Tamagawa measure}\label{sec: TMk}

Let $p$ be the characteristic of $k$ and $\psi_\infty: k_\infty \rightarrow \CC^\times$ be the additive character defined by:
$$
\psi_\infty\left(\sum_i a_i \varpi^i\right) := \exp\left(\frac{2\pi \sqrt{-1}}{p} \cdot \text{Trace}_{\, \FF_q/\FF_p}(-a_1)\right), \quad \forall \sum_i a_i \varpi^{i} \in k_\infty.
$$
The conductor of $\psi_\infty$ is $\varpi^2 O_\infty$ and $\psi_\infty(A) = 1$.
Since
$$ 
k_\AA = k + \left(k_\infty \times \prod_{\pfk } O_\pfk\right)
\quad \text{ and } \quad 
k \cap \left(\prod_{\pfk} O_\pfk\right) = A,
$$
we may extend $\psi_\infty$ uniquely to an additive character $\psi: k_\AA \rightarrow \CC^\times$
so that $\psi(\alpha) = 1$ for all $\alpha \in k + \big((\varpi^2 O_\infty) \times \prod_{\pfk} O_\pfk\big)$ and $\psi\big|_{k_\infty} = \psi_\infty$.
Put $\psi_\pfk := \psi\big|_{k_\pfk}$ for each finite place $\pfk$ of $k$, which is a non-trivial additive character on $k_\pfk$ with trivial conductor.\\

For each place $v$ of $k$, let $dx_v$ be the ``self-dual'' Haar measure on $k_v$ with respect to $\psi_v$, i.e.\
$$\text{vol}(O_\pfk,dx_\pfk) = 1 \text{ for each finite place $\pfk$ of $k$,} \quad \text{and}
\quad
\text{vol}(O_\infty,dx_\infty) = q.
$$
Define the Haar measure $d^\times x_v$ on $k_v^\times$ by
$$d^\times x_\pfk:= \frac{\Vert\pfk\Vert}{\Vert \pfk \Vert-1} \cdot \frac{dx_\pfk}{|x_\pfk|_\pfk}\quad
\text{ and }\quad
d^\times x_\infty:= \frac{q}{q-1} \cdot \frac{dx_\infty}{|x_\infty|_\infty}.
$$
The \textit{Tamagawa measure} on $k_\AA^\times$ (with respect to $\psi$) is $d^\times x = \prod_v d^\times x_v$.

\subsection{Imaginary quadratic fields and class numbers}\label{sec: HCN}

A quadratic field extension $K/k$ is called \textit{imaginary} if the infinite place of $k$ does not split in $K$.
Let $K_\AA := K\otimes_k k_\AA$ and $\text{T}_{K/k}:= K_\AA \rightarrow k_\AA$ be the trace map.
Then the Tamagawa measure on $K_\AA^\times$ (with respect to the additive character $\psi \circ \text{T}_{K/k}$)
and the one on $k_\AA^\times$ induce a Haar measure $d^\times \alpha$ on the quotient group $K_\AA^\times/k_\AA^\times$.
More precisely,
let $O_K$ (resp.\ $O_{K_\infty}$) be the integral closure of $A$ (resp.\ $O_\infty$) in $K$ (resp.\ $K_\infty:= K\otimes_k k_\infty$).
For each non-zero prime ideal $\pfk$ of $A$, put $K_\pfk := K \otimes_k k_\pfk$ and  $O_{K_\pfk}:= O_K \otimes_A O_\pfk$.
We normalize the Haar measure $d^\times \alpha_v$ on $K_v^\times/k_v^\times$ for each place of $v$ by
$$\text{vol}(O_{K_\pfk}^\times/O_\pfk^\times)
= \Vert\pfk\Vert^{-1+1/e_\pfk(K/k)} \quad
\text{ and } \quad 
\text{vol}(O_{K_\infty}^\times/O_\infty^\times)
= q^{1/e_\infty(K/k)}.
$$
Here $e_v(K/k)$ is the ramification index of the place $v$ of $k$ in $K$.
Then $d^\times \alpha = \prod_v d^\times \alpha_v$.

\begin{prop}\label{prop: VOL-CN}
Let $K$ be an imaginary quadratic field over $k$, and $O_K$ be the integral closure of $A$ in $K$. 
Let $\Delta(O_K/A)$ be the discriminant ideal of $O_K$ over $A$, $h(O_K)$ be the class number of $O_K$, and put 
$w(O_K):= \#(O_K^\times)/(q-1)$.
We have
\begin{eqnarray}
\text{\rm vol}(K^\times  \backslash K_\AA^\times/k_\AA^\times)
&=& \frac{h(O_K)}{w(O_K)} 
\cdot e_\infty(K/k) \cdot \prod_v \text{\rm vol}(O_{K_v}^\times/O_v^\times) \nonumber \\
&=& \frac{h(O_K)}{w(O_K)} \cdot e_\infty(K/k) \cdot q^{1/e_\infty(K/k)}
\cdot \Vert \Delta(O_K/A)\Vert^{-1/2}. \nonumber
\end{eqnarray}
\end{prop}

\begin{proof}
As $A$ is a principal ideal domain, one gets
$$k_\AA^\times = k^\times \cdot (k_\infty^\times \times \prod_{\pfk} O_\pfk^\times).
$$
Thus the exact sequence
$$
1 \longrightarrow \frac{O_K^\times}{\FF_q^\times} \longrightarrow \frac{K_\infty^\times \times \prod_\pfk O_{K_\pfk}^\times}{k_\infty^\times \times \prod_\pfk O_\pfk^\times} \rightarrow 
\frac{K^\times \cdot (K_\infty^\times \times \prod_\pfk O_{K_\pfk}^\times)}{K^\times \cdot (k_\infty^\times \times \prod_\pfk O_\pfk^\times)} \longrightarrow 1
$$
implies
\begin{eqnarray}
\text{vol}(K^\times  \backslash K_\AA^\times/k_\AA^\times)
&=&\text{vol}\Big(K^\times  \backslash K_\AA^\times/(k_\infty^\times \times \prod_{\pfk } O_\pfk^\times)\Big) \nonumber \\
&=& \frac{\#\big(K^\times \backslash K_\AA^\times/(K_\infty^\times \times \prod_\pfk O_{K_\pfk}^\times)\big)}{\#(O_K^\times/\FF_q^\times)} \cdot \text{vol}\left(\frac{K_\infty^\times \times \prod_\pfk O_{K_\pfk}^\times}{k_\infty^\times \times \prod_{\pfk} O_\pfk^\times}\right).
\nonumber
\end{eqnarray}
The result follows from
$$
\frac{\#\big(K^\times \backslash K_\AA^\times/(K_\infty^\times \times \prod_\pfk O_{K_\pfk}^\times)\big)}{\#(O_K^\times/\FF_q^\times)}
= \frac{h(O_K)}{w(O_K)}
$$
and
\begin{eqnarray}
\text{vol}\left(\frac{K_\infty^\times \times \prod_\pfk O_{K_\pfk}^\times}{k_\infty^\times \times \prod_{\pfk} O_\pfk^\times}\right)
&=&
\text{vol}(K_\infty^\times/k_\infty^\times) \cdot \prod_{\pfk} \text{vol}(O_{K_\pfk}^\times/O_\pfk^\times) \nonumber \\
&=&
e_\infty(K/k)\cdot q^{1/e_\infty(K/k)}
\cdot \Vert \Delta(O_K/A)\Vert^{-1/2}.
\nonumber
\end{eqnarray}
\end{proof}

\begin{rem}
Let $\varsigma_K : k^\times\backslash k_\AA^\times \rightarrow \{\pm 1\}$ be the quadratic Hecke character associated with $K/k$, and let $L(s,\varsigma_K)$ be the $L$-function of $\varsigma_K$. 
It is known that (cf.\ \cite[Section 2.2]{CWY}, see also \cite[Theorem 5.9]{Ros})
$$
L(1,\varsigma_K)
= \frac{\#(O_K^\times)}{\#(\FF_q^\times)} \cdot q \cdot  \Big(q^{(1-e_\infty(K/k))/2} \cdot \Vert \Delta(O_K/A)\Vert^{-1/2}\Big) \cdot \frac{h(O_K)}{2/e_\infty(K/k)}.
$$
The above proposition says in particular that
$$
\text{vol}(K^\times  \backslash K_\AA^\times/k_\AA^\times) = 2 \cdot L(1,\varsigma_K).
$$
\end{rem}

Recall the following fact (cf.\ \cite[Chapter I (12.12) Theorem]{Neu}):

\begin{lem}\label{lem: CN}
For each $A$-order $\Ocal$ in an imaginary quadratic extension $K$ of $k$, let $h(\Ocal)$ be the proper ideal class number of $\Ocal$ and $w(\Ocal) := \#(\Ocal^\times)/(q-1)$. Then
$$
\frac{h(\Ocal)}{w(\Ocal)} = \frac{h(O_K)}{w(O_K)} \cdot
\prod_{\pfk} \#\left(\frac{O_{K_\pfk}^\times}{\Ocal_\pfk^\times}\right).
$$
Here $\Ocal_\pfk := \Ocal \otimes_A A_\pfk$ for every non-zero prime ideal $\pfk$ of $A$.
\end{lem}

We write $d \prec 0$ for $d \in A$ if the quadratic extension $k(\sqrt{d})$ is imaginary over $k$.
Given $d \in A$ with $d \prec 0$,
let $\Ocal_d := A[\sqrt{d}]$, $h(d):=h(\Ocal_d)$, and $w(d) := w(\Ocal_d)$.

\begin{defn}\label{defn: HCN}
For square-free $\nfk^+,\nfk^- \in A_+$ with $\text{gcd}(\nfk^+,\nfk^-) = 1$, we recall the following \textit{modified Hurwitz class number}
$$
H^{\nfk^+,\nfk^-}(d) := \sum_{\subfrac{\cfk \in A_+}{\cfk^2 \mid d}} \frac{h(d/\cfk^2)}{w(d/\cfk^2)}\cdot \prod_{\pfk \mid \nfk^+}\left(1+\left\{\frac{d/\cfk^2}{\pfk}\right\}\right)\prod_{\pfk \mid \nfk^-}\left(1-\left\{\frac{d/\cfk^2}{\pfk}\right\}\right).
$$
Here
$$\left\{\frac{d}{\pfk}\right\} :=
\begin{cases}
1, & \text{ if either $\pfk$ split in $k(\sqrt{d})$ or $\pfk^2 \mid d$;} \\
-1, & \text{ if $\pfk$ is inert in $k(\sqrt{d})$ and $\ord_\pfk(d) = 0$;} \\
0, & \text{ if $\ord_\pfk(d) = 1$.}
\end{cases}
$$
\end{defn}

Write $$d = d_0 \cdot \prod_{\pfk} \pfk^{2c_{\pfk}},$$
where $d_0 \in A$ is square-free
(and $c_\pfk = 0$ for almost all irreducible $\pfk \in A_+$).
For each irreducible $\pfk \in A_+$ and integer $\ell_\pfk$ with $0\leq \ell_\pfk \leq c_\pfk$, put
$$
e^{\nfk^+,\nfk^-}_\pfk(\ell_\pfk) :=
\begin{cases}\displaystyle 1 \pm \left\{\frac{d_0 \pfk^{2\ell_\pfk}}{\pfk}\right\},
& \text{ if $\pfk \mid \nfk^{\pm}$;}\\
1, & \text{ otherwise.}
\end{cases}
$$

We obtain the following expression for the modified Hurwitz class numbers in later use:

\begin{prop}\label{prop: HCN}
Given $d \in A$ with $d \prec 0$, write $d = d_0 \prod_\pfk \pfk^{2c_\pfk}$.
Then
$$
H^{\nfk^+,\nfk^-}(d)
= \frac{h(d_0)}{w(d_0)} \cdot  \prod_\pfk \left[
\sum_{0\leq \ell_\pfk \leq c_\pfk}\#\left(\frac{\Ocal_{d_0,\pfk}^\times}{\Ocal_{d_0\pfk^{2\ell_\pfk},\pfk}^\times}\right) \cdot e^{\nfk^+,\nfk^-}_\pfk(\ell_\pfk)
\right].
$$
%Here $\Ocal_{d,\pfk} = \Ocal_d \otimes_A O_\pfk$.
\end{prop}

\begin{proof}
For $\boell = (\ell_\pfk)_\pfk \in \prod_\pfk \ZZ$ with $0\leq \ell_\pfk\leq c_\pfk$,
put $d_0(\boell):= d_0 \prod_\pfk \pfk^{2\ell_\pfk}$.
Then
$$
\left\{\frac{d_0(\boell)}{\pfk}\right\} = \left\{\frac{d_0\pfk^{2\ell_\pfk}}{\pfk}\right\}
\quad \text{ and } \quad
\Ocal_{d_0(\boell),\pfk} = \Ocal_{d_0\pfk^{2\ell_\pfk}, \pfk}.
$$
Therefore
\begin{eqnarray}
H^{\nfk^+,\nfk^-}(d)
&=& \sum_{\subfrac{\boell \in \prod_\pfk \ZZ}{0\leq \ell_\pfk \leq c_\pfk}}
\frac{h(d(\boell))}{w(d(\boell))}\cdot
\prod_{\pfk \mid \nfk^+}\left(1+\left\{\frac{d(\boell)}{\pfk}\right\}\right) \cdot
\prod_{\pfk \mid \nfk^-}\left(1-\left\{\frac{d(\boell)}{\pfk}\right\}\right) \nonumber \\
&=&
\frac{h(d_0)}{w(d_0)} \cdot \sum_{\subfrac{\boell \in \prod_\pfk \ZZ}{0\leq \ell_\pfk \leq c_\pfk}}
\left[\prod_\pfk \#\left(\frac{O_{d_0,\pfk}^\times}{\Ocal_{d_0\pfk^{2\ell_\pfk},\pfk}^\times}\right) \cdot \prod_\pfk e_\pfk^{\nfk^+,\nfk^-}(\ell_\pfk)\right] \nonumber \\
&=& 
\frac{h(d_0)}{w(d_0)} \cdot 
\prod_\pfk \left[
\sum_{0\leq \ell_\pfk \leq c_\pfk}\#\left(\frac{\Ocal_{d_0,\pfk}^\times}{\Ocal_{d_0\pfk^{2\ell_\pfk},\pfk}^\times}\right) \cdot e^{\nfk^+,\nfk^-}_\pfk(\ell_\pfk)
\right]. \nonumber
\end{eqnarray}
\end{proof}

\begin{rem}\label{rem: H(0)}
For convention, we put
$$
H^{\nfk^+,\nfk^-}(0)
:=
-\frac{1}
{q^2-1} \cdot \prod_{\pfk\mid \nfk^+}(\Vert \pfk \Vert + 1) \prod_{\pfk \mid \nfk^-}(\Vert \pfk \Vert -1).
$$
This is related to a volume quantity with respect to the ``Tamagawa measure'' on quaternion algebras discussed below.
\end{rem}

\subsection{Tamagawa measure on quaternion algebras}\label{sec: TMB}

Let $B$ be an \textit{indefinite} quaternion algebra over $k$ (i.e.\ $B_\infty := B\otimes_k k_\infty$ is not division).
Put $B_\AA := B\otimes_k k_\AA$.
Let $\tr: B_\AA \rightarrow k_\AA$  be the reduced trace map.
Choose a Haar measure $db = \prod_v db_v$ on $B_\AA$ which is self-dual with respect to the additive character $\psi \circ \tr$.
More precisely, for each non-zero prime ideal $\pfk$ of $A$,
let $R_\pfk$ be a maximal $O_\pfk$-order in $B_\pfk:= B\otimes_k k_\pfk$.
Then
$$
\text{vol}(R_\pfk, db_\pfk)
= \begin{cases} 1/ \Vert \pfk \Vert, & \text{ if $B$ is ramified at $\pfk$;}\nonumber \\
1, & \text{ otherwise.}
\end{cases}
$$
Let $O_{B_\infty}$ be a maximal $O_\infty$-order in $B_\infty$.
Then
$\text{vol}(O_{B_\infty},db_\infty) = q^4$.

Let $\Nr : B_\AA^\times \rightarrow k_\AA^\times$ be the reduced norm map.
For each non-zero prime ideal $\pfk$ of $A$,
we take the Haar measure  $d^\times b_\pfk$ on $B_\pfk^\times$ defined by
$$
d^\times b_\pfk
:= \frac{\Vert \pfk \Vert}{\Vert \pfk \Vert -1}\cdot \frac{db_\pfk}{|\Nr(b_\pfk)|_\pfk}.
$$
In particular,
$$
\text{vol}(R_\pfk^\times, d^\times b_\pfk) =
\left(1-\frac{1}{\Vert \pfk \Vert^{2}}\right) \cdot 
\begin{cases}
1/(\Vert \pfk \Vert -1), & \text{ if $B$ is ramified at $\pfk$;} \\
1, & \text{ otherwise.}
\end{cases}
$$
Similarly, put
$$d^\times b_\infty := \frac{q}{q-1} \cdot \frac{db_\infty}{|\Nr(b_\infty)|_\infty}.
$$
Then
$$
\text{vol}(O_{B_\infty}^\times, d^\times b_\infty) 
= q^4-q^2.
$$
The \textit{Tamagawa measure} $d^\times b$ on $B_\AA^\times$ is the Haar measure satisfying that for every compact open subgroup $\Kcal = \prod_v \Kcal_v$ of $B_\AA^\times$, one has
$$
\text{vol}(\Kcal, d^\times b)
= \prod_v \text{vol}(\Kcal_v, d^\times b_v).
$$

Let $\nfk^- \in A_+$ be the product of the primes at which $B$ is ramified, and 
$\nfk^+ \in A_+$ be a square-free polynomial coprime to $\nfk^-$.
Let $O_B$ be an Eichler $A$-order of type $(\nfk^+,\nfk^-)$ in $B$,
and put $O_{B_\pfk} := O_B \otimes_A O_\pfk$ for non-zero prime $\pfk$ of $A$.
Let $O_{B_\AA} := \prod_v O_{B_v}$.
Then:

\begin{lem}\label{lem: vol-hcn}
The Tamagawa measures on $B_\AA^\times$ and $k_\AA^\times$ induces a Haar measure on $B_\AA^\times/k_\AA^\times$ so that 
$$
\text{\rm vol}(O_{B_\AA}^\times/O_\AA^\times) = \frac{(q-1)(q^2-1)}{\prod_{\pfk \mid \nfk^+}(\Vert \pfk \Vert + 1) \prod_{\pfk \mid \nfk^-}(\Vert \pfk \Vert -1)}
= -\frac{q-1}{H^{\nfk^+,\nfk^-}(0)}.
$$
\end{lem}

\begin{proof}
For each non-zero prime ideal $\pfk$ of $A$, let $R_\pfk$ be a maximal $O_\pfk$-order containing $O_{B_\pfk}$.
As $\nfk^+$ is square-free, we have
$$
\#(R_\pfk^\times/O_{B_\pfk}^\times) = 
\begin{cases}
\Vert \pfk \Vert + 1, & \text{ if $\pfk \mid \nfk^+$;}\\
1, & \text{ otherwise.}
\end{cases}
$$
Thus
$$
\text{vol}(O_{B_\pfk}^\times)
= \left(1-\frac{1}{\Vert \pfk \Vert^2}\right)\cdot
\begin{cases}
1/(\Vert \pfk \Vert -1), & \text{ if $\pfk \mid \nfk^-$;} \\
1/(\Vert \pfk \Vert + 1), & \text{ if $\pfk \mid \nfk^+$;}\\
1, & \text{ otherwise.}
\end{cases}
$$
Notice that
$$
\prod_\pfk\left(1-\frac{1}{\Vert \pfk \Vert^s}\right)^{-1}
= \frac{1}{1-q^{1-s}}, \quad \re(s)>1.
$$
Therefore we obtain
$$\prod_\pfk\left(1-\frac{1}{\Vert \pfk \Vert^2}\right) = 1-q^{1-2} = \frac{q-1}{q},
$$
and
\begin{eqnarray}
\text{vol}(O_{B_\AA}^\times/O_\AA^\times)
&=& \frac{\text{vol}(O_{B_\infty}^\times)}{\text{vol}(O_\infty^\times)} \cdot \prod_\pfk\frac{\text{vol}(O_{B_\pfk}^\times)}{\text{vol}(O_\pfk^\times)} \nonumber \\ 
&=& \frac{q^4-q^2}{q} \cdot \prod_\pfk\left(1-\frac{1}{\Vert \pfk \Vert^2}\right) \prod_{\pfk \mid \nfk^+}\left(\frac{1}{\Vert \pfk \Vert + 1}\right) \prod_{\pfk \mid \nfk^-}\left(\frac{1}{\Vert \pfk \Vert -1}\right) \nonumber \\ 
&=& 
\frac{(q-1)(q^2-1)}{\prod_{\pfk \mid \nfk^+}(\Vert \pfk \Vert + 1) \prod_{\pfk \mid \nfk^-}(\Vert \pfk \Vert -1)} \nonumber \\
&=& - \frac{q-1}{H^{\nfk^+,\nfk^-}(0)}. \nonumber 
\end{eqnarray}
The last equality follows directly from the definition of $H^{\nfk^+,\nfk^-}(0)$.
\end{proof}

\begin{rem}
The Haar measure on $B_\AA^\times/k_\AA^\times$ induced by the Tamagawa measures on $B_\AA^\times$ and on $k_\AA^\times$ satisfies (cf.\ \cite[Theorem 3.3.1]{Weil2})
$$
\text{vol}(B^\times  \backslash B_\AA^\times/k_\AA^\times)
= 2.
$$
\end{rem}

\subsection{Weil representation and theta series}\label{sec: WRTS}

Given a place $v$ of $k$,
let $\widetilde{\SL}_2(k_v)$ be the metaplectic cover of $\SL_2(k_v)$, i.e.\ the non-trivial central externsion of $\SL_2(k_v)$ by $\{\pm 1\}$ associated with the following ``variant'' Kubota $2$-cocycle:
$$\sigma_v^1(g,g'):= \left(\frac{X(gg')}{X(g)}, \frac{X(gg')}{X(g')}\right)_v \cdot \frac{s_v^1(g)s_v^1(g')}{s_v^1(gg')},$$
where
\begin{itemize}
\item $(\cdot,\cdot)_v$ is the (quadratic) Hilbert symbol;
\item 
$$
X\begin{pmatrix} a&b \\c & d \end{pmatrix}:= 
\begin{cases} c & \text{ if $c \neq 0$;}\\
d & \text{ if $c = 0$;}
\end{cases}
$$
\item
$$
s_v^1\begin{pmatrix}a&b \\ c&d\end{pmatrix} :=
\begin{cases}
(c,d)_v, & \text{ if $\ord_v(c)$ is odd and $d \neq 0$;}\\
1, & \text{ otherwise.}
\end{cases}
$$
\end{itemize}
Identifying $\widetilde{\SL}_2(k_v)$ with $\SL_2(k_v) \times \{\pm 1\}$ as sets, the group law of $\widetilde{\SL}_2(k_v)$ is given by
$$(g_1,\xi_1)\cdot (g_2,\xi_2) := \big(g_1g_2, \xi_1\xi_2 \sigma_v^1(g_1,g_2)\big),\quad
\forall (g_1,\xi_1), (g_2,\xi_2) \in \widetilde{\SL}_2(k_v).
$$
It is known that (cf.\ \cite[Section 2.3]{Gel}) the inclusion map $\SL_2(O_v)\hookrightarrow \widetilde{\SL}_2(k_v)$ defined by
$(\kappa_v \mapsto \tilde{\kappa}_v := (\kappa_v, 1))$
is actually a group homomorphism.
Set
$$\sigma^1(g,g'):= \prod_v(g_v,g'_v), \quad \forall g = (g_v)_v, g' = (g_v')_v \in \SL_2(k_\AA),$$
which is a well-defined $2$-cocycle of $\SL_2(k_\AA)$.
This induces a non-trivial central extension $\widetilde{\SL}_2(k_\AA)$ of $\SL_2(k_\AA)$ by $\{\pm 1\}$, called the metaplectic cover of $\SL_2(k_\AA)$.
Moreover, the inclusions $\SL_2(k)\hookrightarrow \widetilde{\SL}_2(k_\AA)$ and $\SL_2(O_\AA)\hookrightarrow \widetilde{\SL}_2(k_\AA)$ defined by
$$\gamma \longmapsto \tilde{\gamma}:=\big(\gamma, \prod_vs_v(\gamma)\big), \quad \forall \gamma \in \SL_2(k)
\quad \text{ and } \quad \kappa \longmapsto \tilde{\kappa}:= (\kappa,1), \quad \forall \kappa \in \SL_2(O_\AA)
$$
are group homomorphisms.\\

Let $(V,Q_V)$ be a non-degenerated quadratic space over $k$ with $\dim_k(V) = n$.
For each place $v$ of $k$, let $V(k_v) := V\otimes_k k_v$ and $S(V(k_v))$ be the space of Schwartz function on $V(k_v)$.
The {\it Weil representation} $\omega_v^V$ of $\widetilde{\SL}_2(k_v) \times \Oo(V)(k_v)$ on $S(V(k_v))$, where $\Oo(V)$ is the orthogonal group of $V$, is given by (cf.\ \cite[Theorem 2.22]{Gel}):

\begin{eqnarray}
(1) && \omega^{V}_v(h) \phi(x)= \phi(h^{-1} x), \text{ $h \in \Oo(V)(k_v)$}, \phi\in S(V(k_v)); \nonumber \\
(2) && \omega^{V}_v(1,\xi) \phi(x)= \xi^n \cdot \phi(x), \text{ $\xi \in \{\pm 1\}$}; \nonumber \\
(3) && \omega^{V}_v \left(\begin{pmatrix}1&u\\0&1\end{pmatrix},1\right)\phi(x) = \psi_v(u Q_{V}(x))\cdot \phi(x), \text{ $u \in k_v$};\nonumber \\
(4) && \omega^{V}_v \left(\begin{pmatrix}a_v&0\\0&a_v^{-1}\end{pmatrix},1\right)\phi(x)=|a_v|_v^{\frac{n}{2}} \cdot (a_v,a_v)^n_v\cdot  \frac{\varepsilon^{V}_v(a_v)}{\varepsilon^{V}_v(1)} \cdot \phi(a_v x), \text{ $a_v \in k_v^{\times}$;} \nonumber \\
(5) && \omega^{V}_v\left(\begin{pmatrix}0&1\\-1&0\end{pmatrix},1\right)\phi(x)= \varepsilon^{V}_v(1) \cdot \widehat{\phi}(x). \nonumber
\end{eqnarray}
Here:
\begin{itemize}
%\item
%$(\cdot,\cdot):k_v^\times \times k_v^\times \rightarrow \{\pm 1\}$ is the (quadratic) Hilbert symbol;
\item 
$\varepsilon^V_v(a_v)$ is the following \textit{Weil index}:
$$\varepsilon^{V}_v(a_v):= \int_{L_{v}} \psi_v(a_v Q_{V}(x))d_{a_v}x, \quad \forall a_v \in k_v^{\times},$$
where $L_{v}$ is a sufficiently large $O_v$-lattice in $V(k_v)$, and the Haar measure $d_{a_v}x$ is self-dual with respect to the pairing 
$$(x,y)\mapsto \psi_v(a_v \cdot \langle x, y \rangle_V),\quad \forall x,y \in V(k_v),$$
with $\langle \cdot,\cdot \rangle_V$ the bilinear form on $V$ associated with $Q_V$;
\item $\widehat{\phi}(x)$ is the Fourier transform of $\phi$ (with respect to the self-dual Haar measure):
$$\widehat{\phi}(x) :=\int_{V(k_v)} \phi(y)\psi_v(\langle x, y\rangle_V)dy.$$
\end{itemize}

The (global) Weil representation of $\widetilde{\SL}_2(k_\AA)\times \Oo(V)(k_\AA)$ on the Schwartz space $S(V(k_\AA))$, where $V(k_\AA) := V\otimes_k k_\AA$, is $\omega^V:= \otimes_v \omega^V_v$.

\begin{rem}\label{rem: sections}
${}$
\begin{itemize}
\item[(1)] 
Set $V_0 = k$ and $Q_0(x):= x^2$ for $x \in V_0$.
For convention we put
$$\varepsilon_v(a_v) := \varepsilon_v^{V_0}(a_v), \quad \forall a_v \in k_v^\times.$$
It is known that for $a_v,a_v' \in k_v^\times$ one has
$$\frac{\varepsilon_v(a_v)\varepsilon_v(a_v')}{\varepsilon_v(a_va_v')\varepsilon_v(1)} = (a_v,a_v')_v.$$
\item[(2)] Let $(V,Q_V)$ be a non-degenerate quadratic space over $k$ with $\dim_k(V) = n$. 
Choose an orthogonal basis $\{x_1,...,x_n\}$ of $V$ and let $\alpha_i = Q_V(x_i)$, $i=1,...,n$.
Then for every place $v$ of $k$ and $a_v \in k_v^\times$, one has
$$\varepsilon_v^V(a_v) = \prod_{i=1}^n \varepsilon_v(a_v \alpha_i).$$
\item[(3)] For $y = (y_v)_v \in k_\AA^\times$, let
$$
\varepsilon^V(y) := \prod_v \varepsilon_v^V(y_v).
$$
The Weil reciprocity says (cf.\ \cite[Proposition 5]{Weil3}):
$$ \varepsilon^V(\alpha) = 1, \quad \forall \alpha \in k^\times.$$

%(2) (cf.\ \cite[Section 2.3]{Gel}) The inclusion map $\SL_2(O_\AA) \hookrightarrow \widetilde{\SL}_2(k_\AA)$ defined by 
%$$
%(\kappa \mapsto \tilde{\kappa}:=(\kappa,1))
%$$ 
%is actually a group homomorphism.
%Moreover,
%for each place $v$ of $k$, we define the map $s_v : \SL_2(k_v) \rightarrow \{\pm 1\}$ by
%$$
%s_v\begin{pmatrix}a&b \\ c&d\end{pmatrix} :=
%\begin{cases}
%(c,d)_v, & \text{ if $\ord_v(c)$ is odd and $d \neq 0$;}\\
%1, & \text{ otherwise.}
%\end{cases}
%$$
%Then the inclusion $\SL_2(k) \hookrightarrow \widetilde{\SL}_2(k_\AA)$ defined by
%$$
%\gamma \mapsto \tilde{\gamma}:= \big(\gamma, \prod_v s_v(\gamma)\big)
%$$
%is also a group homomorphism.\\

\item[(3)] When $\dim_k(V)$ is even, the local and global Weil representations actually factor through $\SL_2(k_v)$ and $\SL_2(k_\AA)$, respectively.
\end{itemize}
\end{rem}

Given $\varphi \in S(V(k_\AA))$, the \textit{theta series associated with $\varphi$} is:
$$
\Theta(\tilde{g},h;\varphi):= \sum_{x \in V}\big(\omega^V(\tilde{g},h)\varphi\big)(x), \quad \forall (\tilde{g},h) \in \widetilde{\SL}_2(k_\AA)\times \Oo(V)(k_\AA).
$$
For every $\gamma \in \SL_2(k)$, $\tilde{g} \in \widetilde{\SL}_2(k_\AA)$, and $\varphi \in S(V(k_\AA))$ we have
$$
\Theta(\tilde{\gamma}\tilde{g};\varphi) = \Theta(\tilde{g};\varphi).
$$

Given $a \in k$ and $y \in k_\AA^\times$,
let:
$$
\Theta^*(a,y;h;\varphi) := \int_{k\backslash k_\AA} \Theta\left(\Bigg(\begin{pmatrix}y&uy^{-1}\\0&y^{-1}\end{pmatrix},1\Bigg) ,h;\varphi\right)\psi(-au)du,
$$
where the Haar measure $du$ is normalized so that $\text{vol}(k\backslash k_\AA,du) = 1$.
For $u \in k_\AA$, one has the following Fourier expansion
$$
\Theta\left(\Bigg(\begin{pmatrix}y&uy^{-1}\\0&y^{-1}\end{pmatrix},1\Bigg),h;\varphi\right)
= \sum_{a \in k}\Theta^*(a,y;h;\varphi) \psi(au).
$$
We shall focus on particular quadratic spaces with degree $3$ and $4$, and study the Fourier coefficients of the theta integrals associated with special Schwartz functions.

\section{The case of pure quaternions}\label{sec: TSP}

Let $B$ be an indefinite division quaternion algebra over $k$, and $B^o$ be the subspace of pure quaternions in $B$, i.e.\
$$B^o:= \{b \in B \mid \tr(b) = 0\}.$$
Let $$(V^o, Q^o):= (B^o, \Nr),$$
where $\Nr:B \rightarrow k$ is the reduced norm map.
Then $V^o$ is an anisotropic quadratic space of degree $3$ over $k$.
Note that we have the following exact sequence: 
$$
1\longrightarrow k^\times \longrightarrow B^\times \longrightarrow \text{SO}(V^o)\longrightarrow 1,
$$
where the map from $B^\times$ to $\text{SO}(V^o)$ is defined by
$$b \longmapsto h_b:= (x \mapsto bxb^{-1}), \quad \forall b \in B^\times.$$
Thus the chosen Haar measure on $B_\AA^\times/k_\AA^\times$ in Section~\ref{sec: TMB} gives us a Haar measure $dh$ on $\text{SO}(V^o)(k_\AA)$ satisfying 
$$\text{vol}(\text{SO}(V^o)(k)\backslash \text{SO}(V^o)(k_\AA),dh) = \text{vol}(B^\times  \backslash B_\AA^\times/k_\AA^\times, d^\times b) = 2.$$

For $\varphi \in S(V^o(k_\AA))$, we are interested in the following theta integral:
$$
I(\tilde{g};\varphi):= \int_{\text{SO}(V^o)(k)\backslash \text{SO}(V^o)(\AA)} \Theta(\tilde{g},h;\varphi) dh,
\quad \forall \tilde{g} \in \widetilde{\SL}_2(k_\AA),
$$

Given $a \in k$, the $a$-th Fourier coefficients of $I(\cdot;\varphi)$ is:
$$
I^*(a,y;\varphi):=
\int_{k\backslash k_\AA}I\left(\bigg(\begin{pmatrix} y & uy^{-1} \\ 0 & y^{-1} \end{pmatrix},1\bigg);\varphi\right) \psi(-a u) du.
$$

\begin{lem}\label{lem: FCo}
For $a \in k$ and $y \in k_\AA^\times$,
we have $I^*(a,y;\varphi) = 0$ unless there exists $x \in B^o$ such that $x^2 = -a$.
In this case,
\begin{eqnarray}
I^*(a,y;\varphi) %&:= &
%\int_{B^\times k_\AA^\times \backslash B_\AA^\times } \Theta^*(a,y;h_b;\varphi) d^\times b\nonumber \\
&=&
|y|_\AA^{3/2} \cdot (y,y)_\AA \cdot \varepsilon^{V^o}(y) \cdot \text{\rm vol}(K_{x}^\times  \backslash K_{x,\AA}^\times/ k_\AA^\times) \cdot  \int_{K_{x,\AA}^\times \backslash B_\AA^\times} \varphi(y b^{-1} x b) d^\times b. \nonumber
\end{eqnarray}
Here $(y,y)_\AA := \prod_v(y_v,y_v)_v$ for $y = (y_v)_v \in k_\AA^\times$, $K_{x}$ is the centralizer of $x$ in $B$, and $K_{x,\AA} := K_x \otimes_k k_\AA$.  
\end{lem}

\begin{proof}
By definition, we get
\begin{eqnarray}
&& I^*(a,y;\varphi) \nonumber \\
&=& \int_{k\backslash k_\AA} 
\left[\int_{B^\times \backslash B_\AA^\times/ k_\AA^\times} \Theta\left(\bigg(\begin{pmatrix} y & uy^{-1} \\ 0 & y^{-1} \end{pmatrix},1\bigg);\varphi\right) d^\times b\right] \psi(-au) du \nonumber \\
&=& \int_{B^\times  \backslash B_\AA^\times/ k_\AA^\times}
\left[ 
\int_{k\backslash k_\AA} \left( \sum_{x \in V^o} \bigg(\omega^{V^o}\bigg(\begin{pmatrix} y & uy^{-1} \\ 0 & y^{-1}\end{pmatrix},1\bigg) \varphi\bigg)(b^{-1}xb)\right) \psi(-au) du\right] d^\times b. \nonumber 
\end{eqnarray}
For $x \in V^o$ and $b \in B_\AA^\times$, it is straightforward that
$$
\Bigg(\omega^{V^o}\bigg(\begin{pmatrix} y & uy^{-1} \\ 0 & y^{-1}\end{pmatrix},1\bigg) \varphi\Bigg)(b^{-1}xb)
= \psi(uQ^o(x)) \cdot |y|_\AA^{3/2} \cdot (y,y)_\AA \cdot \varepsilon^{V^o}(y) \cdot \varphi(y \cdot b^{-1}xb).
$$
Since
$$
\int_{k\backslash k_\AA} \psi\big(uQ^o(x)\big) \cdot \psi(-au) du
= \begin{cases}
1, & \text{ if $Q^o(x) = a$;}\\
0, & \text{ otherwise,}
\end{cases}
$$
we have
\begin{eqnarray}
&& \int_{k\backslash k_\AA} 
\left( \sum_{x \in V^o} \bigg(\omega^{V^o}\bigg(\begin{pmatrix} y & uy^{-1} \\ 0 & y^{-1}\end{pmatrix},1\bigg) \varphi\bigg)(b^{-1}xb)\right) \psi(-au) du \nonumber \\
&=& |y|_\AA^{3/2} \cdot (y,y)_\AA \cdot \varepsilon^{V^o}(y) \cdot \sum_{x \in V_a^o} \varphi(y \cdot b^{-1}xb), \nonumber
\end{eqnarray}
where $V_a^o := \{x \in V^o \mid Q^o(x) = -x^2 = a\}$.
Therefore,
$$
I^*(a,y;\varphi) = 
|y|_\AA^{3/2} \cdot (y,y)_\AA \cdot \varepsilon^{V^o}(y) \cdot \sum_{x \in V_a^o} \int_{B^\times  \backslash B_\AA^\times/ k_\AA^\times} \varphi(y\cdot b^{-1}xb) d^\times b.
$$
In particular, $V_a^o$ is empty unless $a = 0$ or $-a \notin (k^\times)^2$ and $k(\sqrt{a})$ embeds into $B$.
This shows the first assertion of this lemma.

Suppose $V_a^o$ is non-empty, say $x \in V_a^o$.
Note that the conjugation action of $B^\times$ on $V_a^o$ is transitive, and the stablizer of $x$ in $B^\times$ is exactly $K_x^\times$.
Hence
\begin{eqnarray}
I^*(a,y;\varphi) &=&
|y|_\AA^{3/2} \cdot (y,y)_\AA \cdot \varepsilon^{V^o}(y) \cdot
\int_{K_x^\times  \backslash B_\AA^\times/ k_\AA^\times}\varphi(y\cdot b^{-1}xb) d^\times b \nonumber \\
&=& |y|_\AA^{3/2} \cdot (y,y)_\AA \cdot \varepsilon^{V^o}(y) \cdot
\text{vol}(K_x^\times  \backslash K_{x,\AA}^\times/ k_\AA^\times)
\cdot \int_{K_{x,\AA}^\times \backslash B_\AA^\times} \varphi(y\cdot b^{-1}xb) d^\times b, \nonumber
\end{eqnarray}
and the proof is complete.
\end{proof}

Note that for non-zero $x \in B^o$, we have $K_x = k(x)$, the quadratic field generated by $x$ in $B$.
In particular, $I^*(a,y;\varphi) = 0$ unless $a = 0$ or $-a \notin (k^\times)^2$
and the quadratic field $k(\sqrt{-a})$ embeds into $B$.

\begin{rem}
Suppose $\varphi$ is a pure-tensor, i.e.\ $\varphi = \otimes_v \varphi_v$, where $\varphi_v \in S(V^o(k_v))$.
Then for $x \in V^o$, the following equality holds:
$$
\int_{K_{x,\AA}^\times \backslash B_\AA^\times} \varphi(yb^{-1}xb) d^\times b
= \prod_v \int_{K_{x,v}^\times \backslash B_v^\times} \varphi_v(y_vb_v^{-1}x b_v) d^\times b_v.
$$
We shall choose a particular pure-tensor Schwartz function $\varphi^o = \otimes_v \varphi^o_v \in S(V^o(k_\AA))$ so that the associated Fourier coefficients can be expressed in terms of modified Hurwitz class numbers.
\end{rem}

\subsection{Particular theta integral}\label{sec: PTI1}

Let $\nfk^- \in A_+$ be the product of the primes at which $B$ is ramified, and take a square free $\nfk^+ \in A_+$ coprime to $\nfk^-$.
Let $O_B$ be an Eichler $A$-order of type $(\nfk^+,\nfk^-)$ in $B$, and take $\Lambda^o:= O_B \cap B^o$.
For each non-zero prime $\pfk$ of $A$, take
$\varphi^o_v := \mathbf{1}_{\Lambda^o_\pfk}$, the characteristic function of $\Lambda^o_\pfk := \Lambda^o\otimes_A O_\pfk$.
To choose $\varphi^o_\infty$, we first identify $B_\infty:= B\otimes_k k_\infty$ with $\Mat_2(k_\infty)$, and take $O_{B_\infty}:= \Mat_2(O_\infty)$.
Let $L_\infty := \pi_\infty O_{B_\infty}$ and
$$L_\infty':= \left\{\begin{pmatrix} a&b \\ c&d \end{pmatrix} \in L_\infty \ \bigg|\ c \in \pi_\infty^2 O_\infty\right\}.$$
Put $L_\infty^o := L_\infty \cap B^o$ and $L_\infty^{\prime,o}:= L_\infty' \cap B^o$.
We take
$$\varphi_\infty^o := \mathbf{1}_{L_\infty^o} - \frac{q+1}{2} \cdot \mathbf{1}_{L_\infty^{\prime,o}}.$$

\begin{prop}\label{prop: TLo}
Let $\nfk := \nfk^+\nfk^-$ and 
$$
\Kcal_0^{1}(\nfk\infty) := \left\{ \begin{pmatrix} a&b \\c&d \end{pmatrix} \in \SL_2(O_\AA)\ \bigg|\ c \equiv 0  \bmod \nfk\infty\right\}.
$$
Then for $\gamma \in \SL_2(k)$, $\tilde{g} \in \widetilde{\SL_2}(k_\AA)$, and $\kappa \in \Kcal_0^{1}(\nfk \infty)$, we have
$$I(\tilde{\gamma} \tilde{g} \tilde{\kappa}; \varphi^o) = I(\tilde{g};\varphi^o).$$
Here $\tilde{\gamma}$ for $\gamma \in \SL_2(k)$ and $\tilde{\kappa}$ for $\kappa \in \SL_2(O_\AA)$ are introduced in {\rm \textit{Remark}~\ref{rem: sections}}.
\end{prop}

\begin{proof}
It suffices to show that 
$\omega^{V^o}(\kappa) \varphi^o = \varphi^o$
for every $\kappa = (\kappa_v)_v \in \Kcal_0^1(\nfk\infty)$, i.e.\
\begin{eqnarray}\label{eqn: Inv}
\omega_v^{V^o}(\kappa_v) \varphi^o_v = \varphi^o_v \quad \text{ for every place $v$ of $k$}.
\end{eqnarray}
Given a place $v$ of $k$, notice that  $\SL_2(O_v)$ is generated by $\begin{pmatrix} 0 & 1 \\ -1&0\end{pmatrix}$ and $\begin{pmatrix}1& u \\ 0&1\end{pmatrix}$, $u \in O_v$, and for $\begin{pmatrix} a&b \\ c&d \end{pmatrix} \in \SL_2(O_v)$ with $c \equiv 0 \bmod v$, one has $d \in O_v^\times$ and
$$
\begin{pmatrix}
a&b \\ c&d
\end{pmatrix}
=
\begin{pmatrix}
1 & bd^{-1} \\ 0 & 1
\end{pmatrix}
\begin{pmatrix}
d^{-1} & 0 \\ 0 & d
\end{pmatrix}
\begin{pmatrix}
0&1 \\ -1&0
\end{pmatrix}
\begin{pmatrix}
1 & -d^{-1} c \\ 0&1 
\end{pmatrix}
\begin{pmatrix}
0&-1 \\ 1&0
\end{pmatrix}
.
$$
The equality~\eqref{eqn: Inv} then follows from straightforward calculations.
\end{proof}

The above invariant property of $I(\cdot;\varphi^o)$ implies in particular that for every $\alpha \in k^\times$, $\varepsilon \in O_\AA^\times$, and $u \in O_\AA$, the Fourier coefficient $I^*(a,y;\varphi^o)$ satisfies
$$
I^*(\alpha^{-2}a,\alpha y \varepsilon;\varphi^o) = (y, \alpha \varepsilon)_\AA \cdot I^*(a,y;\varphi^o)
\quad \text{ and } \quad
I^*(a,y;\varphi^o) = I^*(a,y;\varphi^o)\cdot \psi(a y^2 u).
$$
Since 
$$k_\AA^\times = k^\times \cdot (k_\infty^\times \times \prod_{\pfk\neq \infty} O_\pfk^\times),$$
it suffices to consider $I^*(a,y;\varphi^o)$ for $y \in k_\infty^\times$, and in this case we get
$$I^*(a,y;\varphi^o) = 0 \quad \text{ unless
$a \in A$ with $\deg a + 2 \leq 2 \ord_\infty(y)$.}
$$
Moreover, we may express $I^*(a,y;\varphi^o)$ via modified Hurwitz class numbers in the following:

\begin{thm}\label{thm: FC1}
Given $a \in A$ and $y \in k_\infty^\times$ with $\deg a + 2 \leq 2\ord_\infty(y)$, we have
$$
 I^*(a,y;\varphi^o) = 
\text{\rm vol}(O_{B_\AA}^\times/O_\AA^\times) \cdot \tilde{\beta}(y) \cdot \begin{cases}
 H^{\nfk^+,\nfk^-}(-a), & \text{ if $-a\preceq 0$;}\\
0, & \text{ otherwise.}
\end{cases}
$$
Here
$$
\tilde{\beta}(y) := |y|_\infty^{3/2}\cdot %(y,y)_\infty \cdot
\frac{\varepsilon_\infty(y)}{\varepsilon_\infty(1)}, \quad \forall y \in k_\infty^\times.
$$
\end{thm}

\begin{proof}
Notice that for each place $v$ of $k$ and $y_v \in k_v^\times$, one has
$$\varepsilon_v(y_v)\cdot \varepsilon_v^{V^o}(y_v)
= \begin{cases}
-1 & \text{ if $B$ is ramified at $v$;} \\
 1 & \text{ otherwise.}
\end{cases}
$$
Thus the product formula of the Weil index in \textit{Remark}~\ref{rem: sections} (3) shows that for $y \in k_\infty^\times$, 
$$|y|_\AA^{3/2}\cdot (y,y)_\AA \cdot \varepsilon^{V^o}(y) = |y|_\infty^{3/2} \cdot (y,y)_\infty \cdot \frac{\varepsilon_\infty(1)}{\varepsilon_\infty(y)} = \tilde{\beta}(y).
$$

If $a = 0$, then $K_x = B$, and Lemma~\ref{lem: FCo} implies
\begin{eqnarray}
I^*(0,y;\varphi^o)
&=& \tilde{\beta}(y) \cdot \text{vol}(B^\times k_\AA^\times\backslash B_\AA^\times) \cdot \varphi^o(0) \nonumber \\
&=& \tilde{\beta}(y) \cdot 2 \cdot \frac{1-q}{2} \nonumber \\
&=& \text{\rm vol}(O_{B_\AA}^\times/O_\AA^\times) \cdot \tilde{\beta}(y) \cdot H^{\nfk^+,\nfk^-}(0). \nonumber
\end{eqnarray}
The last equality follows from Lemma~\ref{lem: vol-hcn}.\\

Suppose $a \neq 0$.
From Lemma~\ref{lem: FCo} we get
\begin{eqnarray}
I^*(a,y;\varphi^o)
&=& \tilde{\beta}(y) \cdot \text{vol}(K_x^\times  \backslash K_{x,\AA}^\times/ k_\AA^\times) \cdot
\int_{K_{x,\infty}^\times \backslash B_\infty^\times} \varphi_\infty^o(y b_\infty^{-1}xb_\infty)d^\times b_\infty
\nonumber \\
&& \quad \quad \quad \quad \quad \quad \quad \quad \quad \quad \cdot 
\prod_\pfk \int_{K_{x,\pfk}^\times \backslash B_\pfk^\times} \varphi_\pfk^o(b_\pfk^{-1}xb_\pfk)d^\times b_\pfk. \nonumber
\end{eqnarray}
Write $a = a_0 \prod_\pfk \pfk^{2c_\pfk}$, where $a_0$ is square-free.
Applying Corollary~\ref{cor: OE1} and \ref{cor: OE2}, we get
$$
\int_{K_{x,\pfk}^\times \backslash B_\pfk^\times} \varphi_\pfk^o(b_\pfk^{-1}xb_\pfk)d^\times b_\pfk
= \frac{\text{vol}(O_{B_\pfk}^\times/O_\pfk^\times)}{\text{vol}(\Ocal_{-a_0,\pfk}^\times/O_\pfk^\times)} \cdot 
\sum_{\ell_\pfk = 0}^{c_{\pfk}}
\#\left(\frac{\Ocal_{-a_0,\pfk}^\times}{\Ocal_{-a_0\pfk^{2\ell_\pfk},\pfk}^\times}\right)\cdot
e(\Ocal_{-a_0\pfk^{2\ell_\pfk},\pfk},O_{B_\pfk}),
$$
and
\begin{eqnarray}
\int_{K_{x,\infty}^\times \backslash B_\infty^\times} \varphi_\infty^o(y b_\infty^{-1}xb_\infty)d^\times b_\infty
&=& \frac{1}{e_\infty(K_x/k)} \cdot 
\frac{\text{vol}(O_{B_\infty}^\times/O_\infty^\times)}{\text{vol}(O_{K_{x,\infty}}^\times/O_\infty^\times)}
\nonumber \\
&& \cdot
\begin{cases}
1, & \text{ if $k(\sqrt{-a})/k$ is imaginary;} \\
0, & \text{ otherwise.}
\end{cases}
\nonumber 
\end{eqnarray}
Thus $I^*(a,y;\varphi^o) = 0$ unless $-a \prec 0$.
In this case,
by Proposition~\ref{prop: VOL-CN} and Lemma~\ref{lem: CN} one has
$$
\text{vol}(K_x^\times  \backslash K_{x,\AA}^\times/k_\AA^\times)
= \frac{h(-a_0)}{w(-a_0)} \cdot e_\infty(K_x/k) \cdot \text{vol}(O_{K_{x,\infty}}^\times/O_\infty^\times) \cdot \prod_{\pfk} \text{vol}(\Ocal_{-a_0,\pfk}^\times/O_\pfk^\times).
$$
Finally, notice that for each non-zero prime ideal $\pfk$ of $A$, from Lemma~\ref{lem: OES1} and \ref{lem: OES2} we have
$$
e(\Ocal_{-a_0\pfk^{2\ell_\pfk},\pfk},O_{B_\pfk})\ = \ 
e^{\nfk^+,\nfk^-}_\pfk(\ell_\pfk).
$$
Hence we conclude that when $-a \prec 0$,
\begin{eqnarray}
I^*(a,y;\varphi^o) &=& \beta(y) \cdot \text{vol}(O_{B_\AA}^\times/O_\AA^\times)
 \nonumber \\
&& \cdot \frac{h(-a_0)}{w(-a_0)} \cdot \prod_\pfk \left[\sum_{\ell_\pfk = 0}^{c_\pfk} 
\#\left(\frac{\Ocal_{-a_0,\pfk}^\times}{\Ocal_{-a_0\pfk^{2\ell_\pfk},\pfk}^\times}\right)\cdot e^{\nfk^+,\nfk^-}(\ell_\pfk)\right]\nonumber \\
&=& \text{vol}(O_{B_\AA}^\times/O_\AA^\times) \cdot \tilde{\beta}(y) \cdot H^{\nfk^+,\nfk^-}(-a). \nonumber
\end{eqnarray}
The last equality follows from Proposition~\ref{prop: HCN}.
\end{proof}

\subsection{Alternative expression of the Fourier coefficients}\label{sec: AFC}
Given $a \in A$, recall that $V^o_a:=\{x \in B^o\mid x^2 = a\}$. 
For $y \in k_\infty^\times$ with $\deg a + 2 \leq 2 \ord_\infty(y)$,
in the proof of Lemma~\ref{lem: FCo} we have the following expression:
$$
I^*(-a,y;\varphi^o)
= \tilde{\beta}(y) \cdot \sum_{x \in V^o_a} \int_{B^\times  \backslash B_\AA^\times/k_\AA^\times}
\varphi^o(yb^{-1}xb)d^\times b.
$$
Put $O_{\widehat{B}}:= O_B\otimes_A \widehat{A}$ with $\widehat{A} := \prod_{\pfk} O_\pfk$.
From the strong approximation theorem one has the following bijection:
$$
O_B^\times  \backslash B_\infty^\times/k_\infty^\times \longleftrightarrow B^\times  \backslash 
B_\AA^\times /k_\AA^\times O_{\widehat{B}}
^\times.$$
Let $\Lambda^o_a:=B^o_a \cap \Lambda^o$.
Then we may write the Fourier coefficient $I^*(-a,y;\varphi^o)$ as
\begin{eqnarray}\label{eqn: AFC}
I^*(-a,y;\varphi^o)
&=& \tilde{\beta}(y)\cdot \text{vol}(O_{\widehat{B}}^\times/\widehat{A}) \cdot \sum_{x \in \Lambda_a^o}
\int_{O_B^\times  \backslash B_\infty^\times/k_\infty^\times} \varphi_\infty^o(y b_\infty^{-1} xb_\infty) d^\times b_\infty. \nonumber \\
&=&
\tilde{\beta}(y)\cdot \text{vol}(O_{\widehat{B}}^\times/\widehat{A}) \cdot \sum_{x \in O_B^\times \backslash \Lambda_a^o}
\int_{(O_B^\times\cap K_x^\times)  \backslash B_\infty^\times/k_\infty^\times} \varphi_\infty^o(y b_\infty^{-1} xb_\infty) d^\times b_\infty.
\end{eqnarray}

\begin{lem}\label{lem: Inf}
Given $a \in A$ with $a \prec 0$ and $x \in \Lambda_a^o$, we have
$$
\int_{(O_B^\times\cap K_x^\times)  \backslash B_\infty^\times/k_\infty^\times} \varphi_\infty^o(y b_\infty^{-1} xb_\infty) d^\times b_\infty
=\text{\rm vol}(O_{B_\infty}^\times/O_\infty^\times) \cdot \frac{q-1}{\#(O_B^\times \cap K_x^\times)}.
$$
\end{lem}

\begin{proof}
Observe that
\begin{eqnarray}
\int_{(O_B^\times\cap K_x^\times)  \backslash B_\infty^\times/k_\infty^\times} \varphi_\infty^o(y \cdot  b_\infty^{-1} xb_\infty) d^\times b_\infty
&=& \text{vol}\big((O_B^\times\cap K_x^\times)\backslash K_{x,\infty}^\times/k_\infty^\times\big) \nonumber \\
&& \cdot
\int_{K_{x,\infty}^\times \backslash B_\infty^\times}
\varphi_\infty^o(y \cdot b_\infty^{-1} x b_\infty) d^\times b_\infty. \nonumber 
\end{eqnarray}
Since
$$
\text{vol}\big((O_B^\times\cap K_x^\times)\backslash K_{x,\infty}^\times/k_\infty^\times\big)
= \frac{q-1}{\#(O_B^\times \cap K_x^\times)} \cdot e_\infty(K_{x,\infty}/k_\infty) \cdot \text{vol}(O_{K_{x,\infty}}^\times/O_\infty^\times)
$$
and by Corollary~\ref{cor: OE2} (as $a \prec 0$) one has
$$
\int_{K_{x,\infty}^\times \backslash B_\infty^\times}
\varphi_\infty^o(y \cdot b_\infty^{-1} x b_\infty) d^\times b_\infty
= \frac{1}{e_\infty(K_{x,\infty}/k_\infty)} \cdot 
\frac{\text{vol}(O_{B_{\infty}}^\times/O_\infty^\times)}{\text{vol}(O_{K_{x,\infty}}^\times /O_\infty^\times)},
$$
the result holds.
\end{proof}

The equation~\eqref{eqn: AFC} and Lemma~\ref{lem: Inf} show that for $a \in A$ with $a \prec 0$ and $y \in k_\infty^\times$ with $\deg a + 2 \leq 2 \ord_\infty(y)$, we get
$$
I^*(-a,y;\varphi^o)
= \text{\rm vol}(O_{B_\AA}^\times/O_\AA^\times) \cdot \tilde{\beta}(y) \cdot
\sum_{x \in O_B^\times \backslash \Lambda_a^o}
\frac{q-1}{\#(O_B^\times \cap K_x^\times)}.
$$
Moreover,
consider the (right) action of $\{\pm 1\}$ on $O_B^\times \backslash \Lambda_a^o$ by scalar multiplication.
For $[x] \in O_B^\times \backslash \Lambda_a^o$ (represented by $x \in \Lambda_a^o$), let $ [x]\cdot \{\pm 1\}$ be the $\{\pm 1\}$-orbit of $[x]$.
Then
$$
\#\big([x] \cdot \{\pm 1\} \big) = 
\begin{cases}
1, & \text{ if there exists $\gamma \in O_B^\times$ so that $\gamma^{-1} x \gamma = -x$;} \\
2, & \text{ otherwise.}
\end{cases} 
$$
For $x \in \Lambda_a^o$, put $\nu(x) := \#\big([x] \cdot \{\pm 1\}\big)$. 
We conclude that:

\begin{cor}\label{cor: AFC}
Given $a \in A$ with $a \prec 0$ and $y \in k_\infty^\times$ with $\deg a + 2 \leq 2 \ord_\infty(y)$, we have
$$
I^*(-a,y;\varphi^o)
= \text{\rm vol}(O_{B_\AA}^\times/O_\AA^\times) \cdot \tilde{\beta}(y) \cdot
\sum_{x \in O_B^\times \backslash \Lambda_a^o/\{\pm 1\}}
\frac{q-1}{\#(O_B^\times \cap K_x^\times)} \cdot \nu(x).
$$
\end{cor}

This expression allows us to connect $I^*(-a,y;\varphi^o)$ with a weighted sum over the ``CM points'' on the ``Drinfeld-Stuhler modular curves'' in Section~\ref{sec: CMBES}.

\subsection{Extension of $I(g;\varphi^o)$}\label{sec: Ext-PTI}

Notice that the Kubota $2$-cocycle $\sigma^1_\infty$ can be extended to $\GL_2(k_\infty)$ by:
$$\sigma_\infty(g,g'):=
\left(\frac{X(gg')}{X(g)},\frac{X(gg')}{\det(g)X(g')}\right)_\infty \cdot \frac{s_\infty(g) s_\infty(g')}{s_\infty(gg')}
$$
where
$$
s_\infty\left(\begin{pmatrix}a&b \\ c&d \end{pmatrix}\right):=
\begin{cases} \big(c,d/(ad-bc)\big)_\infty, &\text{ if $\ord_\infty(c)$ is odd and $d \neq 0$;}\\
1, & \text{ otherwise.}
\end{cases}
$$
Let $\widetilde{\GL}_2(k_\infty)$ be the central extension of $\GL_2(k_\infty)$ by  $\CC^1:=\{z \in \CC^\times \mid z\bar{z} = 1\}$ associated with $\sigma_\infty$, i.e.\ $\widetilde{\GL}_2(k_\infty)$ is identified with $\GL_2(k_\infty) \times \CC^1$ as sets with the following group law:
$$(g_1,\xi_1)\cdot (g_2,\xi_2) := \big(g_1g_2, \xi_1\xi_2 \sigma_\infty (g_1,g_2)\big), \quad \forall (g_1,\xi_1),(g_2,\xi_2) \in\widetilde{\GL}_2(k_\infty).
$$
Then $\widetilde{\SL}_2(k_\infty)$ becomes a subgroup of $\widetilde{\GL}_2(k_\infty)$.
We remark that $\sigma_\infty$ splits on $\GL_2(O_\infty)$, and the inclusion $k_\infty^\times \hookrightarrow \widetilde{\GL}_2(k_\infty)$ defined by
$$ y \longmapsto \tilde{y}:= \left(y, \frac{\varepsilon_\infty(1)}{\varepsilon_\infty(y)}\right) \in \widetilde{\GL}_2(k_\infty)$$
is a group homomorphism.
Moreover, for $y \in k_\infty^\times$ and $\tilde{g} \in \widetilde{\SL}_2(k_\infty)$, it is checked that
$$ \tilde{y} \cdot \tilde{g} = \tilde{g} \cdot \tilde{y}.$$

Let
$$
\Kcal_\infty^+ := \left\{\begin{pmatrix} a&b \\ c&d \end{pmatrix} \in \GL_2(O_\infty)\ \bigg|\ c \equiv 0 \bmod \pi_\infty \text{ and } (\pi_\infty,d)_\infty = 1\right\},
$$
and let $\widetilde{\Kcal}_\infty^+$ be the corresponding subgroup in $\widetilde{\GL}_2(k_\infty)$.
We introduce the following \textit{weight-$s$} operator $T_{\infty,s}$ on the space of functions $f$ on $\widetilde{\GL}_2(k_\infty)/\widetilde{\Kcal}_\infty^+$ (cf.\ \cite[p.\ 34]{CLWY}):
$$
T_{\infty,s} f (\tilde{g})
:= q^{s/2-1}\cdot \sum_{\epsilon \in \FF_q} f\left(\tilde{g}\cdot \left(\begin{pmatrix} \pi_\infty & \epsilon \\ 0&1\end{pmatrix},1\right)\right), \quad \forall \tilde{g} \in \widetilde{\GL}_2(k_\infty)/\widetilde{\Kcal}_\infty^+.
$$
\begin{defn}
Suppose $s \in \frac{1}{2}\ZZ$. A function $f$ on $\widetilde{\GL}_2(k_\infty)/\widetilde{\Kcal}_\infty^+$ is called a \textit{weight-$s$ metaplectic form} if 
$$ f((1,\xi)\cdot \tilde{g}) = \xi^{2s}\cdot f(\tilde{g})
\quad \text{ and } \quad 
T_{\infty,s} f (\tilde{g}) = f(\tilde{g}) \quad \forall \tilde{g} \in \widetilde{\GL}_2(k_\infty).
$$
\end{defn}

We shall extend $I(\cdot;\varphi^o)$ to a weight-$\frac{3}{2}$ metaplectic form.
Given $(g,\xi) \in \widetilde{\GL}_2(k_\infty)$,
suppose $\ord_\infty(\det(g)) \equiv 0 \bmod 2$.
Write $\det(g) = u \cdot \pi_\infty^{2\ell}$ with $\ell \in \ZZ$ and $u \in O_\infty^\times$.
One gets
$$
g^{(1)}:=  \pi_\infty^{-\ell} \cdot g \cdot \begin{pmatrix} u^{-1} & 0 \\ 0 & 1 \end{pmatrix} \in \SL_2(k_\infty).
$$
Thus we can decompose $(g,\xi)$ into
$$
(g,\xi) = \tilde{\pi}_\infty^{\ell}\cdot (g^{(1)},1) \cdot \left(\begin{pmatrix}u&0 \\ 0&1 \end{pmatrix},1\right) \cdot (1,\xi_1),$$
where
$$
\xi_1 := \xi\cdot \big(\pi_\infty^{\ell}, X(g)\big)_\infty \cdot \frac{\varepsilon_\infty(\pi_\infty^{\ell})}{\varepsilon_\infty(1)} \in \CC^1.
$$
When $\ord_\infty(\det(g)) \equiv 1 \bmod 2$, for $\epsilon \in \FF_q$ we put
$$(g_\epsilon,\xi_\epsilon):= (g,\xi) \cdot \left(\begin{pmatrix} \pi_\infty & \epsilon \\ 0 & 1 \end{pmatrix},1\right).$$
In particular, one has $\ord_\infty(g_\epsilon) \equiv 0 \bmod 2$.
We now define
$$
\vartheta^o(g,\xi):=\frac{1}{\text{vol}(O_{B_\AA}^\times/O_\AA^\times)}\cdot
\begin{cases}
\xi_1^3\cdot I\big((g^{(1)},1);\varphi^o\big), &  \text{ if $\ord_\infty(\det(g)) \equiv 0 \bmod 2$;} \\
& \\
\displaystyle q^{-1/4}\cdot \sum_{\epsilon \in \FF_q} \vartheta^o(g_\epsilon,\xi_\epsilon), & \text{ if $\ord_\infty(\det(g)) \equiv 1 \bmod 2$.}
\end{cases}
$$

Let
$$\Gamma_0^{(1)}(\nfk):= \left\{\begin{pmatrix} a&b \\ c&d \end{pmatrix} \in \SL_2(A)\ \bigg|\ c \equiv 0 \bmod \nfk\right\}.$$
Given $a \in A$ and $s \in \RR$, recall the following analogue of Bessel function introduced in Theorem~\ref{thm: MT2}:
for $y \in k_\infty^\times$,
$$\beta_{a,s}(y):=|y|_\infty^{s/2}\cdot
\begin{cases} 1, & \text{ if $\deg a + 2 \leq \ord_\infty(y)$;}\\
0, & \text{ otherwise.}
\end{cases}$$
We derive that:

\begin{thm}\label{thm: FET0}
The function $\vartheta^o$ is a well-defined weight-$\frac{3}{2}$ metaplectic form on $\widetilde{\GL}_2(k_\infty)/\Kcal_\infty^+$ satisfying that
$$
\vartheta^o(\tilde{z} \tilde{\gamma} \tilde{g}) = \vartheta^o(\tilde{g}),
\quad \forall z \in k_\infty^\times \text{ and }\gamma \in \Gamma_0^{(1)}(\nfk).$$
Moreover, for $(x,y) \in k_\infty\times k_\infty^\times$, we have the following Fourier expansion
\begin{eqnarray}\label{eqn: FE1}
\vartheta^o\left(\begin{pmatrix}y&x \\ 0&1\end{pmatrix},1\right)
&=& \sum_{d \in A,\ d \preceq 0} H^{\nfk^+,\nfk^-}(d) \cdot \big(\beta_{d,\frac{3}{2}}(y)\psi_\infty(-dx)\big).
\end{eqnarray}
\end{thm}

\begin{proof}
The definition of $\vartheta^o$ assures that
$$\vartheta^o(\tilde{z}\tilde{g}) = \vartheta^o(\tilde{g}) \quad \text{ for every $z \in k_\infty^\times$ and $\tilde{g} \in \widetilde{\GL}_2(k_\infty)$.}$$
Moreover, the invariant property of $I(\cdot;\varphi^o)$ in Proposition~\ref{prop: TLo} guarantees that 
$$\vartheta^o(\tilde{\gamma}\tilde{g}\tilde{\kappa}) = \vartheta^o(\tilde{g}) \quad \text{ for every } \gamma \in \Gamma_0^{(1)}(\nfk),\ \tilde{g} \in \widetilde{\GL}_2(k_\infty),\ \kappa \in \widetilde{\Kcal}_\infty^+.
$$
Notice that for $(x,y) \in k_\infty\times k_\infty^\times$ with $\ord_\infty(y) \equiv 0 \bmod 2$, we write $y = u \pi_\infty^{2\ell}$ with $\ell \in \ZZ$, and get
$$
\left(\begin{pmatrix}y&x \\ 0&1 \end{pmatrix},1\right)
= \tilde{\pi}_\infty^\ell \cdot \left(\begin{pmatrix} \pi_\infty^{\ell}& x \pi_\infty^{-\ell} \\ 0 & \pi_\infty^{-\ell}\end{pmatrix},1\right)
\cdot \left(1, \frac{\varepsilon_\infty(\pi_\infty^\ell)}{\varepsilon_\infty(1)}\right).
$$
Since $\varepsilon_\infty(y)^4/\varepsilon_\infty(1)^4 = 1$, from Theorem~\ref{thm: FC1} the equality~\eqref{eqn: FE1} holds.
The case when $\ord_\infty(y) \equiv 1 \bmod 2$ is similar.

Finally, for each $\tilde{g} \in \widetilde{\GL}_2(k_\infty)$, there exists $z \in k_\infty^\times$, $\gamma \in \Gamma_0^{(1)}(\nfk)$, $(x,y) \in k_\infty\times k_\infty^\times$, $\kappa \in \Kcal_\infty^+$, and $\xi \in \CC^1$ such that
$$\tilde{g} = \tilde{z}\cdot \tilde{\gamma} \cdot \left(\begin{pmatrix} y&x \\ 0&1 \end{pmatrix},1\right) \cdot \tilde{\kappa} \cdot (1,\xi).
$$
To show the weight-$3/2$ property of $\vartheta^o$, it suffices to assume that
$$\tilde{g} = \left(\begin{pmatrix}y&x \\ 0&1 \end{pmatrix},1\right), \quad \text{ for } (x,y) \in k_\infty\times k_\infty^\times.
$$
Therefore the equality $T_{\infty,\frac{3}{2}} \vartheta^o = \vartheta^o$ holds by comparing the Fourier expansions on both sides.
\end{proof}

\section{CM points on the Drinfeld-Stuhler modular curves}\label{sec: CMBES}

Let $\CC_\infty$ be the completion of a chosen algebraic closure of $k_\infty$.
The \textit{Drinfeld half plane} is
$$\Hfk:= \CC_\infty - k_\infty,$$
which equipped with the M\"obius action of $\GL_2(k_\infty)$:
$$\begin{pmatrix} a&b \\ c&d\end{pmatrix} \cdot z := \frac{az+b}{cz+d}, \quad \forall \begin{pmatrix} a&b \\ c&d\end{pmatrix} \in \GL_2(k_\infty),\ z \in \Hfk.
$$

We recall the analytic construction of Drinfeld-Stuhler modular curves as follows.
Let $B$ be an indefinite quaternion algebra over $k$, and $\nfk^- \in A_+$ be the product of the primes at which $B$ is ramified. 
Take a square-free $\nfk^+ \in A_+$ coprime to $\nfk^-$, and let $O_B$ be an Eichler $A$-order in $B$ of type $(\nfk^+,\nfk^-)$.
Fix an isomorphism $B\otimes_k k_\infty \cong \Mat_2(k_\infty)$, which embeds $\Gamma(\nfk^+,\nfk^-):= O_B^\times$ into $\GL_2(k_\infty)$ as a discrete subgroup.
This induces an action of $O_B^\times$ on the Drinfeld half plane $\Hfk$.
Let 
$$X(\nfk^+,\nfk^-):= \Gamma(\nfk^+,\nfk^-)\backslash \Hfk,$$ which is a rigid analytic space (compact if $B$ is division).
From the moduli interpretation of $X(\nfk^+,\nfk^-)$ (parametrizing the ``$\Bscr$-ellptic sheaves with additional level-$\nfk^+$ structure'', cf.\ \cite{LRS} and \cite{Pak}),
we may identify $X(\nfk^+,\nfk^-)$ (rigidly analytically) with the $\CC_\infty$-valued points of a smooth curve (projective if $B$ is division) over $\CC_\infty$,
called the \textit{Drinfeld-Stuhler modular curve for $\Gamma(\nfk^+,\nfk^-)$}.
For our purpose, we shall only use the analytic description of $X(\nfk^+,\nfk^-)$.

Notice that when $B = \Mat_2(k)$,  every Eichler $A$-order $O_B$ of type $(\nfk^+,1)$ is equal (up to conjugation) to
$$\left\{\begin{pmatrix} a&b \\ c&d\end{pmatrix} \in \Mat_2(A)\ \bigg|\ c \equiv 0  \bmod \nfk^+\right\},
$$
and so $\Gamma(\nfk^+,\nfk^-)$ coincides with the congruence subgroup 
$$
\Gamma_0(\nfk^+)
:= \left\{\begin{pmatrix} a&b \\ c&d \end{pmatrix} \in \GL_2(A)\ \bigg|\ c \equiv 0 \bmod \nfk^+\right\}.
$$
The ``compactification'' 
$$ X_0(\nfk^+) := \Gamma_0(\nfk)\backslash \Big(\Hfk\cup \PP^1(k)\Big)$$ 
is called the \textit{Drinfeld modular curve for $\Gamma_0(\nfk^+)$}.\\

Recall that $B^o$ consists of all the pure quaternions in $B$, $\Lambda^o:=O_B\cap B^o$, and $$\Lambda^o_a := \{x \in \Lambda^o_a\mid Q^o(x) = -x^2 = a\}.$$

\begin{defn}\label{defn: CMP}
Given non-zero $d \in A$, a point $\boz  \in X(\nfk^+,\nfk^-)$ is  \textit{CM with discriminant $d$} if there exists a representative $z \in \Hfk$ of $\boz$ and $x \in \Lambda^o_d$
so that $x\cdot z = z$.
\end{defn}

\begin{rem}
The CM points on $X_0(\nfk^+,\nfk^-)$ can be viewed as analogue of CM points on Shimura curves in the classical case.
Indeed, when $\nfk^+ = 1$ (i.e.\ $O_B$ is a maximal $A$-order in $B$), the CM points are in bijection with the isomorphism classes of the Drinfeld-Stuhler $O_B$-modules with ``complex multiplications'', see \cite[Theorem 4.10 and \textit{Remark} 4.11]{Pak}.
\end{rem}

Let $\boz \in X(\nfk^+,\nfk^-)$ be a CM point with discriminant $d \in A$ which is represented by $z \in \Hfk$.
Set
$$
w(z) := \frac{q-1}{\#\big(\text{Stab}_{\Gamma(\nfk^+,\nfk^-)}(z)\big)}.$$
Here $\Stab_{\Gamma(\nfk^+,\nfk^-)}(z)$ is the stablizer of $z$ in $\Gamma(\nfk^+,\nfk^-)$.
Then $w(z)$ only depends on the point $\boz \in X(\nfk^+,\nfk^-)$.
Put $w(\boz) = w(z)$.
For non-zero $d \in A$, we are interested in the following \textit{mass}:
$$
\Mcal(d):= \sum_{\subfrac{\text{CM } \boz \in X(\nfk^+,\nfk^-)}{\text{with discriminant $d$}}}
\frac{1}{w(\boz)}.
$$
Let 
$$
S_d:=\{z \in \Hfk \mid \text{ there exists } x \in \Lambda_d^o \text{ so that } x \cdot z = z\}.
$$
Then $S_d$ is empty unless $-d \prec 0$.
In this case, we have $k_\infty(z) = k_\infty(\sqrt{-d})$ for each $z \in S_d$.
Identifying $\gal(k_\infty(\sqrt{-d})/k_\infty)$ with $\{\pm 1\}$, we have a (right) action of $\{\pm 1\}$ on $S_d$ commuting with the (left) action of $O_B^\times$.
This then induces an action of $\{\pm 1\}$ on $O_B^\times \backslash S_d$.
Denote by $\bar{z} \in k_\infty(\sqrt{-d})$ as the conjugate of $z$.
For $z \in S_d$, put
$$
\nu(z):=
\begin{cases}
1, & \text{ if there exists $\gamma \in O_B^\times$ so that $\gamma \cdot z = \bar{z}$;}\\
2, & \text{ otherwise.}
\end{cases}
$$
Then we may express $\Mcal(d)$ in the following:

\begin{prop}\label{prop: M(a)}
Given $d \in A$ with $d \prec 0$, we have 
$$
\Mcal(d) = \sum_{z \in O_B^\times \backslash S_d/\{\pm 1\}} \frac{\nu(z)}{w(z)}.
$$
\end{prop}

\begin{proof}
It is from the definition that
$$
\Mcal(d) = \sum_{z \in O_B^\times \backslash S_d} \frac{1}{w(z)}.
$$
Moreover, for each element in $O_B^\times \backslash S_d$ represented by $z \in S_d$, the cardinality of its $\{\pm 1\}$-orbit equals to $\nu(z)$.
Therefore the result follows.
\end{proof}

We have the following bijection
$$S_d/\{\pm 1\} \quad \longleftrightarrow \quad  \Lambda_d^o/\{\pm 1\}$$
by sending $[z] \in S_d/\{\pm 1\}$ to $[x_z] \in \Lambda_d^o/\{\pm 1\}$, where
$x_z \cdot z = z$.
To see this is, for $z \in S_d$ and $x_z \in \Lambda_d^o$ with $x_z \cdot z = z$,
one has $x_z \cdot \bar{z} = \bar{z}$, which says that two vectors $(z,1)^t$ and $(\bar{z},1)^t$ are distinct eigenvectors of $x_z$. 
If $x \in \Lambda_d^o$ with $x \cdot z = z$, then $x$ and $x_z$ share the same eigenvectors, which implies that $x \in K_{x_z} = k(x_z)\subset B$. Since $x^2 = -d = x_z^2$, we must have $x = \pm x_z$, i.e.\
$$
\{x \in \Lambda_d^o \mid x \cdot z = z\} = \{\pm x_z\}.
$$
On the other hand for $x \in \Lambda_d^o$ and $z_x \in \Hfk$ with $x\cdot z_x = z_x$,
we have
$$\{z \in \Hfk \mid x \cdot z = z\} = \{z_x, \bar{z}_x\}.$$
Therefore the above map is well-defined and bijective.
It is clear that the bijection is actually $O_B^\times$-equivariant.
Moreover:
\begin{lem}
Given $d \in A$ with $d \prec 0$, $z \in S_d$ and $x_z \in \Lambda^o_d$ with $x_z\cdot z = z$, we have
$$
w(z) = \frac{q-1}{\#(O_B^\times \cap K_{x_z}^\times)}
\quad \text{ and } \quad
\nu(z) = \nu(x_z).
$$
\end{lem}

\begin{proof}
We first show that $O_B^\times \cap K_{x_z}^\times = \text{Stab}_{O_B^\times}(z)$.
To see this, for $\gamma \in O_B^\times$, one has
$\gamma \in K_{x_z}$ if and only if $\gamma$ and $x_z$ share the same eigenvectors, which is equivalent to that $\gamma$ fixes $z$ (and $\bar{z}$), i.e.\ $\gamma \in \text{Stab}_{O_B^\times}(z)$.

To prove the second assertion, suppose
$\nu(z) = 1$, i.e.\ there exists $\gamma \in O_B^\times$ so that $\gamma z = \bar{z}$.
In this case, we obtain that
$$(\gamma x_z \gamma^{-1})\cdot z = z,$$ 
which implies $\gamma x_z \gamma^{-1} = \pm x_z$.
If $\gamma x_z \gamma^{-1} = x_z$, i.e.\ $\gamma \in K_{x_z}$, then $\gamma \cdot z = z$, a contradiction.
Hence $\gamma x_z \gamma^{-1} = -x_z$, showing that $\nu(x_z) = 1$.
Conversely, if $\nu(x_z) = 1$, i.e.\ there exists $\gamma \in O_B^\times$ so that $\gamma x_z \gamma^{-1} = -x_z$.
Then $$x_z \cdot (\gamma \cdot z) = \gamma \cdot z,$$
showing that $\gamma \cdot z = z$ or $\bar{z}$.
If $\gamma \cdot z = z$, then $\gamma$ and $x_z$ share the same eigenvectors. Thus $\gamma x_z \gamma^{-1} = x_z$, a contradiction.
Therefore $\gamma \cdot z = \bar{z}$, i.e.\ $\nu(z) = 1$.
\end{proof}

From Proposition~\ref{prop: M(a)}, the above lemma enables us to rewrite $\Mcal(d)$ as
$$
\Mcal(d) = \sum_{x \in O_B^\times\backslash \Lambda_d^o/\{\pm 1\}}\frac{q-1}{\#(O_B^\times \cap K_{x}^\times)} \cdot \nu(x).
$$
For convention, let
$$\Mcal(0) := H^{\nfk^+,\nfk^-}(0).$$
Together with Corollary~\ref{cor: AFC}, we arrive at:

\begin{thm}\label{thm: CNo-geo}
Let $\nfk^+,\nfk^- \in A_+$ be two square-free polynomials with $\text{\rm gcd}(\nfk^+,\nfk^-) = 1$. 
Suppose the number of prime factors of $\nfk^-$ is positive and even.
Given $d \in A$ and $y \in k_\infty^\times$ with $\deg d + 2 \leq 2 \ord_\infty(y)$, the following equality holds:
$$
I^*(-d,y;\varphi^o) = \text{\rm vol}(O_{B_\AA}^\times/O_\AA^\times) \cdot \tilde{\beta}(y) \cdot \Mcal^{\nfk^+,\nfk^-}(d).
$$
Consequently, the mass $\Mcal^{\nfk^+,\nfk^-}(d)$ over the CM points with discriminant $d$ coincides with $H^{\nfk^+,\nfk^-}(d)$ for every $d \in A$ with $d \prec 0$.
\end{thm}

Recall that $\vartheta^o$ is a weight-$\frac{3}{2}$ metaplectic form on $\widetilde{\GL}_2(k_\infty)$ extended from $I(\cdot;\varphi^o)$.
%$$\vartheta^o(\tilde{g}):= \frac{1}{\text{\rm vol}(O_{B_\AA}^\times/O_\AA^\times)} \cdot I(\tilde{g},\varphi^o), \quad \forall \tilde{g} \in \widetilde{\SL}_2(k_\AA).$$
From Theorem~\ref{thm: FET0}, the above theorem leads us to:

\begin{cor}\label{cor: TSo-FC}
For $(x,y) \in k_\infty \times k_\infty^\times$, the Fourier expansion of $\vartheta^o$ can be written via masses:
$$
\vartheta^o\left(\begin{pmatrix} y&x\\ 0&1\end{pmatrix},1\right)
=
\sum_{d \in A}
\Mcal^{\nfk^+,\nfk^-}(d)\cdot  \big(\beta_{d,\frac{3}{2}}(y) \psi_\infty(-dx)\big).
$$
\end{cor}

\begin{rem}
(1) Using Eichler's theory of optimal embeddings (cf.\ \cite[p.\ 94]{Vie}), the equality
$$H^{\nfk^+,\nfk^-}(d) = \Mcal^{\nfk^+,\nfk^-}(d) \quad \forall d \in A,\ d\prec 0 $$
also holds for $\nfk^- = 1$.\\
(2) When the number of prime factors of $\nfk^-$ is odd, the generating function 
$$
\sum_{d \in A, \ d \prec 0}
H^{\nfk^+,\nfk^-}(d) \cdot \big(\beta_{d,\frac{3}{2}}(y)\psi_\infty(-dx)\big), \quad \forall (x,y) \in k_\infty\times k_\infty^\times
$$
extends as well to a ``weight-$3/2$'' metaplectic form on $\widetilde{\GL}_2(k_\infty)$ (cf.\ \cite[Proposition 2.2]{Wei} and \cite[Corollary 5.4]{CLWY}).
\end{rem}

\section{Theta series with nebentypus}\label{sec: TSN}

Fix a square-free $\dfk \in A_+$ with $\deg \dfk$ even.
Let $F = k(\sqrt{\dfk})$.
For each $\alpha \in F$, the Galois conjugate of $\alpha$ (over $k$) is denoted by $\alpha'$.
Given $x = \begin{pmatrix} a&b \\ c&d\end{pmatrix} \in \Mat_2(F)$, put
$$
\bar{x} := \begin{pmatrix} d&-b \\ -c&a\end{pmatrix} \quad \text{ and } \quad
x' := \begin{pmatrix} a'&b' \\ c'&d'\end{pmatrix}.
$$
Given $\nfk \in A_+$,
let $*$ be the involution on $\Mat_2(F)$
defined by: for $x \in \Mat_2(F)$,
$$
x^* := \begin{pmatrix}0& 1/\nfk \\ 1&0 \end{pmatrix} \bar{x}' \begin{pmatrix} 0 & 1 \\ \nfk&0 \end{pmatrix}.
$$
Let
$$
V:= \{x \in \Mat_2(F)\mid x^* = x\} \quad \text{ and } \quad Q_V := \det\big|_V.
$$
Then $(V,Q_V)$ is a quadratic space with degree $4$ over $k$.
In concrete terms, we have
$$
V = \left\{\begin{pmatrix}a&\beta \\ -\nfk\beta' & d \end{pmatrix}\ \bigg|\
a,d \in k,\ \beta \in F\right\}.
$$

Let 
$$B_1 := \{b \in \Mat_2(F)\mid b^* = \bar{b}\}
= 
\left\{\begin{pmatrix}\alpha&\beta \\ \nfk\beta' & \alpha' \end{pmatrix}\ \bigg|\
\alpha,\beta \in F\right\}.
$$
We may identify $B_1$ with the quaternion algebra 
$$\left(\frac{\dfk,\nfk}{k}\right) := k+k\mathbf{i}+k\mathbf{j}+k\mathbf{ij}, \quad 
\text{where }
\mathbf{i}^2 = \dfk,\ \mathbf{j}^2 = \nfk, \mathbf{ij}=-\mathbf{ji},
$$
where $\mathbf{i}$ corresponds to $\begin{pmatrix} \sqrt{\dfk}&0 \\ 0& - \sqrt{\dfk} \end{pmatrix}$
and $\mathbf{j}$ corresponds to $\begin{pmatrix}0&1 \\ \nfk &0\end{pmatrix}$.
From now on, we make the following assumptions:
\begin{apn}\label{apn: level}
${}$
\begin{itemize}
    \item[(1)] The polynomial $\nfk \in A_+$ is square-free and coprime to $\dfk$.
    \item[(2)] Write $\nfk = \nfk^+\cdot \nfk^-$ (resp.\ $\dfk = \dfk^+\cdot \dfk^-$), where each prime factor $\pfk$ of $\nfk^{\pm}$ (resp.\ $\dfk^{\pm}$) satisfies that the Legendre quadratic symbol $\left(\frac{\dfk}{\pfk}\right) = \pm 1$
    (resp.\ $\left(\frac{\nfk}{\pfk}\right) = \pm 1$). Then $\deg (\dfk^-\nfk^-) > 0$.
\end{itemize}
\end{apn}

\begin{rem}
Under the above assumptions, we observe that $B_1$ is the indefinite division quaternion algebra over $k$ ramified precisely at prime factors of $\dfk^-\nfk^-$.
\end{rem}

Consider the following left exact sequence
$$
1 \longrightarrow k^\times \longrightarrow B_1^\times \longrightarrow \text{SO}(V),
$$
where the map from $B_1^\times$ into $\text{SO}(V)$ is defined by
$$
b \longmapsto h_b := (x \mapsto bxb^{-1}), \quad \forall b \in B_1^\times.
$$

Given $\varphi \in S(V(k_\AA))$, we are interested in the following theta integral:
$$I(g;\varphi) := \int_{B_1^\times  \backslash B_{1,\AA}^\times/k_\AA^\times} \Theta(g,h_b;\varphi) d^\times b, \quad \forall g \in \SL_2(\AA).
$$
For $a \in k$ and $y \in k_\AA^\times$, let
$$
I^*(a,y;\varphi):=
\int_{k\backslash k_\AA}I\left(\bigg(\begin{pmatrix} y & uy^{-1} \\ 0 & y^{-1} \end{pmatrix},1\bigg);\varphi\right) \psi(-a u) du.
$$
Put $V_a:= \{x \in V\mid Q_V(x) = a\}$.
Adapting the proof of Lemma~\ref{lem: FCo}, we obtain:

\begin{lem}\label{lem: FCV}
For $a \in k$ and $y \in \AA^\times$, we have
\begin{eqnarray}
I^*(a,y;\varphi) 
%&:=& \int_{B^\times k_\AA^\times \backslash B_\AA^\times} \Theta^*(a,y;h_b;\varphi) d^\times b \nonumber \\
&=& |y|_\AA^2 \cdot \sum_{x \in B_1^\times \backslash V_a} \text{\rm vol}(K_x^\times  \backslash K_{x,\AA}^\times/k_\AA^\times) \cdot 
\int_{K_{x,\AA}^\times \backslash B_{1,\AA}^\times} \varphi(yb^{-1}xb) d^\times b. \nonumber 
\end{eqnarray}
Here $K_x$ is the centralizer of $x$ in $B_1$, and $K_{x,\AA} = K_x \otimes_k k_\AA$.
\end{lem}

%Similar to the Section~\ref{sec: TSP}, 
We shall take a special pure-tensor Schwartz function and determine the Fourier coefficients of the associated theta integral.

\subsection{Particular Schwartz function}\label{sec: TSN-PSF}

Note that the trace map $\tr: \Mat_2(F) \rightarrow F$ restricting to $V$ gives a $k$-linear functional on $V$.
%Via the embedding $B \hookrightarrow B\otimes_k F \cong \Mat_2(F)$, the restriction of $\tr$ to $B$ coincides with the reduced trace map on $B$.
%Recall that 
For $x \in V$, put
$$x^{\natural}:= \left(x - \frac{\tr(x)}{2}\right) \cdot \sqrt{\dfk} \in B_1^o, $$
where 
$B_1^o$ is the space of of pure quaternions in $B_1$, i.e.\
$$B_1^o = \{b \in B \mid \tr(b) = 0\}.$$
Then 
$$
K_x = \begin{cases}
k(x^\natural), & \text{ a quadratic field over $k$ if $x^\natural \neq 0$,}\\
B_1, & \text{otherwise.}
\end{cases}
$$
For $a \in k$, notice that two elements $x_1$ and $x_2$ in $V_a$ belong to the same orbit of $B_1^\times$ (under the conjugation action) if and only if $\tr(x_1) = \tr(x_2)$.
This follows from 
$$
a = Q_V(x) = \frac{\tr(x)^2}{4} - \frac{(x^\natural)^2}{\dfk}, \quad \forall x \in V_a.
$$

Take
$$\Lambda := \Mat_2(O_F) \cap V =
\left\{\begin{pmatrix}a&\beta \\ -\beta'\nfk & d \end{pmatrix} \bigg|\ a,d \in A,\ \beta \in O_F\right\}
$$
and
$$O_{B_1} := \Mat_2(O_F) \cap B_1 = 
\left\{\begin{pmatrix} \alpha & \beta \\ \beta' \nfk & \alpha' \end{pmatrix}\bigg|\ \alpha,\beta \in O_F\right\}.
$$
It is checked that $O_{B_1}$ is an Eichler $A$-order in $B_1$ of type $(\dfk^+\nfk^+, \dfk^-\nfk^-)$, and $u^{-1}xu \in \Lambda$ for every $x \in \Lambda$ and $u \in O_{B_1}^\times$.
For each non-zero prime ideal $\pfk$ of $A$,
put $\Lambda_\pfk:= \Lambda\otimes_A O_\pfk$ and $\Lambda_\pfk^\natural := O_{B_{1,\pfk}}^o = \{b \in O_{B_{1,\pfk}}\mid \tr(b) = 0\}$.
Then:

\begin{lem}
For $x_\pfk \in V(k_\pfk)$ with $Q_V (x_\pfk) \in O_\pfk$, we have
$$
x_\pfk \in \Lambda_\pfk \quad \text{ if and only if } \quad 
\tr(x_\pfk) \in O_\pfk \text{ and } x_\pfk^\natural \in \Lambda_\pfk^\natural.
$$
\end{lem}

\begin{proof}
It is straightforward that when $x_\pfk \in \Lambda_\pfk$, one has $\tr(x_\pfk) \in O_\pfk$ and $x_\pfk^\natural \in \Lambda_\pfk^\natural$.
Conversely, suppose $\tr(x_\pfk) \in O_\pfk$ and $x_\pfk^\natural \in \Lambda_\pfk^\natural$.
Write
$$x_\pfk^\natural = \begin{pmatrix} a\sqrt{\dfk} & \beta \\ \beta'\nfk & -a\sqrt{\dfk} \end{pmatrix} \quad \text{ with } a \in O_\pfk,\ \beta \in O_{F,\pfk} := O_F \otimes_A O_\pfk, \quad \text{ and } \quad t = \tr(x_\pfk) \in O_\pfk.
$$
Then 
$$
Q_V(x_\pfk) = \frac{t^2}{4} - \frac{(x_\pfk^\natural)^2}{\dfk}
= \frac{t^2}{4} + \frac{a \dfk + \Nr_{F/k}(\beta)\nfk}{\dfk} \in O_\pfk.
$$
Since $\nfk$ is coprime to $\dfk$, we obtain that
$$\frac{\Nr_{F/k}(\beta)}{\dfk} \in O_\pfk,
$$
which shows that $\beta = \sqrt{\dfk} \cdot \tilde{\beta}$ for some $\tilde{\beta} \in O_{F,\pfk}$ (as $\dfk$ is square-free).
Therefore
$$x_\pfk = \frac{t}{2} + \frac{x_\pfk^\natural}{\sqrt{\dfk}}
= \begin{pmatrix}
t/2 + a & \tilde{\beta} \\ -\tilde{\beta}\nfk & t/2 - a
\end{pmatrix} \in \Lambda_\pfk.
$$
\end{proof}

Take
$$
\varphi_{\Lambda,\pfk} := \mathbf{1}_{\Lambda_\pfk} \in S(V(k_\pfk)),
$$
and
$$
\varphi^{\natural}_\pfk := \mathbf{1}_{\Lambda_\pfk^\natural} \in S(B_{1,\pfk}^o).
$$
The above lemma says that for $x_\pfk \in V(k_\pfk)$, we have
\begin{eqnarray}\label{eqn: compare-phi}
\varphi_{\Lambda,\pfk}(x_\pfk) = 1 \quad \text{ if and only if } \quad 
\tr(x) \in O_\pfk \text{ and } \varphi^\natural_\pfk (x^\natural_\pfk) = 1.
\end{eqnarray}

As $\dfk$ is monic with even degree, the field $F$ is \textit{real} over $k$, i.e.\ the infinite place $\infty$ of $k$ splits in $F$.
Fix an embedding $F\hookrightarrow k_\infty$, which induces a $k_\infty$-algebra isomorphism
$$B_{1,\infty} = B_1\otimes_k k_\infty \cong \Mat_2(k_\infty).$$
On the other hand, the natural inclusion $V \hookrightarrow \Mat_2(F)$ also induces an isomorphism $V(k_\infty) \cong \Mat_2(k_\infty)$ (as quadratic spaces over $k_\infty$).
Recall in Section~\ref{sec: PTI1} that we take
$$L_\infty := \varpi \cdot \Mat_2(O_\infty),$$
and
$$L_\infty' := \left\{\begin{pmatrix} a & b \\ c & d \end{pmatrix}
\in L_\infty \ \bigg|\ c \in \varpi^2 O_\infty\right\}. 
$$
Via the identification $V(k_\infty) \cong \Mat_2(k_\infty)$, 
we may view $L_\infty$ and $L_\infty'$ as two $O_\infty$-lattices in $V_\infty$.
Choose
$$
\varphi_{\Lambda,\infty} := \mathbf{1}_{L_\infty} - \frac{q+1}{2} \cdot \mathbf{1}_{L_\infty'} \in S(V(k_\infty))
\quad \text{ and } \quad
\varphi_\infty^\natural := \mathbf{1}_{L_\infty^o} - \frac{q+1}{2} \cdot \mathbf{1}_{L_\infty^{\prime,o}} \in S(B_{1,\infty}^o).
$$
It is straightforward that:

\begin{lem}
For $x \in V(k_\infty)$ with $ \tr(x) \in \varpi O_\infty$, we have
$$
 \varphi_{\Lambda,\infty}( x) = 
 \varphi_\infty^\natural(\sqrt{\dfk} \cdot  x^\natural).
$$
\end{lem}

Let 
$$\varphi_\Lambda := \otimes_\pfk \varphi_{\Lambda,\pfk} \otimes \varphi_{\Lambda,\infty} \in S(V(\AA)).$$
Similar to Proposition~\ref{prop: TLo}, we have the following transformation law of $I(g;\varphi_\Lambda)$:

\begin{prop}\label{prop: TLN}
Given $\gamma \in \SL_2(k)$, $g \in \SL_2(k_\AA)$ and $\kappa \in \Kcal_0^1(\dfk \nfk \infty)$, we have
$$
I(\gamma g \kappa; \varphi_\Lambda)
= \chi_\dfk(\kappa) \cdot I(g;\varphi_\Lambda).
$$
Here for $\kappa = \begin{pmatrix} a&b \\ c&d\end{pmatrix} \in \Kcal_0^1(\dfk \nfk \infty)$
with $d = (d_v)_v \in O_\AA$ $($and so $d_\pfk \in O_\pfk^\times$ for each prime factor $\pfk$ of $\dfk)$, we put
$$
\chi_\dfk(\kappa) := \prod_{\pfk \mid \dfk}\left(\frac{d_\pfk}{\pfk}\right).
$$
\end{prop}

This transformation law implies in particular that for $a \in k$, $y \in k_\AA^\times$, $\varepsilon \in O_\AA^\times$, and $u \in O_\AA$, we have
$$
I^*(a,y;\varphi_\Lambda) = I^*(\alpha^{-2}a, \alpha y \varepsilon; \varphi_\Lambda)=
I^*(a,y;\varphi_\Lambda) \cdot \psi(ay^2 u).
$$
Thus it suffices to consider $y \in k_\infty^\times$, and in this case we get
$$I^*(a,y;\varphi_\Lambda) = 0
\quad \text{ unless $a\in A$ with $\deg a + 2 \leq 2 \ord_\infty(y)$}.
$$
Next, we shall express $I^*(a,y;\varphi_\Lambda)$ in terms of the modified Hurwitz class numbers.

\subsection{Fourier coefficients of $I(g;\varphi_\Lambda)$}

For $y \in k_\infty^\times$ and $a \in A$ with $\deg a + 2 \leq 2 \ord_\infty(y)$, we get
\begin{eqnarray}
I^*(a,y;\varphi_\Lambda)
&=& |y|_\infty^2 \cdot \sum_{x \in B_1^\times \backslash V_a} \text{vol}(K_x^\times  \backslash K_{x,\AA}^\times/k_\AA^\times) \nonumber \\
&& \cdot \left(\prod_\pfk \int_{K_{x,\pfk}^\times \backslash B_{1,\pfk}^\times}\varphi_{\Lambda,\pfk}(b_\pfk^{-1}xb_\pfk) d^\times b_\pfk\right)
\cdot \int_{K_{x,\infty}^\times \backslash B_{1,\infty}^\times} \varphi_{\Lambda,\infty}(yb_\infty^{-1}xb_\infty) d^\times b_\infty.
\nonumber
\end{eqnarray}
Thus for $x \in V_a$,
$$
\prod_\pfk \int_{K_{x,\pfk}^\times \backslash B_{1,\pfk}^\times}\varphi_{\Lambda,\pfk}(b_\pfk^{-1}xb_\pfk) d^\times b_\pfk = 0
\quad \text{ unless \ \ $t = \tr(x) \in A$.}
$$
In this case, 
$$
K_x = \begin{cases}
k(\sqrt{\dfk(t^2-4a)}), & \text{ if $x \notin k$;}\\
B_1, & \text{ otherwise.}
\end{cases}
$$
Moreover, for $x \in V_a$ so that $K_x/k$ is an imaginary quadratic extension, the condition $\deg a + 2 \leq 2 \ord_\infty(y)$ implies that $y\cdot \tr(x) \in \varpi O_\infty$.
Hence by \eqref{eqn: compare-phi} we have
\begin{eqnarray}
&& \prod_\pfk \int_{K_{x,\pfk}^\times \backslash B_{1,\pfk}^\times}\varphi_{\Lambda,\pfk}(b_\pfk^{-1}xb_\pfk) d^\times b_\pfk
\cdot \int_{K_{x,\infty}^\times \backslash B_{1,\infty}^\times} \varphi_{\Lambda,\infty}(yb_\infty^{-1}xb_\infty) d^\times b_\infty \nonumber \\
&=& 
\prod_\pfk \int_{K_{x,\pfk}^\times \backslash B_{1,\pfk}^\times}\varphi_{\pfk}^\natural(b_\pfk^{-1}x^\natural b_\pfk) d^\times b_\pfk
\cdot \int_{K_{x,\infty}^\times \backslash B_{1,\infty}^\times} \varphi^\natural_\infty\big((y\sqrt{\dfk}) \cdot b_\infty^{-1}x^\natural b_\infty) d^\times b_\infty.
\nonumber
\end{eqnarray}

For $x \in V_a$ with $\tr(x) = t \in A$, one has $$(x^\natural)^2 = \dfk(t^2-4a).$$
As the Eichler $A$-order $O_{B_1}$ is of type $(\dfk^+\nfk^+,\dfk^-\nfk^-)$, similar to Theorem~\ref{thm: FC1} we obtain
\begin{eqnarray}
&& \text{vol}(K_x^\times  \backslash K_{x,\AA}^\times/k_\AA^\times) \cdot \left(\prod_\pfk \int_{K_{x,\pfk}^\times \backslash B_{1,\pfk}^\times}\varphi_{\Lambda,\pfk}(b_\pfk^{-1}xb_\pfk) d^\times b_\pfk\right)
\nonumber \\
&& \quad \quad \quad \quad \quad \quad \quad \quad \,
\cdot \int_{K_{x,\infty}^\times \backslash B_{1,\infty}^\times} \varphi_{\Lambda,\infty}(yb_\infty^{-1}xb_\infty) d^\times b_\infty \nonumber \\
&=& 
\text{vol}(K_x^\times  \backslash K_{x,\AA}^\times/k_\AA^\times) \cdot \left(\prod_\pfk \int_{K_{x,\pfk}^\times \backslash B_{1,\pfk}^\times}\varphi_{\pfk}^\natural(b_\pfk^{-1}x^\natural b_\pfk) d^\times b_\pfk\right) 
\nonumber \\
&& \quad \quad \quad \quad \quad \quad \quad \quad \,
\cdot \int_{K_{x,\infty}^\times \backslash B_{1,\infty}^\times} \varphi^\natural_\infty\big((y\sqrt{\dfk}) \cdot b_\infty^{-1}x^\natural b_\infty) d^\times b_\infty \nonumber \\
&=& \text{vol}(O_{B_\AA}^\times/O_{\AA}^\times) \cdot 
\begin{cases}
H^{\dfk^+\nfk^+,\dfk^-\nfk^-}(\dfk(t^2-4a)), & \text{ if $\dfk(t^2-4a) \preceq 0$;}\\
0, & \text{ otherwise.}
\end{cases} \nonumber
\end{eqnarray}

Notice that two elements $x_1, x_2 \in V_a$ belong to the same $B_1^\times$-orbit if and only if $\tr(x_1) = \tr(x_2)$.
From the above discussion, we conclude that:

\begin{thm}\label{thm: TINFC}
Given $a \in A$ and $y \in k_\infty^\times $ with $\deg a +2 \leq 2\ord_\infty(y)$, the following equality holds:
$$
I^*(a,y;\varphi_\Lambda)
= \text{\rm vol}(O_{{B_1},\AA}^\times/O_\AA^\times) \cdot |y|_\infty^2 \cdot \sum_{\subfrac{t \in A}{t^2\preceq 4a}} H^{\dfk^+ \nfk^+,\dfk^-\nfk^-}(\dfk(t^2-4a)).
$$
\end{thm}

\subsection{Alternative expression of the Fourier coefficients}

For $y \in k_\infty^\times$ and $a \in A$,
one may express $I^*(a,y;\varphi_\Lambda)$ as
$$
I^*(a,y;\varphi_\Lambda)
= |y|_\infty^2 \cdot \sum_{x \in V_a} \int_{B_1^\times  \backslash B_{1,\AA}^\times/k_\AA^\times} \varphi_\Lambda(y b_\infty ^{-1} x b_\infty ) d^\times b_\infty .
$$
As in Section~\ref{sec: AFC}, we apply the strong approximation theorem and get
$$
I^*(a,y;\varphi_\Lambda)
= |y|_\infty^2 \cdot \text{vol}(O_{\widehat{B}_1}^\times/\widehat{A}) \cdot \sum_{x \in \Lambda_a}
\int_{O_{B_1}^\times  \backslash B_{1,\infty}^\times/k_\infty^\times} \varphi_{\Lambda,\infty}(y b_\infty ^{-1} xb_\infty ) d^\times b_\infty .
$$
Let
$$\Gamma = \Gamma_{0,F}(\nfk)
:= \left\{\begin{pmatrix}    
a&b \\ c&d \end{pmatrix} \in \GL_2(O_F)\ \bigg|\ ad-bc \in \FF_q^\times,\ c \equiv 0 \bmod \nfk \right\}.
$$
Define an action $\star$ of $\Gamma$ on $\Lambda$ by:
$$\gamma \star x := \gamma x \gamma^* \cdot \det(\gamma)^{-1}, \quad \forall \gamma \in \Gamma,\ x \in \Lambda.$$
Given a non-zero $a \in A$ and $x \in \Lambda_a$, let
$$B_x:= \{b \in \Mat_2(F) \mid x b^* = \bar{b} x\}, \text{ and }
\Gamma_x := B_x^\times \cap \Gamma.
$$
Then $\Gamma_x$ is actually the stablizer of $x$ in $\Gamma$, and:
\begin{lem}\label{lem: B_x}
Given a non-zero $a\in A$ and $x \in \Lambda_a$, $B_x$ is a quaternion algebra over $k$ which is isomorphic to:
$$
\left(\frac{\dfk, a\nfk}{k}\right)
:= k +k\mathbf{i}+k\mathbf{j}+k\mathbf{ij},
\quad
\text{ where $\mathbf{i}^2 = \dfk$, $\mathbf{j}^2 = a\nfk$, and $\mathbf{ji} = -\mathbf{ij}$.}
$$
\end{lem}

\begin{proof}
Write $x = \begin{pmatrix} d_1 & \beta \\ -\nfk \beta' & d_2 \end{pmatrix}$
where $d_1,\ d_2 \in A$ and $\beta \in O_F$ with $d_1d_2 + \nfk \beta\beta' = a$.
Take 
$$
U:= \begin{cases}
\begin{pmatrix}1 & 0 \\ \nfk \beta' & d_1 \end{pmatrix}, & \text{ if $d_1 \neq 0$;} \\
\begin{pmatrix}  d_2 & -\beta \\ 0 & a \end{pmatrix}, & \text{ if $d_1 = 0$ and $d_2 \neq 0$ } \\
\begin{pmatrix} -1 & \beta \\ \nfk \beta & \nfk \beta^2 \end{pmatrix}, & \text{ if $d_1 = d_2 = 0$.}
\end{cases}
$$
Then
$$
x_U:= U x U^* = \begin{pmatrix} 1& 0 \\ 0 &a \end{pmatrix} \cdot 
\begin{cases}
d_1, &\text{ if $d_1 \neq 0$;}\\
a d_2, &\text{ if $d_1 = 0$ and $d_2 \neq 0$;}\\
2 a, & \text{ if $d_1 = d_2 = 0$.}
\end{cases}
$$
It is straightforward that 
$B_x = U^{-1}  B_{x_U}  U$
and
$$
B_{x_U}
= \left\{\begin{pmatrix} \alpha & \beta \\ a \nfk \beta' & \alpha' \end{pmatrix}
\ \bigg|\ \alpha, \beta \in F\right\}.
$$
Thus
$$
B_x \cong B_{x_U} \cong
\left(\frac{\dfk,a\nfk}{k}\right).
$$
\end{proof}

\begin{rem} 
Observe that $B_x = B_1$ if and only if $x \in k^\times$.
In this case, $a$ is a square in $A$, and $\Gamma_x = O_{B_1}^\times$.
\end{rem}

The following lemma is straightforward.

\begin{lem}
For $a \in A$ and $y \in k_\infty^\times$, we have
\begin{eqnarray}
&& I^*(a,y;\varphi_\Lambda) \nonumber \\
&=& |y|^2_\infty \cdot \text{\rm vol}(O_{\widehat{B}_1}^{\times}/\widehat{A})
\cdot
\sum_{x \in \Gamma\backslash \Lambda_a}
\sum_{\gamma \in \Gamma_1 \backslash \Gamma / \Gamma_x}
\int_{(\Gamma_1 \cap \Gamma_{\gamma \star x})\backslash B_{1,\infty}^\times /k_\infty^\times} \varphi_{\Lambda,\infty}(yb_\infty^{-1}(\gamma \star x)b_\infty)d^\times b_\infty .
\nonumber
\end{eqnarray}
\end{lem}

The integral inside the above sum expression is determined by:

\begin{lem}
Let $a \in A$ and $y \in k_\infty^\times$ with $\deg a + 2 \leq 2 \ord_\infty(y)$. Take $x \in V_a$.
\begin{itemize}
    \item[(1)]
    If $B_{x} = B_1$, then
    $\Gamma_{x} = \Gamma_1$ and
$$
\int_{(\Gamma_1\cap \Gamma_{x})\backslash B_{1,\infty}^\times /k_\infty^\times} \varphi_{\Lambda,\infty}(yb_\infty ^{-1}x b_\infty )d^\times b_\infty
= \frac{1-q}{2} \cdot \text{\rm vol}(\Gamma_1 \backslash B_{1,\infty}^\times/k_\infty^\times).
$$
\item[(2)]
If $B_x \neq B_1$,
then $B_{x} \cap B_1 = K_{x}$, and
\begin{eqnarray}
&& \int_{(\Gamma_1 \cap \Gamma_{x})\backslash B_{1,\infty}^\times /k_\infty^\times} \varphi_{\Lambda,\infty}(yb_\infty ^{-1}x b_\infty )d^\times b_\infty  \nonumber \\
&=& \text{\rm vol}(O_{B_{1,\infty}}^\times/O_\infty^\times)\cdot
\begin{cases}
\displaystyle\frac{q-1}{\#(\Gamma_1 \cap \Gamma_{x})}, & \text{ if $K_{x}/k$ is imaginary;}\\
0, & \text{ otherwise.}
\end{cases}
\nonumber 
\end{eqnarray}
\end{itemize}
\end{lem}

\begin{proof}
When $B_x = B_1$, we get $x \in k^\times$ with $x^2 = a$.
Thus the condition $\deg a + 2 \leq  2\ord_\infty(y)$ implies
$$
\varphi_{\Lambda,\infty}(yb_\infty^{-1}xb_\infty) = 1-\frac{q+1}{2} = \frac{1-q}{2}, \quad \forall b_\infty \in B_{1,\infty}^\times.
$$
Hence the assertion (1) holds.\\

For (2), the integral vanishes unless $K_x/k$ is imaginary.
In this case, $\Gamma_1\cap \Gamma_x$ is a finite subgroup of $K_x^\times$,
and 
\begin{eqnarray}
&& \int_{(\Gamma_1 \cap \Gamma_{x})\backslash B_{1,\infty}^\times /k_\infty^\times} \varphi_{\Lambda,\infty}(yb_\infty ^{-1}x b_\infty )d^\times b_\infty  \nonumber \\
&=& \frac{\text{vol}(K_{x,\infty}^\times/k_\infty^\times)}{\#(\Gamma_1\cap \Gamma_x)}
\cdot \int_{K_{x,\infty}^\times \backslash B_{1,\infty}^\times} \varphi_\infty^{\natural}\big((y\sqrt{\dfk}) b_\infty^{-1} x^\natural b_\infty\big) d^\times b_\infty \nonumber \\
&=& 
\text{\rm vol}(O_{B_{1,\infty}}^\times/O_\infty^\times)\cdot
\frac{q-1}{\#(\Gamma_1 \cap \Gamma_{x})}. \nonumber 
\end{eqnarray}
The last equality follows from Corollary~\ref{cor: OE2}.
\end{proof}

For $x \in \Lambda$, put
$$i(x):=
\begin{cases}
-H^{\dfk^+\nfk^+,\dfk^-\nfk^-}(0), & \text{ if $B_x = B_1$;} \\
\displaystyle\frac{q-1}{\#(\Gamma_1\cap \Gamma_x)}, & \text{ if $K_x/k$ is imaginary;} \\
0, & \text{ otherwise.}
\end{cases}
$$
Define
$$
\Ical(x) := \sum_{\gamma \in \Gamma_1 \backslash \Gamma/\Gamma_x} i(\gamma\star x).
$$
We then obtain:

\begin{thm}\label{thm: AFCN}
Given $a \in A$ and $y \in k_\infty^\times$ with $2\ord_\infty(y)+2 \geq \deg a$, we have:
$$
 I^*(a,y;\varphi_\Lambda)
= \text{\rm vol}(O_{B_\AA}^\times/O_\AA^\times)
\cdot |y|_\infty^2 \cdot \sum_{x \in \Gamma \backslash \Lambda_a} \Ical(x).
$$
\end{thm}

The above theorem enables us to connect the Fourier coefficients of the theta integral  $I(g;\varphi_\Lambda)$ with the intersection numbers of the ``Hirzebruch-Zagier-type divisors'' on the ``Drinfeld-Stuhler modular surfaces'' in Section~\ref{sec: Intersection}.

\subsection{Extension of $I(g;\varphi_\Lambda)$}
\label{sec: Ext-TI}
Let
$$
\Kcal_\infty := \left\{ \begin{pmatrix} a&b \\ c&d \end{pmatrix} \in \GL_2(O_\infty) \ \bigg|\ c \equiv 0 \bmod \varpi\right\}
$$
and
$$
\Gamma_0(\dfk \nfk) :=
\left\{ \begin{pmatrix} a&b \\ c&d \end{pmatrix} \in \GL_2(A) \ \bigg|\ c \equiv 0 \bmod \dfk \nfk \right\}.
$$
Put $\Kcal_\infty^1 := \Kcal_\infty \cap \SL_2(k_\infty)$ and $\Gamma_0^1(\dfk \nfk) := \Gamma_0(\dfk\nfk) \cap \SL_2(A)$.
From the strong approximation theorem, the natural embedding $\SL_2(k_\infty)\hookrightarrow \SL_2(k_\AA)$ induces the following bijection
$$ \Gamma^1_0(\dfk\nfk) \backslash \SL_2(k_\infty)/\Kcal^1_\infty \longleftrightarrow 
\SL_2(k)\backslash \SL_2(k_\AA)/\Kcal_0^1(\dfk\nfk\infty).
$$
This allows us to view $I(g;\varphi_\Lambda)$ as a function on $\SL_2(k_\infty)/\Kcal^1_\infty$ satisfying
$$
I(\gamma g_\infty;\varphi_\Lambda)
= \chi_\dfk(\gamma) I(g_\infty;\varphi_\Lambda), \quad \forall g_\infty \in \SL_2(k_\infty) \text{ and } \gamma \in \Gamma_0^1(\dfk\nfk).
$$
We shall extend $I(\cdot;\varphi_\Lambda)$ to a function $\vartheta_\Lambda$ on $\GL_2(k_\infty)/k_\infty^\times \Kcal_\infty$ which is ``Drinfeld-type'', i.e.\ the following \textit{harmonic property} holds:
for $g_\infty \in \GL_2(k_\infty)$ we have
$$
\vartheta_\Lambda(g) + \vartheta_\Lambda\left(g \begin{pmatrix} 0&1\\ \varpi&0\end{pmatrix}\right) = 0 = 
\sum_{\kappa \in \GL_2(O_\infty)/\Kcal_\infty}\vartheta_\Lambda(g\kappa).
$$

\begin{rem}\label{rem: wt-2}
Let $f$ be a Drinfeld-type automorphic form on $\GL_2(k_\infty)/k_\infty^\times \Kcal_\infty$.
The harmonicity of $f$ implies that for $g \in \GL_2(k_\infty)$ we have
$$
\sum_{\epsilon \in \FF_q} f\left(g\begin{pmatrix} \pi_\infty &\epsilon \\ 0&1 \end{pmatrix}\right) = f(g).
$$
Extending $f$ to a function on the metaplectic group $\widetilde{\GL}_2(k_\infty)$ (introduced in Section~\ref{sec: Ext-PTI}) by
$$f(g,\xi) := \xi^4 \cdot f(g), \quad \forall \xi \in \CC^1.$$
The harmonicity implies in particular that $f$ is actually a weight-$2$ form, i.e.\ (cf.\ \cite[Example 4.2]{CLWY})
$$
f((1,\xi)\tilde{g}) = \xi^4 \cdot f(\tilde{g})
\quad \text{ and } \quad T_{\infty,2}f(\tilde{g}) = f(\tilde{g}), \quad \forall \xi \in \CC^1 \text{ and } \tilde{g} \in \widetilde{\GL}_2(k_\infty).
$$
Viewed as
analogue to classical weight-two modular forms, Drinfeld-type automorphic forms are objects of great interest in the study of
function field arithmetic.
We refer the readers to \cite{G-R}, \cite{CLWY}, \cite{CWY}, and \cite{Wei}) for further discussions.
\end{rem}

Let $w_\infty:= \begin{pmatrix} 0&1\\ \varpi&0\end{pmatrix}$.
We first prove that:

\begin{lem}\label{lem: har}
Given $g_\infty \in \SL_2(k_\infty)$, the following equality holds:
$$
\sum_{\kappa \in \SL_2(O_\infty)/\Kcal^1_\infty} I(g_\infty \kappa;\varphi_\Lambda) = 0 =
\sum_{\kappa \in \SL_2(O_\infty)/\Kcal^1_\infty} I(g_\infty w_\infty^{-1} \kappa w_\infty ;\varphi_\Lambda).
$$
\end{lem}

\begin{proof}
Notice that $\omega_\infty^V(\kappa)\mathbf{1}_{L_\infty} = \mathbf{1}_{L_\infty}$ for every $\kappa \in \SL_2(O_\infty)$, and
$$
\sum_{\kappa \in \SL_2(O_\infty)/\Kcal^1_\infty}\omega_\infty^V(\kappa)\mathbf{1}_{L_\infty'}
= \mathbf{1}_{L_\infty} + \mathbf{1}_{w_\infty L_\infty w_\infty'}.
$$
Thus
\begin{eqnarray}
&& \sum_{\kappa \in \SL_2(O_\infty)/\Kcal_\infty^1}I(g_\infty\kappa;\varphi_\Lambda) \nonumber \\
&=&
(q+1)\cdot I(g_\infty;\otimes_\pfk \varphi_{\Lambda,\pfk} \otimes \mathbf{1}_{L_\infty})
- \frac{q+1}{2}\cdot
I\Big(g_\infty;\otimes_\pfk \varphi_{\Lambda,\pfk} \otimes \left(\mathbf{1}_{L_\infty}+\mathbf{1}_{w_\infty L_\infty w_\infty^{-1}}\right)\Big) \nonumber \\
&=&
(q+1)\cdot I(g_\infty;\otimes_\pfk \varphi_{\Lambda,\pfk} \otimes \mathbf{1}_{L_\infty})
- \frac{q+1}{2}\cdot 2 \cdot 
I(g_\infty;\otimes_\pfk \varphi_{\Lambda,\pfk} \otimes  \mathbf{1}_{L_\infty}) \nonumber \\
&=& 0. \nonumber 
\end{eqnarray}
Similarly, let 
$$
L_\infty'':= \begin{pmatrix} \varpi O_\infty & O_\infty \\ \varpi O_\infty & O_\infty \end{pmatrix}.
$$
Then
$$
\sum_{\kappa \in \SL_2(O_\infty)/\Kcal^1_\infty}\omega_\infty^V\left(w_\infty \kappa w_\infty^{-1}\right) \mathbf{1}_{L_\infty}
= 
q \cdot \sum_{\kappa \in SL_2(O_\infty)/\Kcal^1_\infty}
\mathbf{1}_{\kappa L_\infty'' \kappa^{-1}},
$$
and
$$
\sum_{\kappa \in \SL_2(O_\infty)/\Kcal^1_\infty}\omega_\infty^V(w_\infty \kappa w_\infty^{-1}) \mathbf{1}_{L_\infty'}
= q\cdot \left(\mathbf{1}_{L_\infty''} + \mathbf{1}_{w_\infty L_\infty'' w_\infty^{-1}}\right).
$$
Therefore 
\begin{eqnarray}
&& \sum_{\kappa \in \SL_2(O_\infty)/\Kcal_\infty^1}I\left(g_\infty w_\infty \kappa w_\infty^{-1};\varphi_\Lambda\right) \nonumber \\
&=&
\sum_{\kappa \in \SL_2(O_\infty)/\Kcal_\infty^1}
q\cdot I\left(g_\infty; 
\otimes_\pfk \varphi_{\Lambda,\pfk} \otimes \mathbf{1}_{\kappa L_\infty''\kappa^{-1}}\right) \nonumber \\
&& - \frac{q+1}{2}\cdot q\cdot 
I\Big(g_\infty;\otimes_\pfk \varphi_{\Lambda,\pfk} \otimes \left(\mathbf{1}_{L_\infty''}+\mathbf{1}_{w_\infty L_\infty'' w_\infty^{-1}}\right)\Big) \nonumber \\
&=&
(q+1)\cdot q \cdot I\left(g_\infty;\otimes_\pfk \varphi_{\Lambda,\pfk} \otimes \mathbf{1}_{L_\infty''}\right)
- \frac{q+1}{2}\cdot q \cdot 2 \cdot 
I\left(g_\infty;\otimes_\pfk \varphi_{\Lambda,\pfk} \otimes  \mathbf{1}_{L_\infty''}\right) \nonumber \\
&=& 0. \nonumber 
\end{eqnarray}
\end{proof}

Let $$\GL_2^+(k_\infty):= \{g \in \GL_2^+(k_\infty)\mid \ord_\infty(\det g) \equiv 0 \bmod 2\}.$$
The natural inclusion $\SL_2(k_\infty) \hookrightarrow \GL_2(k_\infty)$ gives a bijection $$\SL_2(k_\infty)/\Kcal^1_\infty \longleftrightarrow \GL_2^+(k_\infty)/k_\infty^\times \Kcal_\infty.$$
Thus $I(\cdot;\varphi_\Lambda)$ can be viewed as a function on 
$\GL_2^+(k_\infty)/k_\infty^\times \Kcal_\infty$.
For $g_\infty \in \GL_2(k_\infty)$, 
define $\vartheta_\Lambda(g_\infty)$ by:
$$
\vartheta_\Lambda(g_\infty) := \frac{2}{\text{vol}(O_{B_\AA}^\times/O_\AA^\times)} \cdot 
\begin{cases}
I(g_\infty;\varphi_\Lambda), & \text{ if $g_\infty \in \GL_2^+(k_\infty)$;}\\
-I(g_\infty w_\infty ;\varphi_\Lambda), & \text{ otherwise.}
\end{cases}
$$

The above lemma implies immediately that:

\begin{thm}\label{thm: Ext-DTAF}
The function $\vartheta_\Lambda$ on $\GL_2(k_\infty)/k_\infty^\times \Kcal_\infty$ satisfies the harmonic property, i.e.\ for $g_\infty \in \GL_2(k_\infty)$,
$$
\vartheta_\Lambda(g_\infty) + \vartheta_\Lambda(g_\infty w_\infty) = 0 = 
\sum_{\kappa \in \GL_2(O_\infty)/\Kcal_\infty}\vartheta_\Lambda(g_\infty \kappa).
$$
Moreover, for $\gamma \in \Gamma_0^{(1)}(\dfk\nfk)$ we have
$$
\vartheta_\Lambda(\gamma g_\infty) = \chi_\dfk(\gamma) \cdot \vartheta_\Lambda(g_\infty), \quad \forall g_\infty \in \GL_2(k_\infty).
$$
Here for $\gamma = \begin{pmatrix}  a&b \\ \dfk\nfk c & d \end{pmatrix} \in \Gamma_0^{(1)}(\dfk\nfk)$, 
$\chi_\dfk(\gamma)$ is equal to the Legendre quadratic symbol $\left(\frac{d}{\dfk}\right)$.
\end{thm}

\begin{proof}
The second assertion follows directly from Proposition~\ref{prop: TLN}.
To show the harmonicity of $\vartheta_\Lambda$, by definition we get immediately that
$$\vartheta_\Lambda(g_\infty) + \vartheta_\Lambda(g_\infty w_\infty) = 0.
$$
Moreover, suppose $g_\infty \in \GL_2^+(k_\infty)$.
Then by Lemma~\ref{lem: har}, one has
$$\sum_{\kappa \in \GL_2(O_\infty)/\Kcal_\infty}\vartheta_\Lambda(g_\infty \kappa)
= \frac{2}{\text{vol}(O_{B_\AA}^\times/O_\AA^\times)} \cdot \sum_{\kappa \in \SL_2(O_\infty)/\Kcal_\infty^1} I(g_\infty \kappa; \varphi_\Lambda) = 0.
$$
When $g_\infty \notin \GL_2^+(k_\infty)$,
by Lemma~\ref{lem: har} again we get
$$\sum_{\kappa \in \GL_2(O_\infty)/\Kcal_\infty}\vartheta_\Lambda(g_\infty \kappa)
= \frac{-2}{\text{vol}(O_{B_\AA}^\times/O_\AA^\times)} \cdot \sum_{\kappa \in \SL_2(O_\infty)/\Kcal_\infty^1} I\big((g_\infty w_\infty) w_\infty^{-1} \kappa w_\infty; \varphi_\Lambda\big) = 0.
$$
Therefore the proof is complete.
\end{proof}

\begin{rem}\label{rem: DTS-FC}
In conclusion, we extend the theta integral $I(\cdot;\varphi_\Lambda)$ to a Drinfeld-type automorphic form $\vartheta_\Lambda$ on $\GL_2(k_\infty)$ for the congruence subgroup $\Gamma_0^1(\dfk\nfk)$ with nebentypus $\chi_\dfk$, whose Fourier expansion is: 
for $(x,y) \in k_\infty\times k_\infty^\times$,
$$
\vartheta_\Lambda\begin{pmatrix} y & x \\ 0 & 1\end{pmatrix}
= 2 \cdot \sum_{a \in A} \left(\sum_{\subfrac{t \in A}{t^2 \preceq 4a}}H^{\dfk^+\nfk^+,\dfk^-\nfk^-}\big(\dfk(t^2-4a)\big)\right) \cdot \big(\beta_{a,2}(y)\psi_\infty(a x)\big).
$$
\end{rem}

\section{Intersections of the Hirzebruch-Zagier-type divisors}\label{sec: Intersection}

Here we follow the notations in the last section.
Let $\dfk \in A_+$ be square-free with $\deg \dfk$ even and $F = k(\sqrt{\dfk})$.
Identifying $F_\infty := F\otimes_k k_\infty \cong k_\infty \times k_\infty$,
we denote the image of $\alpha \in F$ in $k_\infty^2$ by $(\alpha, \alpha')$ (i.e.\ $\alpha'$ is the Galois conjugate of $\alpha$ over $k$).
Let $\Hfk_F:= \Hfk\times \Hfk$, equipped
with the M\"obius action of $\GL_2(k_\infty)^2$.
The above embedding $F\hookrightarrow k_\infty\times k_\infty$ gives $\GL_2(F) \hookrightarrow \GL_2(k_\infty)^2$, which induces an action of $\GL_2(F)$ on $\Hfk_F$. 
In concrete terms, for $g = \begin{pmatrix}a&b \\ c&d \end{pmatrix} \in \GL_2(F)$ and $\vec{z} = (z_1,z_2) \in \Hfk_F$, define
$$
g \cdot \vec{z} := \left(\frac{az_1+b}{cz_1+d}\ ,\  \frac{a'z_2+b'}{c'z_2+d'}\right).
$$

For $\nfk \in A_+$,
recall that
$$\Gamma_{0,F}(\nfk)
= \left\{\begin{pmatrix} 
a&b \\ c&d \end{pmatrix} \in \GL_2(O_F)\ \bigg|\ ad-bc \in \FF_q^\times,\ c \equiv 0 \bmod \nfk \right\}.
$$
The \textit{Drinfeld-Stuhler modular surface for $\Gamma_{0,F}(\nfk)$} is $$\Scal_{0,F}(\nfk) := \Gamma_{0,F}(\nfk)\backslash \Hfk_F.$$
From the work of Stuhler \cite{Stu}, $\Scal_{0,F}(\nfk)$ is a moduli space of the so-called ``Frobenius-Hecke sheaves'' (an analogue of the Hilbert-Blumental abelian surfaces in the classical case) with additional ``level-$\nfk$ structure''.
This provides the algebraic structure of the surface $\Scal_{0,F}(\nfk)$.
%We remark that the surface $\Scal_{0,F}(\nfk)$ also parametrizes the ``Hilbert-Blumental $t$-modules'' over $\CC_\infty$ introduced by Anderson in \cite{And}. 
For our purpose, we only consider $\Scal_{0,F}(\nfk)$ as a rigid analytic space, and study the intersections of the ``Hirzebruch-Zagier-type'' divisors on $\Scal_{0,F}(\nfk)$.

\subsection{Hirzebruch-Zagier-type divisors}

Recall in Section~\ref{sec: TSN} that we let
$$V = \{x \in \Mat_2(F) \mid x^* = x\} \quad \text{ and } \quad \Lambda  = V \cap \Mat_2(O_F).$$
Given $x \in \Lambda$ with $\det(x) \neq 0$,
let $\Ccal_x := \Gamma_x \backslash \Hfk$, the Drinfeld-Stuhler modular curve for $\Gamma_x$.
Put $$S_x := \begin{pmatrix}0&1 \\ \nfk & 0 \end{pmatrix} \bar{x}.$$
The closed immersion
$\Hfk \rightarrow \Hfk_F$ defined by $(z\mapsto (z,S_x z))$ induces a (rigid analytic) proper morphism
$f_x : \Ccal_x \rightarrow \Scal_{0,F}(\nfk)$.
We put 
$X_x := f_x(\Ccal_x)$ and $\Zcal_x:= f_{x,*}(\Ccal_x)$, the pushforward divisor of $\Ccal_x$ under $f_x$ on $\Scal_{0,F}(\nfk)$.
Let
$$
\widehat{\Gamma}_x := \{\gamma \in \Gamma_{0,F}(\nfk)\mid \gamma \star x = \pm x\}.
$$
Then $[\widehat{\Gamma}_x : \Gamma_x] = 1 \text{ or } 2$, and:
\begin{lem}
For $x \in \Lambda$ with $\deg(x) \neq 0$, one has
$$\Zcal_x = [\widehat{\Gamma}_x : \Gamma_x]\cdot X_x.
$$
\end{lem}

\begin{proof}
We need to show that the proper morphism $f_x : \Ccal_x \rightarrow X_x$
has degree equal to $[\widehat{\Gamma}_x :\Gamma_x]$.

Let $\boz_1, \boz_2 \in \Ccal_x$ be two points with $f_x(\boz_1) = f_x(\boz_2) \in X_x$.
Take representatives $\vec{z}_1 = (z_1,S_x z_1)$ and $\vec{z}_2 = (z_2, S_x z_2)$ of $\boz_1$ and $\boz_2$ on $\Hfk_F$, respectively.
There exists $\gamma \in \Gamma$ so that 
$$\vec{z}_1 = \gamma \cdot \vec{z}_2,\ \text{ i.e.\ }
(z_1,S_x z_1) = (\gamma z_2, \gamma' S_x z_2).
$$
Thus
$$
z_1 =  \gamma z_1 = \gamma ((\gamma'S_x)^{-1}S_x) z_1
= \big((\gamma \star x) \bar{x}\big) z_1.
$$
When $z_1$ is in ``general position'', e.g.\ the stablizer of $z_1$ in $\GL_2(F)$ is $F^\times$,
one has
$(\gamma\star x) \bar{x} \in F^\times$.
Taking the determinant of $(\gamma\star x) \bar{x}$, we obtain
$\gamma \star x = \pm x$, which says that $\gamma \in \widehat{\Gamma}_x$.
Therefore the result holds.
\end{proof}

Suppose now that $\nfk$ satisfies Assumption~\ref{apn: level}.
We shall study the number of intersections of $\Zcal_1$ and $\Zcal_x$ by lifting to a ``fine covering'' of $\Scal_{0,F}(\nfk)$.
More precisely, for $\mfk \in A_+$, we let
$$
\Gamma_F(\mfk) := \left\{ \gamma \in \GL_2(O_F)\ \bigg|\ \gamma \equiv \begin{pmatrix} 1&0 \\ 0&1 \end{pmatrix} \bmod \mfk\right\}.
$$
Choose $\mfk$ so that $\nfk^2$ divides $\mfk$.
Then 
$$\Gamma_F(\mfk)^*:= \{ \gamma^* \mid \gamma \in \Gamma_F(\mfk)\} \subset \Gamma_F(\nfk).$$
Consider the finite morphism
$$\pi:\Gamma_F(\mfk)\backslash \Hfk_F =: \Scal_F(\mfk) \twoheadrightarrow \Scal_{0,F}(\nfk).$$
For $x \in \Lambda$ with $\det(x) \neq 0$,
let $\Hfk_x := \{(z,S_x z)\mid z \in \Hfk\} \subset \Hfk_F$ and
$\widetilde{X}_x$ be the image of $\Hfk_x$ in $\Scal_F(\mfk)$ under the canonical map from $\Hfk_F$ onto $\Scal_F(\mfk)$.
Let $\Gamma_x(\mfk) := \Gamma_x \cap \Gamma_F(\mfk)$, and put
$\widetilde{\Ccal}_x := \Gamma_x(\mfk) \backslash \Hfk$.
We have:

\begin{lem}
Assume $\nfk^2 \det(x)$ divides $\mfk$.
Then the identification between $\Hfk \cong \Hfk_x$ induces an isomorphism $\tilde{f}_x: \widetilde{\Ccal}_x \cong \widetilde{X}_x$.
\end{lem}

\begin{proof}
Notice that the defining equation of $\Hfk_x$ in $\Hfk_F$ makes it smooth everywhere.
As each point in $\Hfk_F$ has trivial stablizer in $\Gamma_F(\mfk)$, we may identify a sufficiently small admissible open neighborhood of a given point in $\Hfk_F$ with the corresponding affinoid subdomains in $\Scal_F(\mfk)$.
This assures the smoothness of $\widetilde{X}_x$.
Therefore it suffices to show that the morphism
from $\tilde{f}_x: \widetilde{C}_x\rightarrow \widetilde{X}_x$ is a bijection.

The surjectivity of $\tilde{f}_x$ comes directly from the definition.
On the other hand, let $\tilde{\boz}_1$ and $\tilde{\boz}_2$ be two points on $\widetilde{\Ccal}_x$ so that $\tilde{f}_x(\tilde{\boz}_1) = \tilde{f}_x(\tilde{\boz}_2)$.
Take representatives $\vec{z}_1 = (z_1,S_x z_1)$ and $\vec{z}_2 = (z_2, S_x z_2)$ of $\tilde{\boz}_1$ and $\tilde{\boz}_2$ on $\Hfk_x$, respectively.
Then there exists $\gamma \in \Gamma_F(\mfk)$ so that
$\vec{z}_1 = \gamma \cdot \vec{z}_2$, i.e.\
$$
(z_1, S_x z_1 ) = (\gamma z_2, \gamma' S_x z_2).
$$
Thus
$$ z_1 = \gamma z_2 = \big(\gamma (\gamma' S_x)^{-1} S_x\big) z_1 = \big((\gamma \star x)\bar{x}\big) z_1.
$$
Since $\nfk^2 \det (x)$ divides $\mfk$, we obtain that 
$$
(\gamma \star x)x^{-1} = \gamma x \gamma^* x^{-1} \cdot \det(\gamma)^{-1} \in \Gamma_F(\nfk), \quad \text{i.e. }
(\gamma \star x)x^{-1} \equiv \begin{pmatrix} 1&0 \\ 0 &1 \end{pmatrix} \bmod \nfk.
$$
As it fixes $z_1 \in \Hfk$, we obtain that 
$$
(\gamma \star x)x^{-1} = 1, \quad \text{i.e.} \quad
\gamma \star x = x.
$$
Hence $\gamma \in \Gamma_x \cap \Gamma_F(\mfk) = \Gamma_x(\mfk)$.
In other words, the morphism $\tilde{f}_x$ is bijective, and the proof is complete.
\end{proof}

\subsection{Formula of intersections}

Let $x \in \Lambda$ with $\det(x) \neq 0$, and $\mfk \in A_+$ with $\nfk^2 \det(x) \mid \mfk$.
We first verify the transversality of the intersections of $\widetilde{X}_1$ and $\widetilde{X}_x$ on $\Scal_F(\mfk)$.

\begin{lem}
Suppose $\widetilde{X}_1 \neq \widetilde{X}_x$.
Then $\widetilde{X}_1$ and $\widetilde{X}_x$ intersect transversally.
\end{lem}

\begin{proof}
It suffices to check that the preimages of $\widetilde{X}_1$ and $\widetilde{X}_x$ in $\Hfk_F$ intersect transversally.
Since $\gamma \Hfk_x = \Hfk_{\gamma \star x}$ for every $\gamma \in \Gamma$, it is reduced to show the transversality of the intersection of $\Hfk_1$ and $\Hfk_x$ when $x \notin A$.

Suppose $\vec{z} (z,S_1z) = (z,S_x z) \in \Hfk_1 \cap \Hfk_x$.
Write $x = \begin{pmatrix} d_1 & \beta \\ -\nfk \beta' & d_2\end{pmatrix}$ with $d_1,d_2 \in k$ and $\beta \in F$,
and put $a := \det(x) = d_1d_2 + \nfk \beta\beta' \neq 0$.
Then $\bar{x} z = z$, i.e.\
$$
\frac{d_2 z - \beta}{\nfk\beta z + d_1} = z.
$$
Thus $\nfk \beta z^2 +(d_1-d_2)z +\beta = 0$.
Multiplying $\beta$ on both sides we get
\begin{eqnarray}\label{eqn: T1}
(a-d_1d_2)z^2 + (d_1-d_2)\beta z + \beta^2 = 0.
\end{eqnarray}
On the other hand, the tangent vectors of $\vec{z}$ along $\Hfk_1$ and $\Hfk_x$, respectively, are
$$
\left(1, -\frac{1}{\nfk z^2}\right) \quad \text{ and }\quad
\left(1, \frac{-a}{\nfk(d_2z-\beta)^2}\right).
$$
If these two coincide, we get
$a z^2 = (d_2z-\beta)^2$, which says that
\begin{eqnarray}\label{eqn: T2}
(a-d_2^2)z^2 + 2d_2 \beta z - \beta^2 = 0.
\end{eqnarray}
As $z \in \Hfk$, comparing the coefficients of the equations~\eqref{eqn: T1} and \eqref{eqn: T2} we get
$$
(a-d_1d_2)\cdot 2d_2 \beta = (a-d_2^2) \cdot (d_1-d_2)\beta \quad \text{ and } \quad
(a-d_1d_2)\cdot (-\beta^2) = (a-d_2^2)\cdot \beta^2.
$$
Hence either $\beta = 0$ or
$$2ad_2 - 2d_1d_2^2 = ad_1 -ad_2 - d_1d_2^2 + d_2^3 \quad \text{ and } \quad
d_1d_2-a = a-d_2^2.$$
In both cases we can obtain $d_1 =d_2$ and $\beta = 0$, which says that $x \in A$ and $\Hfk_x = \Hfk_1$.
Therefore as $\Hfk_x \neq \Hfk_1$, the two tangent vectors much be different, i.e.\
the intersection of $\Hfk_x$ and $\Hfk_1$ at $\vec{z}$ must be transversal.
\end{proof}

Let $\widetilde{\Zcal}_x$ be the prime divisor associated with $\widetilde{X}_x$ on $\Scal_F(\mfk)$. We get
$$
\pi_*(\widetilde{\Zcal}_x) = [\Gamma_x : \Gamma_x(\mfk)\cdot \FF_q^\times] \cdot \Zcal_x.$$
%In particular:

%\begin{lem}
%$\widetilde{\Zcal}_1$ is a Cartier divisor of $\Scal_F(\mfk)$, and $\Zcal_1$ is a $\QQ$-Cartier divisor of $\Scal_{0,F}(\nfk)$.
%\end{lem}

%\begin{proof}

%\end{proof}

%\quad \text{and} \quad
%\pi^*(\Zcal_x) = \sum_{\gamma \in \Gamma/\Gamma_F(\mfk)\cdot \Gamma_x} \gamma \cdot \widetilde{\Zcal}_x.

%Here we view $\Zcal_x$ and $\widetilde{\Zcal}_x$ as (analytic) Weil divisors of $\Scal_{0,F}(\nfk)$ and $\Scal_{F}(\mfk)$, respectively.

From the above lemmas, the intersection number of $\Zcal_1$ and $\Zcal_x$ is determined in the following:

\begin{prop}\label{prop: projection formula}
Given $x \in \Lambda$ with $\det(x) \neq 0$, suppose $\Zcal_1 \neq \Zcal_x$. Choose $\mfk \in A_+$ so that $\nfk^2 \det(x) \mid \mfk$.
The intersection number of $\Zcal_1$ and $\Zcal_x$ is equal to
$$\Zcal_1\cdot \Zcal_x = \frac{q-1}{[\Gamma_1:\Gamma_1(\mfk)]\cdot[\Gamma_x:\Gamma_x(\mfk)]}
\cdot \sum_{\gamma \in \Gamma/\Gamma_F(\mfk)} \widetilde{\Zcal}_1 \cdot \gamma \widetilde{\Zcal}_x.
$$
\end{prop}

\begin{proof}
Observe that $\Zcal_1$ is a $\QQ$-Cartier divisor on $\Scal_{0,F}(\nfk)$.
Thus the result is a rigid-analytic version of the projection formula for the intersection of $\Zcal_1$ and $\Zcal_x = \pi_*(\widetilde{\Zcal}_x)$.
We include the argument here for completeness.

Let $\Gamma = \Gamma_{0,F}(\nfk)$.
Given $x \in \Lambda$ with $\det(x) \neq 0$,
the irreducible curve
$X_x$ has normalization isomorphic to $\widehat{\Gamma}_x \backslash \Hfk$.
Notice that
for each $\boz \in X_1 \cap X_x$, 
take $\tilde{\boz}_1 $ and $\tilde{\boz}_x $ be two lifts of $\boz$ in $\widetilde{X}_1$ and $\widetilde{X}_x$, respectively.
The intersection multiplicity of $X_1$ and $X_x$ at $\boz$ is actually equal to
$$
m_{\boz}(X_1,X_x)
= \frac{(q-1)\cdot \#\Stab_{\Gamma}(\tilde{\boz}_1)}{\#\Stab_{\widehat{\Gamma}_1}(\tilde{\boz}_1) \cdot \#\Stab_{\widehat{\Gamma}_x}(\tilde{\boz}_x)}.
$$

Now, consider the disjoint union
$$
\Phi:= \coprod_{\gamma \in \Gamma/\widehat{\Gamma}_x\cdot \Gamma_F(\mfk)} \widetilde{X}_1 \cap \gamma \widetilde{X}_x
$$
which maps surjectively to $X_1 \cap X_x$ via the finite morphism $\pi$ on each component (we denote this surjection $\Phi \twoheadrightarrow X_1\cap X_x$ by $\tilde{\pi}$).
For each point $\boz \in X_1 \cap X_x$, the pre-image of $\boz$ in $\Phi$ has cardinality equal to
$$[\widehat{\Gamma}_1 : \Stab_{\widehat{\Gamma}_1}(\tilde{\boz}_1)\cdot  \Gamma_1(\mfk)] \cdot 
\frac{\#\Stab_{\Gamma}(\tilde{\boz}_1)}{ \#\Stab_{\widehat{\Gamma}_x}(\tilde{\boz}_x)}.
$$
Thus the cardinality of $\Phi$ can be expressed as
\begin{eqnarray}
\sum_{\gamma \in \Gamma/\widehat{\Gamma}_x\cdot \Gamma_F(\mfk)} \#(\widetilde{X}_1 \cap \gamma \widetilde{X}_x)
&=&
\sum_{\boz \in X_1\cap X_x} \#\big(\tilde{\pi}^{-1}(\boz)\big) \nonumber \\
&=&
\sum_{\boz \in X_1 \cap X_x}
[\widehat{\Gamma}_1 : \Stab_{\widehat{\Gamma}_1}(\tilde{\boz}_1)\cdot  \Gamma_1(\mfk)] \cdot 
\frac{\#\Stab_{\Gamma}(\tilde{\boz}_1)}{ \#\Stab_{\widehat{\Gamma}_x}(\tilde{\boz}_x)} \nonumber \\
&=& [\widehat{\Gamma}_1 : \Gamma_1(\mfk)\FF_q^\times] \cdot 
\sum_{\boz \in X_1 \cap X_x}
\frac{(q-1)\cdot \#\Stab_{\Gamma}(\tilde{\boz}_1)}{\#\Stab_{\widehat{\Gamma}_1}(\tilde{\boz}_1) \cdot \#\Stab_{\widehat{\Gamma}_x}(\tilde{\boz}_x)} \nonumber \\
&=& [\widehat{\Gamma}_1 : \Gamma_1(\mfk)\FF_q^\times] \cdot 
\sum_{\boz \in X_1 \cap X_x} m_{\boz}(X_1,X_x).\nonumber
\end{eqnarray}
Therefore
\begin{eqnarray}
\Zcal_1\cdot \Zcal_x
&=& [\widehat{\Gamma}_1:\Gamma_1]\cdot [\widehat{\Gamma}_x:\Gamma_x] \cdot \sum_{\boz \in X_1\cap X_x}m_{\boz}(X_1,X_x) \nonumber \\
&=& \frac{[\widehat{\Gamma}_1:\Gamma_1]\cdot [\widehat{\Gamma}_x:\Gamma_x]}{[\widehat{\Gamma}_1 : \Gamma_1(\mfk)\FF_q^\times]} \cdot
\sum_{\gamma \in \Gamma/\widehat{\Gamma}_x\cdot \Gamma_F(\mfk)} \#(\widetilde{X}_1 \cap \gamma \widetilde{X}_x) \nonumber \\
&=& 
\frac{q-1}{[\Gamma_1:\Gamma_1(\mfk)]\cdot[\Gamma_x:\Gamma_x(\mfk)]}
\cdot \sum_{\gamma \in \Gamma/\Gamma_F(\mfk)} \widetilde{\Zcal}_1 \cdot \gamma \widetilde{\Zcal}_x, \nonumber
\end{eqnarray}
where the last equality holds as $\widetilde{X}_1$ and $\gamma \widetilde{X}_x (= \widetilde{X}_{\gamma \star x})$ intersect transversally for every $\gamma \in \Gamma$. 
%From the projection formula (cf.\ \cite[Chapter 2, Proposition 2.3 (c)]{Ful}), we have
%\begin{eqnarray}
%\Zcal_1\cdot \Zcal_x &=& \frac{1}{[\Gamma_1:\Gamma_1(\mfk)\FF_q^\times]} \cdot \left(\pi_*(\widetilde{\Zcal}_1) \cdot \Zcal_x\right) \nonumber \\
%&=& \frac{1}{[\Gamma_1:\Gamma_1(\mfk)\FF_q^\times]} \cdot \left(\widetilde{\Zcal}_1 \cdot \pi^*(\Zcal_x)\right) \nonumber \\
%&=& \frac{1}{[\Gamma_1:\Gamma_1(\mfk)\FF_q^\times]} \cdot
%\sum_{\gamma \in \Gamma/\Gamma_F(\mfk)\cdot \Gamma_x} \widetilde{\Zcal}_1 \cdot \gamma \widetilde{\Zcal}_x \nonumber \\
%&=& 
%\frac{q-1}{[\Gamma_1:\Gamma_1(\mfk)]\cdot[\Gamma_x:\Gamma_x(\mfk)]}
%\cdot \sum_{\gamma \in \Gamma/\Gamma_F(\mfk)} \widetilde{\Zcal}_1 \cdot \gamma \widetilde{\Zcal}_x. \nonumber 
%\end{eqnarray}
\end{proof}

%Notices that $\widetilde{\Zcal}_1$ and $\gamma \widetilde{\Zcal}_x$ intersect transversally for every $\gamma \in \Gamma$.
%Thus:

\begin{prop}\label{prop: Intersection}
Let $x \in \Lambda$ with $\det(x) \neq 0$, and $\mfk \in A_+$ so that $\nfk^2 \det(x)$ divides $\mfk$.
Given $\gamma \in \Gamma$, suppose $\widetilde{\Zcal}_1 \neq \gamma \widetilde{\Zcal}_x$. 
We have
$$
\widetilde{\Zcal}_1 \cdot \gamma \widetilde{\Zcal}_x
= \sum_{\gamma_0 \in \Gamma_1(\mfk) \backslash \Gamma_F(\mfk)/\Gamma_{\gamma \star x}(\mfk)} \#(\Hfk_1 \cap \gamma_0 \gamma \Hfk_x).
$$
\end{prop}

\begin{proof}
It suffices to show that the union
$$\bigcup_{\gamma_0 \in \Gamma_1(\mfk) \backslash \Gamma_F(\mfk)/\Gamma_{\gamma \star x}(\mfk)} \Hfk_1 \cap \gamma_0 \gamma \Hfk_x \ (\subset \Hfk_F)$$
is disjoint and in bijection with the intersection points of $\widetilde{X}_1$ and $\gamma \widetilde{X}_x$ under the canonical map $\Hfk_F \rightarrow \Scal_F(\mfk)$.

The surjectivity is straightforward.
On the other hand, given $\gamma_1, \gamma_2 \in \Gamma_F(\mfk)$ and $\vec{z}_i \in \Hfk_1 \cap \gamma_i\gamma \Hfk_x$ for $i = 1,2$,
write
$$\vec{z}_i = (z_i, S_1 z_i) = (\gamma_i\gamma w_i, \gamma_i'\gamma' S_x w_i) \quad \text{ with } z_i,w_i \in \Hfk \text{ for $i=1,2$.}
$$
Suppose the image of $\vec{z}_1$ and $\vec{z}_2$ coincides in $\Scal_F(\mfk)$, i.e.\ there exists $\gamma_0 \in \Gamma_F(\mfk)$ so that
$(z_1, S_1 z_1) = (\gamma_0 z_2, \gamma_0' S_1 z_2)$.
Then
$$
S_1 z_1 = \gamma_0' S_1 z_2 = \gamma_0' S_1 \gamma_0^{-1} z_1,
$$
which says
$(\gamma_0 \gamma_0^*) z_1 = z_1$.
From our choice of $\mfk$, we get $\gamma_0 \gamma_0^* \in \Gamma_F(\nfk\det(x))$ which fixes $z_1$. This implies $\gamma_0 \gamma_0^* = 1$.
Thus $\gamma_0 \in \Gamma_1 \cap \Gamma_F(\mfk) = \Gamma_1(\mfk)$.
Moreover, let 
$$\gamma_3 = \gamma^{-1}\gamma_1^{-1} \gamma_0 \gamma_2 \gamma \in \Gamma_F(\mfk) \quad \text{ (as $\Gamma_F(\mfk)$ is normal in $\Gamma$)}.$$
We get $(w_1,S_x w_1) = (\gamma_3 w_2, \gamma_3' S_x w_2)$, which says
$$(\gamma_3 \cdot \bar{x}^{-1} \gamma_3^*\bar{x}) w_1 = w_1.$$
Similarly, from our choice of $\mfk$ we get $\gamma_3 \cdot \bar{x}^{-1} \gamma_3^*\bar{x} \in \Gamma_F(\nfk)$ and fixes $w_1$.
Thus 
$$\gamma_3 \cdot \bar{x}^{-1} \gamma_3^*\bar{x} = 1, \quad \text{ which shows that $x \gamma_3^* = \bar{\gamma}_3 x$.}$$
Therefore $\gamma_3 \in \Gamma_x \cap \Gamma_F(\mfk) = \Gamma_x(\mfk)$.
In conclusion, we have
$$\gamma_1 \cdot (\gamma \gamma_3 \gamma^{-1}) = \gamma_0 \gamma_2,$$
i.e.\ $\gamma_1$ and $\gamma_2$ represents the same double cosets in
$\Gamma_1(\mfk)\backslash \Gamma_F(\mfk)/\Gamma_{\gamma \star x}(\mfk)$.
This assures the injectivity and completes the proof.
\end{proof}

\begin{lem}\label{lem: Intersection}
Given $x \in \Lambda$ with $\det(x) \neq 0$.
For $\gamma \in \Gamma$ with $\gamma \Hfk_x \neq \Hfk_1$ one has
$$
\Hfk_1 \cap \gamma \Hfk_x
= \{\vec{z} = (z,S_1 z)\mid (\gamma \star x) \cdot z = z\}.
$$
Consequently, 
put
$$
\tilde{\iota}(x):= \begin{cases} 1 & \text{ if $K_{x}/k$ is an imaginary quadratic field extension;}\\
0 & \text{ otherwise.}
\end{cases}
$$
Then
$$
\#(\Hfk_1\cap \gamma\Hfk_x) = 2 \cdot \tilde{\iota}(\gamma \star x).
$$
%For $\boz \in \Hfk_1 \cap \gamma \Hfk_x$, we have
%$$
%\text{\rm Stab}_{\Gamma_1}(\boz) = \Gamma_1 \cap  \Gamma_{\gamma \star x} = \text{\rm Stab}_{\Gamma_{\gamma \star x}}(\boz).
%$$
\end{lem}

\begin{proof}
Given $\vec{z} \in \Hfk_1 \cap \gamma \Hfk_x$, write $\vec{z} = (z,S_1 z) = (\gamma w, \gamma'S_x w)$ for $z,w \in \Hfk$.
We get
$$
\gamma'S_x \gamma^{-1} z = \gamma'S_x w = S_1 z.
$$
Thus
\begin{eqnarray}
z &=& \gamma S_x^{-1} (\gamma')^{-1} S_1 z \nonumber \\
&=& \gamma x (S_1^{-1} \bar{\gamma}' S_1) \cdot z \nonumber \\
&=& (\gamma \star x) \cdot z. \nonumber 
\end{eqnarray}
Conversely, given $z \in \Hfk$ so that $(\gamma \star x)\cdot z = z$,
we obtain
$$
\gamma'S_x \gamma^{-1} z = S_1 z.
$$
Let $w = \gamma^{-1} z$. Then
$$(z,S_1 z) = (\gamma w, \gamma' S_x w).$$
Hence $\vec{z} = (z,S_1 z) \in \Hfk_1 \cap \gamma \Hfk_x$.
This shows the first equality of (1).
Note that from the assumption that $\gamma \Hfk_x \neq \Hfk_1$, the element $\gamma \star x \notin k^\times$.
Observe that $(\gamma \star x)\cdot z = z$ if and only if the column vector $(z,1)^t$ is an eigen-vector of $\gamma \star x$.
This guarantees the second equality of (1).

%Now, given $\boz = (z, S_1 z) \in \Hfk_1 \cap \gamma \Hfk_x$, from (1) we have that $z$ is fixed by $\gamma \star x$.
%For each $\gamma_1 \in \text{Stab}_{\Gamma_1}(\boz)$, we get $\gamma_1 z = z$.
%Note that $k_\infty(z)/k_\infty$ is a quadratic field extension. Let $\tilde{z}$ be the Galois conjugate of $z$ in $k_\infty(z)$.
%We then have
%$$(\gamma \star x) \cdot \tilde{z} = \tilde{z} = \gamma_1 \tilde{z}.$$
%In other words, the elements $\gamma \star x$ and $\gamma_1$ share the same fix points $z$ and $\tilde{z}$, which says that 
%$$ \gamma_1 \cdot (\gamma \star x) = (\gamma \star x) \cdot \gamma_1.$$
%Since $\gamma_1 \in \Gamma_1$, one has
%$\gamma_1^* = \bar{\gamma}_1$.
%Therefore
%$$
%\gamma_1(\gamma \star x) \gamma_1^*
%= \gamma(\gamma \star x) \bar{\gamma}_1
%= \det (\gamma_1) x,
%$$
%which shows that $\gamma_1 \in \Gamma_1 \cap \Gamma_{\gamma\star x}$.
%Conversely, if $\gamma_2 \in \Gamma_1 \cap \Gamma_{\gamma \star x}$,
%then $\gamma_2$ commutes with $\gamma \star x$, which implies that $\gamma_2 z = z$.
%Thus $\gamma_2 \in \text{Stab}_{\Gamma_1}(\boz)$.
%Therefore $\text{Stab}_{\Gamma_1}(\boz) = \Gamma_1 \cap \Gamma_{\gamma \star x}$.
%Similarly, we can show that $\text{Stab}_{\Gamma_{\gamma \star x}}(\boz) = \Gamma_1 \cap \Gamma_{\gamma \star x}$,
%and so (2) holds.
\end{proof}

\begin{rem}\label{rem: CM-sup}
For non-zero $a \in A$, $x \in \Lambda_a$, and $\gamma \in \Gamma$ with $\gamma \Hfk_x \neq \Hfk_1$, let $t = \tr(\gamma \star x)$.
Then we get $2(\gamma \star x)^\natural \in O_{B_1}^o$ with
$$
 \big(2 (\gamma \star x)^{\natural}\big)^2
= \dfk(t^2-4a).
$$
Moreover, for $z \in \Hfk$ with $(\gamma \star x) \cdot z = z$, we must have
$\big(2(\gamma \star x)^\natural\big) \cdot z = z$.
Thus the above lemma tells in particular that the intersection of $\Zcal_1$ and $\Zcal_x$, when $\Zcal_1 \neq \Zcal_x$, are supported by the image of the CM points of $\Ccal_1$ with discriminant $\dfk(t^2-4a)$ for some $t \in A$ with $t^2-4a \prec 0$.
\end{rem}

Observe that for $\gamma \in \Gamma$, $\gamma_1 \in \Gamma_1$, and $\gamma_x \in \Gamma_x$, one has
$$\tilde{\iota}((\gamma_1\gamma \gamma_x) \star x) = \tilde{\iota}(\gamma \star x) = \#(\Gamma_1 \cap \Gamma_{\gamma \star x}) \cdot \iota(\gamma \star x).$$
We are now able to express the intersection number $\Zcal_1\cdot \Zcal_x$ as follows:

\begin{thm}\label{thm: CN-Int}
Given $x \in \Lambda$ with $\det x \neq 0$. Suppose $\Zcal_1 \neq \Zcal_x$, or equivalently, $\gamma \star x \notin k$ for every $\gamma \in \Gamma$.
Then
\begin{eqnarray}
\Zcal_1\cdot \Zcal_x
&=& 
2 \cdot \sum_{\gamma \in \Gamma_1\backslash \Gamma / \Gamma_x} \iota(\gamma \star x). \nonumber 
\end{eqnarray}
%\sum_{\gamma \in \Gamma_1\backslash \Gamma / \Gamma_x} \sum_{\boz \in \Hfk_1\cap \gamma \Hfk_x} \frac{q-1}{\#\text{\rm Stab}_{\Gamma_1}(\boz)}  \nonumber \\
%&=& 
\end{thm}

\begin{proof}
We have
\begin{eqnarray}
\Zcal_1 \cdot \Zcal_x &=& 
\frac{q-1}{[\Gamma_1:\Gamma_1(\mfk)]\cdot[\Gamma_x:\Gamma_x(\mfk)]}
\cdot \sum_{\gamma \in \Gamma/\Gamma_F(\mfk)} \widetilde{\Zcal}_1 \cdot \gamma \widetilde{\Zcal}_x \nonumber \\
&=& 
\frac{q-1}{[\Gamma_1:\Gamma_1(\mfk)]\cdot [\Gamma_x:\Gamma_x(\mfk)]}
\cdot \sum_{\gamma \in \Gamma_1(\mfk) \backslash \Gamma / \Gamma_x(\mfk)} 2 \cdot \tilde{\iota}(\gamma \star x) \nonumber \\
&=&
\sum_{\gamma \in \Gamma_1\backslash \Gamma / \Gamma_x} \frac{q-1}{\#(\Gamma_1 \cap \Gamma_{\gamma \star x})} \cdot 2 \cdot \tilde{\iota}(\gamma \star x) \nonumber \\
&=& 
2 \cdot \sum_{\gamma \in \Gamma_1\backslash \Gamma / \Gamma_x} \iota(\gamma \star x). \nonumber 
\end{eqnarray}
\end{proof}

We now define the self-intersection number of $\Zcal_1$ (following \cite[p.\ 84]{H-Z}).
First, put
$$\text{vol}(X_1)\ := \ \frac{2}{[\widehat{\Gamma}_1:\Gamma_1]} \cdot
H^{\dfk^+\nfk^+,\dfk^-\nfk^-}(0)
$$
and
$$
\text{vol}(\Zcal_1) \ := \ [\widehat{\Gamma}_1:\Gamma_1] \cdot \text{vol}(X_1) \ =\   2H^{\dfk^+\nfk^+,\dfk^-\nfk^-}(0).
$$
For each point $\boz \in X_1$,
take a lift $\tilde{\boz} \in \Hfk_1$, and let
$$r_{\boz} := \frac{\#\big(\Stab_\Gamma(\tilde{\boz})\big)}{\#\big(\Stab_{\widehat{\Gamma}_1}(\tilde{\boz})\big)}.
$$
We set the following ``Plücker-type'' number:
$$
\mu_{\boz}(X_1)
:= \frac{q-1}{\#(\Stab_{\Gamma}(\tilde{\boz}))} \cdot \big(r_{\boz}(r_{\boz}-1)\big).
$$
\begin{defn}\label{defn: SIN}
The self-intersection number of $\Zcal_1$ is then defined to be:
\begin{eqnarray}
\Zcal_1\cdot \Zcal_1 := [\widehat{\Gamma}_1:\Gamma_1]^2 \cdot\left( - \text{vol}(X_1) + \sum_{\boz \in X_1} \mu_{\boz}(X_1)\right). \nonumber
\end{eqnarray}
\end{defn}
\begin{lem}\label{lem: SIN}
We may express the self-intersection number of $\Zcal_1$ as follows:
$$
\Zcal_1\cdot \Zcal_1 = 2 \cdot \sum_{\gamma \in \Gamma_1 \backslash \Gamma/\Gamma_1} \iota(\gamma \star 1).
$$
\end{lem}

\begin{proof}
Given $\gamma \in \Gamma$, notice that $\gamma \star 1 \in k$ if and only if $\gamma \in \widehat{\Gamma}_1$.
As $\Gamma_1$ is normal in $\widehat{\Gamma}_1$, one has
\begin{eqnarray}
2\cdot \sum_{\gamma \in \Gamma_1\backslash \Gamma/\Gamma_1} \iota(\gamma \star 1)
&=& 2\cdot \sum_{\gamma_1 \in \widehat{\Gamma}_1/\Gamma_1} \iota(\gamma \star 1)
+ 2 \cdot \sum_{\subfrac{\gamma \in \Gamma_1\backslash \Gamma/\Gamma_1}{\gamma \notin \widehat{\Gamma}_1}} \iota(\gamma \star 1) \nonumber \\
&=& [\widehat{\Gamma}_1:\Gamma_1]\cdot \big(-2H^{\dfk^+\nfk^+,\dfk^-\nfk^-}(0)\big)
+ 2 \cdot \sum_{\subfrac{\gamma \in \Gamma_1\backslash \Gamma/\Gamma_1}{\gamma \notin \widehat{\Gamma}_1}} \iota(\gamma \star 1). \nonumber
\end{eqnarray}
Hence the result holds if we show
\begin{eqnarray}\label{eqn: angle part}
[\widehat{\Gamma}_1:\Gamma_1]^2 \cdot \sum_{\boz \in X_1}\mu_{\boz}(X_1) \ = \ 2 \cdot \sum_{\subfrac{\gamma \in \Gamma_1\backslash \Gamma/\Gamma_1}{\gamma \notin \widehat{\Gamma}_1}} \iota(\gamma \star 1).
\end{eqnarray}

Take $\mfk \in A_+$ with $\nfk^2 \mid \mfk$.
Adapting the proof of Propostion~\ref{prop: projection formula} we get
$$
[\widehat{\Gamma}_1:\Gamma_1]^2 \cdot \sum_{\boz \in X_1}\mu_{\boz}(X_1)
= \frac{q-1}{[\Gamma_1:\Gamma_1(\mfk)]^2}\cdot \sum_{\subfrac{\gamma \in \Gamma/\Gamma_F(\mfk)}{\gamma \notin \widehat{\Gamma}_1}} \widetilde{\Zcal}_1 \cdot \gamma \widetilde{\Zcal}_1.
$$
From Proposition~\ref{prop: Intersection} and Lemma~\ref{lem: Intersection}, we have
\begin{eqnarray}
\sum_{\subfrac{\gamma \in \Gamma/\Gamma_F(\mfk)}{\gamma \notin \widehat{\Gamma}_1}} \widetilde{\Zcal}_1 \cdot \gamma \widetilde{\Zcal}_1 
&=& \sum_{\subfrac{\gamma \in \Gamma_1(\mfk)\backslash \Gamma/\Gamma_1(\mfk)}{\gamma \notin \widehat{\Gamma}_1}} 2 \cdot  \tilde{\iota}(\gamma \star 1) \nonumber \\
&=&\frac{[\Gamma_1:\Gamma_1(\mfk)]^2}{q-1} \cdot 
\sum_{\subfrac{\gamma \in \Gamma_1\backslash \Gamma/\Gamma_1}{\gamma \notin \widehat{\Gamma}_1}} \frac{q-1}{\#(\Gamma_1\cap \Gamma_{\gamma \star 1})} \cdot 2 \cdot  \tilde{\iota}(\gamma \star 1) \nonumber \\
&=& \frac{[\Gamma_1:\Gamma_1(\mfk)]^2}{q-1} \cdot 2 \cdot 
\sum_{\subfrac{\gamma \in \Gamma_1\backslash \Gamma/\Gamma_1}{\gamma \notin \widehat{\Gamma}_1}} \iota(\gamma \star 1). \nonumber
\end{eqnarray}
Therefore the equality~\eqref{eqn: angle part} follows and the proof is complete.
\end{proof}

%When $\Zcal_1 = \Zcal_x$, i.e.\ $\gamma \star x \in k^\times$ for some $\gamma \in \Gamma$, we define
%$$\Zcal_1 \cdot \Zcal_x := 2 h^{\dfk\nfk^+,\nfk^-}(0) =: -\text{vol}(\Zcal_1).$$
For nonzero $a \in A$, consider the following \textit{Hirzebruch-Zagier-type divisor}
$$\Zcal(a) := \sum_{\Gamma \backslash \Lambda_a} \Zcal_x.$$
From Theorem~\ref{thm: AFCN}, Theorem~\ref{thm: CN-Int}, and Lemma~\ref{lem: SIN}, we finally arrive at:

\begin{cor}\label{cor: TINFC}
Given nonzero $a \in A$ and $y \in k_\infty^\times$ with $\deg a \leq 2 \ord_\infty(y)+2$, we have
$$\text{\rm vol}(O_{B_\AA}^\times/O_\AA)^{-1} \cdot I^*(a,y;\varphi_\Lambda) = \frac{|y|_\infty^2}{2} \cdot \sum_{x \in \Gamma \backslash \Lambda_a} \Zcal_1 \cdot \Zcal(a),
$$
and
$$
\text{\rm vol}(O_{B_\AA}^\times/O_\AA)^{-1} \cdot I^*(0,y;\varphi_\Lambda) = - \frac{|y|_\infty^2}{2} \cdot \text{\rm vol}(\Zcal_1).
$$
\end{cor}

\begin{rem}
From \textit{Remark}~\ref{rem: DTS-FC}, we may express the Fourier expansion of the Drinfeld-type automorphic form $\vartheta_\Lambda$ in terms of the corresponding intersection numbers: for $(x,y) \in k_\infty^\times \times k_\infty$,
$$
\vartheta_\Lambda\begin{pmatrix} y & x \\ 0&1\end{pmatrix}
=  -\text{vol}(\Zcal_1)\beta_{0,2}(y) + \sum_{0 \neq a \in A}\big(\Zcal_1\cdot \Zcal(a)\big) \cdot \big(\beta_{a,2}(y)\psi_\infty(ax)\big).
$$
\end{rem}

\appendix

\section{Local optimal embeddings}\label{sec: App-LOE}

Here we recall the needed properties of local optimal embeddings from a quadratic order into a hereditary order of a quaternion algebra over a local field. Further details are referred to \cite[Chapter 2, Section 3]{Vie} and \cite[Chapter 5, Section 1.1]{CLWY}. \\

Let $(L,|\cdot|_L)$ be a non-archimedean local field, and $O_L$ be the ring of integers in $L$.
Given a separable quadratic algebra $E$ over $L$ and a quaternion algebra $\Dcal$ over $L$ together with a fixed embedding $\iota: E\hookrightarrow \Dcal$, it is known that every embedding from $E$ into $\Dcal$ must be conjugates of $\iota$ by elements of $\Dcal^\times$.
Let $\Ocal$ be an $O_L$-order in $E$ and $O_\Dcal$ a maximal $O_L$-order in $\Dcal$.
Put
$$
\Ecal(\Ocal,O_\Dcal) := \{b \in \Dcal^\times \mid b^{-1}Eb \cap O_\Dcal = b^{-1} \Ocal b \}.
$$
Here we identify $E$ as a subalgebra of $\Dcal$ via $\iota$.
For $\alpha \in E$, $b \in \Ecal(\Ocal,O_\Dcal)$, and $\kappa \in O_\Dcal^\times$, one has
$$\alpha \cdot b \cdot  \kappa \in \Ecal(\Ocal,O_\Dcal).$$
Moreover, the following result holds
(cf.\ \cite[Chapter 2, Theorem 3.1 and 3.2]{Vie}):

\begin{lem}\label{lem: OES1}
\begin{itemize}
    \item[(1)] Let $O_E$ be the maximal $O_L$-order in $E$. Then
$$
e(O_E,O_\Dcal):= \#\left(E^\times \backslash \Ecal(O_E,O_\Dcal)/O_\Dcal^\times\right)
=
\begin{cases}
2, & \text{ if $\Dcal$ is division and $E/L$ is inert;}\\
0, & \text{ if $\Dcal$ is division and $E/L$ is split;}\\
1, & \text{ otherwise.}
\end{cases}
$$
\item[(2)] If $\Ocal \subsetneq  O_E$, then
$$
e(\Ocal,O_\Dcal):= \#\left(E^\times \backslash \Ecal(\Ocal,O_\Dcal)/O_\Dcal^\times\right)
= 
\begin{cases}
0, & \text{ if $\Dcal$ is division;} \\
1, & \text{ otherwise.}
\end{cases}
$$
\end{itemize}
\end{lem}

Suppose $\Dcal$ is not division (i.e.\ $\Dcal \cong \Mat_2(L)$). Let $O_\Dcal'$ be a hereditary $O_L$-order in $O_\Dcal$.
Put
$$\Ecal(\Ocal,O_\Dcal') := \{b \in \Dcal^\times \mid b^{-1}Eb \cap O_\Dcal' = b^{-1}\Ocal b\}.$$
Then for $\alpha \in E$, $b \in \Ecal(\Ocal,O_\Dcal')$, and $\kappa' \in (O_\Dcal')^\times$, one has
$$ \alpha \cdot b \cdot \kappa' \in \Ecal(\Ocal,O_\Dcal').$$
Moreover (cf.\ \cite[Chapter 2, Theorem 3.2]{Vie}):

\begin{lem}\label{lem: OES2}
\begin{itemize}
    \item[(1)]
$$
e(O_E,O_\Dcal') := \#(E^\times \backslash \Ecal(O_E,O_\Dcal')/(O_\Dcal')^\times) = 
\begin{cases}
0, & \text{ if $E/L$ is inert;}\\
1, & \text{ if $E/L$ is ramified;}\\
2, & \text{ if $E/L$ is split.}
\end{cases}
$$
\item[(2)]
If $\Ocal \subsetneq  O_E$, then
$$
e(\Ocal,O_\Dcal'):= \#(E^\times \backslash \Ecal(\Ocal,O_\Dcal')/(O_\Dcal')^\times) = 2.
$$
\end{itemize}
\end{lem}

\section{Special local integrals}\label{sec: App-SLI}

Given $c \in \ZZ_{\geq 0}$, put
$\Ocal(c) := O_L + \pi_L^c O_E$, where $\pi_L \in O_L$ is a uniformizer in $L$.
For $x \in E\backslash L$,
we can find a unique $c_x \in \ZZ_{\geq 0}$ if $x \in O_E$ so that $O_L[x] = \Ocal(c_x)$; and put $c_x := -1$ if $x \notin O_E$.
Let $\Dcal^o$ be the space of pure quaternions in $\Dcal$, i.e.\
$$\Dcal^o:= \{ b \in \Dcal \mid \tr(b) = 0\}.$$
Put $O_\Dcal^o := O_\Dcal\cap \Dcal^o$ and $O_{\Dcal}^{\prime, o} := O_{\Dcal}'\cap \Dcal^o$.
We observe that:

\begin{lem}
Given $x \in E \backslash L$ with $\tr(x) = 0$, one has
$$\mathbf{1}_{O_\Dcal^o}(b^{-1}xb) = \sum_{\ell = 0}^{c_x} \mathbf{1}_{\Ecal(\Ocal(\ell),O_\Dcal)}(b),
$$
Moreover, if $\Dcal$ is not division, then
$$
\mathbf{1}_{O_\Dcal^{\prime,o}}(b^{-1}xb) = \sum_{\ell = 0}^{c_x} \mathbf{1}_{\Ecal(\Ocal(\ell),O_\Dcal')}(b).$$
\end{lem}
\begin{proof}
Notice that $\mathcal{E}\left(\mathcal O(\ell),\mathcal O_{\mathcal{D}}\right)$ 
and 
$\mathcal{E}\left( \mathcal O(\ell'),\mathcal O_{\mathcal{D}}\right)$
are disjoint if $\ell\neq \ell'$. 
Thus for $b \in \Dcal^\times$ one has
$$\sum^{c_x}_{\ell=0}\mathbf{1}_{\mathcal{E}(\mathcal O(\ell),\mathcal O_{\mathcal{D}})}(b)=0\ \text{or}\ 1.
$$
Suppose the value is $1$, i.e. $b\in\mathcal{E}(\mathcal O(\ell_0),\mathcal O_{\mathcal{D}})$ for some $0\leq \ell_0\leq c_x$. 
Then 
$$x\in \mathcal O[x]=\mathcal O(c_x)\subset\mathcal{O}(\ell_0)\subset E\cap b\mathcal{O}_{\mathcal{D}}b^{-1}\subset b\mathcal{O}_{\mathcal{D}}b^{-1}.$$
Since $\tr(x)=0$, we get $b^{-1}xb\in\mathcal{O}^\circ_{\mathcal{D}}$, i.e. $\mathbf{1}_{\mathcal{O}^\circ_{\mathcal{D}}}(b^{-1}xb)=1$.

Conversely, let $b\in\mathcal{D}^\times$ with $\mathbf{1}_{\mathcal{O}^\circ_{\mathcal{D}}}(b^{-1}xb)=1$. Then $x\in b\mathcal{O}^\circ_{\mathcal{D}}b^{-1}$, which implies $\mathcal{O}(c_x)\subset E\cap b\mathcal{O}_{\mathcal{D}}b^{-1}$. Thus there exists $\ell_0$ with $0\leq \ell_0\leq c_x$ such that 
$$E\cap b\mathcal{O}_{\mathcal{D}}b^{-1}=\mathcal{O}(\ell_0).$$
which means that $b\in\mathcal{E}(\mathcal{O}(\ell_0),\mathcal{O}_{\mathcal{D}})$. Therefore $$\sum^{c_x}_{\ell=0}\mathbf{1}_{\mathcal{E}(\mathcal{O}(\ell),\mathcal{O}_{\mathcal{D}})}(b)= \mathbf{1}_{\mathcal{E}(\mathcal{O}(\ell_0),\mathcal{O}_{\mathcal{D}})}(b)=1.
$$
\end{proof}

Suppose Haar measures of $\Dcal^\times$ and $E^\times$ are chosen, respectively.
The above lemma leads to:

\begin{cor}\label{cor: OE1}
For $x \in E\backslash L$ with $\tr(x) = 0$, one has
$$
\int_{E^\times \backslash \Dcal^\times} \mathbf{1}_{O_\Dcal^o}(b^{-1}x b) d^\times b
\ =\ \frac{\text{\rm vol}(O_\Dcal^\times)}{\text{\rm vol}(O_E^\times)} \cdot \sum_{\ell = 0}^{c_x}\#\left(\frac{O_E^\times}{\Ocal(\ell)^\times}\right) \cdot e(\Ocal(\ell),O_\Dcal).
$$
Suppose $\Dcal$ is not division, then
$$
\int_{E^\times \backslash \Dcal^\times}
\mathbf{1}_{O_\Dcal^{\prime,o}}(b^{-1}xb) d^\times b
\ =\ \frac{\text{\rm vol}((O_\Dcal')^\times)}{\text{\rm vol}(O_E^\times)} \cdot \sum_{\ell = 0}^{c_x}\#\left(\frac{O_E^\times}{\Ocal(\ell)^\times}\right) \cdot e(\Ocal(\ell),O_\Dcal').
$$
\end{cor}
\begin{proof}
Given $0\leq \ell\leq c_x$, one has
\begin{align*}
\text{vol}(E^{\times}\backslash\mathcal{E}(\mathcal{O}(\ell),\mathcal{O}_{\mathcal{D}}))&=\sum_{b\in E^{\times}\backslash\mathcal{E}(\mathcal{O}(\ell),\mathcal{O}_D)/\mathcal{O}^\times_{\mathcal{D}}}\frac{\text{vol}(\mathcal{O}^\times_{\mathcal{D}})}{\text{vol}(E^\times\cap b\mathcal{O}^\times_{\mathcal{D}}b^{-1})}\\
&=\frac{\text{vol}(\mathcal{O}^\times_{\mathcal{D}})}{\text{vol}\left(\mathcal{O}^\times_E\right)}\cdot\#\left(\frac{\mathcal{O}^\times_E}{\mathcal{O}(\ell)^\times}\right)\cdot e\left(\mathcal{O}(\ell),\mathcal{O}_{\mathcal{D}}\right).
\end{align*}
Thus
\begin{align*}
\int_{E^\times\backslash\mathcal{D}^\times}\mathbf{1}_{\mathcal{O}^\circ_{\mathcal{D}}}(b^{-1}xb)d^\times b&=\sum^{c_x}_{\ell=0}\int_{E^\times\backslash\mathcal{D}^\times}\mathbf{1}_{\mathcal{E}(\mathcal{O}(\ell),\mathcal{O}_{\mathcal{D}})}(b)d^\times b\\
&=\frac{\text{vol}(\mathcal{O}^\times_{\mathcal{D}})}{\text{vol}(\mathcal{O}^\times_E)}\cdot\sum^{c_x}_{\ell=0}\#\left(\frac{\mathcal{O}^\times_E}{\mathcal{O}(\ell)^\times}\right)\cdot e\left(\mathcal{O}(\ell),\mathcal{O}_{\mathcal{D}}\right).
\end{align*}

\end{proof}

Let $q_L$ be the cardinality of the residue field of $L$.
Since 
$$\text{vol}((O_\Dcal')^\times) = \frac{1}{q_L+1}\cdot \text{vol}(O_\Dcal^\times),$$
combining Lemma~\ref{lem: OES1},  Lemma \ref{lem: OES2}, and Corollary~\ref{cor: OE1} we obtain:

\begin{cor}\label{cor: infinite part}\label{cor: OE2}
Suppose $\Dcal$ is not division. Then for $x \in O_E\backslash O_L$ with $\tr(x) = 0$, one has
\begin{eqnarray}
&&
\int_{E^\times\backslash\Dcal^\times}
\left(\mathbf{1}_{O_\Dcal^o}(b^{-1}x b)- \frac{q_L+1}{2}\cdot \mathbf{1}_{O_\Dcal^{\prime,o}}(b^{-1}x b)\right) d^\times b \nonumber \\
&=& 
\begin{cases}
\displaystyle \frac{1}{e(E/L)} \cdot \frac{\text{\rm vol}(O_\Dcal^\times)}{\text{\rm vol}(O_E^\times)}, &\text{ if $E$ is a field,} \\
0, & \text{otherwise.}
\end{cases}
\nonumber
\end{eqnarray}
Here $e(E/L)$ is the ramification index of $E/L$.
\end{cor}

\end{document}